\documentclass[11pt, leqno, twoside]{article}

\usepackage[english]{babel}

\usepackage{amsmath}
\usepackage{amssymb}
\usepackage{titletoc}
\usepackage{mathrsfs}
\usepackage{amsthm}
\usepackage{indentfirst}
\usepackage{color}
\usepackage{txfonts}
\usepackage{enumerate}

\usepackage[colorlinks=true,
linkcolor=blue,
citecolor=red,
urlcolor=magenta,
]{hyperref}

\usepackage{anysize}

\allowdisplaybreaks

\newtheorem{theorem}{Theorem}[section]
\newtheorem{lemma}[theorem]{Lemma}
\newtheorem{corollary}[theorem]{Corollary}
\newtheorem{proposition}[theorem]{Proposition}
\theoremstyle{definition}
\newtheorem{example}[theorem]{Example}
\newtheorem{remark}[theorem]{Remark}
\newtheorem{definition}[theorem]{Definition}

\numberwithin{equation}{section}

\begin{document}

\title{\bf\Large Local Matrix Muckenhoupt Weights and Quantitative Weighted
Inequalities Achieving Global Best Known
Exponents\footnotetext{\hspace{-0.25cm}2020 {\it Mathematics Subject Classification}.
Primary 47A30; Secondary 47A56, 46E30, 42B25, 42B20, 47B06.\endgraf
{\it Key words and phrases:}
local matrix weight,
local fractional maximal operator,
local fractional integral operator,
local Haar square function,
Calder\'{o}n--Zygmund operator with exponential decay.}}
\date{}
\author{Tuomas Hyt\"onen,
Dachun Yang\footnote{Corresponding author,
E-mail: \texttt{dcyang@bnu.edu.cn}/
{\color{red}\today}/Final Version.},\ \ Wen Yuan
and Mingdong Zhang}

\maketitle

\vspace{-0.8cm}

\begin{center}
\begin{minipage}{13cm}
{\small {\bf Abstract}\quad
In this article, we give various real-variable properties of
local Muckenhoupt matrix weights $W$
and establish the quantitative boundedness of several operators
on local matrix-weighted Lebesgue spaces $L^p(W)$, including
local (fractional) maximal operators,
local fractional integral operators, local Haar square functions,
and Calder\'{o}n--Zygmund operators with exponential decay.
For the local maximal operators,
we obtain the sharp quantitative bounds when $p\in(1,2]$,
while, for local fractional integral operators,
we obtain the quantitative bounds matching the global best known
exponents, whose scalar case is known to be sharp.
The key used strategies include giving a new extension
property (which can clarify their relationships with
global ones) and an optimal scale lifting property
(which can balance the locality of weights and operators under
consideration) of local matrix weights.
As an application, we establish
the quantitative boundedness on $L^p(W)$ of the Riesz transform
associated with Schr\"{o}dinger operators $-\Delta+m^2I$
with $m\in(0,\infty)$ being large enough,
whose quantitative bound when $p=2$ is precisely
$[W]_{\mathscr{A}^{\operatorname{loc}}_{2}(r)}^{\frac 32}$.
}
\end{minipage}
\end{center}

\vspace{0.2cm}

\tableofcontents

\vspace{0.2cm}

\section{Introduction}

Within harmonic analysis, the weighted theory is of
fundamental importance, and its impact extends
significantly to partial differential equations
and other branches of analysis.
Let $w$ be a non-negative locally integrable function on
$\mathbb{R}^n$. For any $p\in(1,\infty)$, the
\emph{scalar weighted Lebesgue space} is defined to be the
set of all measurable functions
$f$ on $\mathbb{R}^n$ such that
$$\|f\|_{L^p(w)}:=\left[\int_{\mathbb{R}^n}|f(x)|^p
w(x)\,dx\right]^{\frac{1}{p}}<\infty.$$
[Let $L^p:=L^p(w)$ with $w\equiv1$.]
A fundamental question in this theory
is to characterize the condition on $w$ such that
the Hardy--Littlewood maximal operator $M$ is bounded on $L^p(w)$.
In \cite{muc72},  Muckenhoupt answered this question by
proving that, for any $p\in(1,\infty)$, $M$ is bounded on
$L^p(w)$ if and only if $w$ satisfies
\begin{align}\label{eq-w-Ap}
	[w]_{A_{p}}:=\sup_{Q\subset\mathbb{R}^n}
	\frac{1}{|Q|}\int_Q w(x)\,dx\left\{\frac{1}{|Q|}
	\int_Q \left[w(x)\right]^{-\frac{p^{\prime}}{p}}
	\,dx\right\}^{\frac{p}{p^{\prime}}}<\infty,
\end{align}
where the supremum is taken over all cubes $Q\subset\mathbb{R}^n$
and $|Q|$ denotes the Lebesgue measure of $Q$.
The set of all non-negative locally integrable
functions on $\mathbb{R}^n$ satisfying \eqref{eq-w-Ap}
is nowadays referred to as the \emph{scalar Muckenhoupt weight class $A_p$}.
Subsequently, Hunt et al. \cite{hmw73} proved an
analogous result on $\mathbb{R}$ by replacing the
Hardy--Littlewood maximal operator with the
Hilbert transform.

Inspired by these results on classes $A_p$,
it is also natural to consider the weighted boundedness of
fractional operators.
For any $\alpha\in(0, n)$, the \emph{fractional integral operator}
$I_\alpha$ is defined by setting, for any suitable measurable function
$f$ on $\mathbb{R}^n$ and for any $x\in\mathbb{R}^n$,
\begin{align}\label{eq-fract}
I_\alpha f(x):=\int_{\mathbb{R}^n}
\frac{f(y)}{|x-y|^{n-\alpha}} \,dy.
\end{align}
Similarly, for any $\alpha\in[0, n)$, the \emph{fractional maximal operator}
$M_\alpha$ is defined by setting, for any measurable function
$f$ on $\mathbb{R}^n$ and for any $x\in\mathbb{R}^n$,
\begin{align*}
M_\alpha f(x):=\sup_{x\in Q\subset\mathbb{R}^n}
\frac{1}{|Q|^{1-\frac{\alpha}{n}}}
\int_Q|f(y)|\,dy,
\end{align*}
where the supremum is taken over all cubes $Q\subset\mathbb{R}^n$ that contain $x$.
Note that the operator $M_0$ coincides exactly with the
\emph{Hardy--Littlewood maximal operator $M$}.
Let $\alpha\in(0,n)$, $p\in(1,\frac{n}{\alpha})$,
and $q:=\frac{np}{n-\alpha p}$. To characterize the weighted boundedness of
fractional integrals $I_\alpha$, Muckenhoupt and
Wheeden \cite{mw74} introduced the \emph{weight class $A_{p,q}$}
and proved that $I_\alpha$ and $M_\alpha$ are
bounded from $L^p(w^\frac{p}{q})$
to $L^q(w)$ if and only if the non-negative weight $w$ satisfies
\begin{align*}
[w]_{A_{p,q}}:=
\sup_{Q\subset\mathbb{R}^n}
\fint_Q w(x)\,dx\left\{\fint_Q \left[w(x)\right]^{-\frac{p^{\prime}}{q}}
\,dx\right\}^{\frac{q}{p^{\prime}}}<\infty,
\end{align*}
where the supremum is taken over all cubes $Q\subset\mathbb{R}^n$.

Later, there has been significant interest in tracking the
precise dependence of weighted inequalities on the weight constants.
Initially, Buckley \cite[Theorem 2.5]{buc93} proved that, for any
$p\in(1,\infty)$ and $w\in A_p$,
\begin{align*}
\left\|M\right\|_{L^p(w)\to L^p(w)}\lesssim[w]^{\frac{1}{p-1}}_{A_p},
\end{align*}
where the implicit positive constant is independent of $w$.
Later, Lacey et al. \cite{lmpt10} proved, for any
$\alpha\in(0,n)$, $p\in(1,\frac{n}{\alpha})$,
$q:=\frac{np}{n-\alpha p}$, and $w\in A_{p,q}$,
\begin{align}
\left\|M_\alpha\right\|_{L^p(w^\frac{p}{q})\to L^q(w)}
&\lesssim [w]^{(1-\frac{\alpha}{n})\frac{p'}{q}}_{A_{p,q}} \label{eq-quantativebound-fracmax}\\
\intertext{and}
\left\|I_\alpha\right\|_{L^p(w^\frac{p}{q})\to L^q(w)} &\lesssim [w]^{(1-\frac{\alpha}{n})\max(1,\frac{p'}{q})}_{A_{p,q}},
\label{eq-quantativebound-fracint}
\end{align}
where both implicit positive constants are independent of
$w$ and both exponents are sharp.  Moreover, the study of sharp
quantitative weighted inequalities for Calder\'on--Zygmund operators,
particularly the well-known $A_2$ conjecture,
is also an extensive and fascinating branch of harmonic analysis.
We refer the reader to \cite{cru23} for an excellent introduction
to this topic.

In the study of the prediction theory of
multivariate stochastic processes,
the matrix-weighted Lebesgue space $L^2(W)$ emerged
in the work of Wiener and Masani
\cite[Section 4]{wm58}. As a natural higher-dimensional
generalization of $L^p(w)$, the matrix-weighted
Lebesgue spaces $L^p(W)$ have attracted
much attention. Let $W$ be an $m\times m$ matrix-valued function
on $\mathbb{R}^n$ such that
$W$ is positive definite almost everywhere and all the
entries of $W$ are locally integrable functions on $\mathbb{R}^n$,
where $m\in\mathbb{N}$. For any $p\in(1,\infty)$,
the \emph{matrix-weighted Lebesgue space} $L^p(W)$ is
defined to be the set of all measurable vector-valued
functions $\vec{f}: \mathbb{R}^n\to \mathbb{C}^m$ such that
\begin{align*}
\left\|\vec{f}\right\|_{L^{p}(W)}:=\left[\int_{\mathbb{R}^n}
\left|W^{\frac{1}{p}}(x)\vec{f}(x)\right|^{p}\,dx
\right]^{\frac{1}{p}}<\infty.
\end{align*}
[Let $L^p:=L^p(W)$ with $W\equiv I_m$.]
To investigate multivariate random stationary
processes and the invertibility of Toeplitz operators,
Treil and Volberg \cite{tv97} found the
matrix $\mathscr{A}_2$ condition and showed that
the boundedness of the Hilbert transform on $L^2(W)$ is equivalent to
$W\in\mathscr{A}_2$. Later, Nazarov and Treil \cite{nt96}
and Volberg \cite{v97} independently determined the matrix
$\mathscr{A}_p$ condition for any $p\in(1, \infty)$
and proved that the boundedness of the Hilbert transform on
$L^{p}(W)$ is equivalent to $W\in\mathscr{A}_p$.
Since then, the study of weighted inequalities
has developed extensively. For example,
for any $p\in(1,\infty)$ and $W\in\mathscr{A}_p$,
the boundedness of the Hardy--Littlewood maximal operator
and singular integral operators on $L^p(W)$ was
established by Christ and Goldberg \cite{cg01,g03}.
Isralowitz and Moen \cite{im19} introduced the
matrix weight class $\mathscr{A}_{p,q}$
for any $1<p\leq q<\infty$ and showed
the quantitative boundedness of fractional maximal
operators and fractional integral operators
in the matrix-weighted setting (see Theorems \ref{thm-bound-max}
and \ref{thm-bound-fracint}).
Note that the matrix weight class $\mathscr{A}_{p,p}$ is
exactly the matrix weight class $\mathscr{A}_{p}$.
Recently, Bownik and Cruz-Uribe \cite{bc22}
established the Jones factorization theorem and the Rubio de Francia
extrapolation theorem for $\mathscr{A}_p$. Moreover,
we also refer to \cite{cimpr21,llor24a,llor24b}
for the study of matrix-weighted weak type inequalities
and to \cite{ipt22} for the results on commutators
in two matrix-weighted setting.

In recent years, the quantitative boundedness of
Calder\'on--Zygmund operators on $L^{p}(W)$ has
attracted widespread attention.
A major breakthrough in this area occurred in 2017 when
Nazarov et al. \cite{nptv} established the $L^2(W)$-norm
inequality for Calder\'on--Zygmund operators,
obtaining an upper bound of $C[W]^{\frac{3}{2}}_{\mathscr{A}_2}$
via convex body domination. Surprisingly, Domelevo et al. \cite{dptv24}
recently proved that the exponent $\frac{3}{2}$ is indeed sharp for the Hilbert transform.

On the other hand, Lemari\'{e}-Rieusset \cite{le94}
first introduced the local Muckenhoupt weight class $A^{\operatorname{loc}}_p$
for any $p\in(1,\infty)$ and studied
unconditional wavelet bases for $L^p(w)$
with $w\in A^{\operatorname{loc}}_p$.
Compared to classical $A_p$ weights, such weights are
defined as in \eqref{eq-w-Ap}
with the supremum taken over only all cubes with edge length
not greater than 1. Later, Rychkov \cite{ry01} developed the
Littlewood--Paley theory of inhomogeneous weighted Besov spaces
$B^s_{p,q}(w)$ and Triebel--Lizorkin spaces
$F^s_{p,q}(w)$ with $w\in A^{\operatorname{loc}}_p$.
Due to the restriction on the scale, such classes
contain the function $e^{|\cdot|}$ as a typical example
(see Proposition \ref{prop-exp-locAp}).
Recently, the theory of local weights was
developed in various directions.
We refer to \cite{inns20,ns21} for the study of variable local weights,
\cite{ooi24,ooi25,ooi26} for the study of capacity theory associated
with $A^{\operatorname{loc}}_p$ weights, and
\cite{csyyy25,is09} for the study of function spaces associated with
$A^{\operatorname{loc}}_p$ weights.

In contrast to the well-developed theory of local scalar weights,
the study of local matrix weights and their
quantitative inequalities remains largely unexplored.
In this article, we give various real-variable properties of
local Muckenhoupt matrix weights $W$
and establish the quantitative boundedness of several operators
on $L^p(W)$. The main challenge arises from the non-separability
between matrix weights and vector-valued functions
under consideration combined with the scale restriction.
To overcome these difficulties, we first give an extension
property of the local matrix weight class
$\mathscr{A}^{\operatorname{loc}}_{p,q}(r)$
(see Theorem \ref{thm-extension}),
which provides a way to relate local matrix weights to
global ones while maintaining control over the weight constants
and hence clarifies their relationships with
global matrix weights.
Next, we establish a sharp quantitative scale lifting
theorem for $\mathscr{A}^{\operatorname{loc}}_{p,q}(r)$
to balance the locality of weights and operators under consideration
(see Theorem \ref{thm:Apq(1+eta)} and Example \ref{exam-sharpness-power-2}
for the sharpness). Via these two tools,
we obtain quantitative weighted bounds on
local matrix-weighted Lebesgue spaces for several classes of operators,
including local (fractional) maximal operators,
local fractional integral operators,
local Haar square functions, and
Calder\'{o}n--Zygmund operators with exponential decay
(see Theorems \ref{thm-bound-maxloc}, \ref{thm-bound-fracintloc},
\ref{thm-HAAR-Aploc}, \ref{thm-Haar-Aploc}, and \ref{thm-CZloc}).
For local fractional maximal operators,
we obtain the sharp quantitative bounds when $q\leq p'$
(see Remark \ref{rmk-max}),
while, for local fractional integral operators,
we obtain the quantitative bounds matching the global best known
exponents, whose scalar case is known to be sharp
(see Theorem \ref{thm-bound-fracintloc} and Corollary \ref{cor-bound-fracintloc-scalar}).
In the one-dimensional scalar-valued setting, we also derive
the sharp quantitative weighted boundedness of the local maximal operator
for any $p\in(1,\infty)$ (see Section \ref{s4}). As an application, we establish
the quantitative boundedness on $L^p(W)$ of the Riesz transform
associated with Schr\"{o}dinger operators $-\Delta+m^2I$
with $m\in(0,\infty)$ being large enough,
whose quantitative bound when $p=2$ is precisely
$[W]_{\mathscr{A}^{\operatorname{loc}}_{2}(r)}^{\frac 32}$.

The remainder of this article is organized as follows.

In Section \ref{s1}, we first recall the concept of
local $A^{\operatorname{loc}}_p(r)$ weights and
identify the precise range of $c, \alpha\in\mathbb{R}$ such
that $e^{c|\cdot|^{\alpha}}\in A^{\operatorname{loc}}_{p}(r)$
(see Definition \ref{def-w-locAp} and Proposition \ref{prop-exp-locAp}).
Next, we introduce local matrix $\mathscr{A}^{\operatorname{loc}}_{p,q}(r)$
weights and establish their equivalent definition in terms of reducing operators
(see Definition \ref{def-locAp} and Lemma \ref{lem-W-dualW-dkps}).
Finally, we show that the local matrix weight constant is at least 1
(see Corollary \ref{coro-[W]geq1}).

In Section \ref{s2}, we establish an extension theorem for the
local matrix weight class $\mathscr{A}^{\operatorname{loc}}_{p,q}(r)$
(see Theorem \ref{thm-extension}), which clarifies its relationship
with the global one. Moreover, via a quantitative estimate of
local matrix weight constants,
we show that the class $\mathscr{A}^{\operatorname{loc}}_{p,q}(r)$
has the sharp scale lifting
property (see Theorem \ref{thm:Apq(1+eta)}),
which further indicates that this class is independent of the
scale $r\in(0,\infty)$. Moreover, in Example \ref{exam-sharpness-power-2},
we also show that the estimate in
Theorem \ref{thm:Apq(1+eta)} is sharp.

In Section \ref{s3}, we establish the
quantitative boundedness of several operators
on local matrix-weighted Lebesgue spaces.
To be precise, Subsections \ref{s3-1} and \ref{s3-2}
provide the boundedness of fractional maximal
and fractional integral operators (see Theorems \ref{thm-bound-maxloc}
and \ref{thm-bound-fracintloc}).
We also characterize the local matrix weights via
the boundedness of local maximal operators (see Corollary \ref{cor-Apq-Mloc}).
In Subsection \ref{s3-3}, for a fixed scale $j\in\mathbb{Z}$,
we use two types of local Haar square functions
$S_{W,p,j}$ and $\widetilde{S}_{W,p,j}$ and
the dyadic averaging operators $E_j$ to characterize
the norm of $L^p(W)$, where $p\in(1,\infty)$ and
$W\in \mathscr{A}^{\operatorname{loc}}_{p}(2^{-j})$
(see Theorems \ref{thm-HAAR-Aploc} and \ref{thm-Haar-Aploc}).
Subsection \ref{s3-4} turns to Calder\'{o}n--Zygmund operators.
To match the exponential doubling property of local matrix
$\mathscr{A}^{\operatorname{loc}}_{p}(r)$ weights,
motivated by \cite{dlt25}, we introduce the class of Calder\'{o}n--Zygmund
operators with exponential decay
and show their boundedness on local matrix-weighted Lebesgue spaces
(see Theorem \ref{thm-CZloc}).  As an application,
we show that the Riesz transform associated with
Schr\"{o}dinger operators $-\Delta+m^2I$
is exactly a Calder\'{o}n--Zygmund operator with exponential decay
and hence establish the quantitative boundedness
on $L^p(W)$ of this operator with $m\in(0,\infty)$ being large enough
(see Example \ref{exam-Riesz-bessel}).

Finally, Section \ref{s4} presents the sharp weighted
bound on local weighted Lebesgue spaces for local maximal
operators for the full range
in the one-dimensional and scalar-valued setting
[see Remark \ref{rmk-max}(ii) and Theorem \ref{thm:strong}].

We conclude this introduction with a few notational conventions.
Throughout this article, we work in $\mathbb{R}^n$ and, unless
otherwise specified, take $\mathbb{R}^n$ as the underlying space.
Let $\mathbb{Z}$ be the set of all integers, $\mathbb{N}:=\{1,2,\dots\}$,
and $\mathbb{Z}_+:=\mathbb{N}\cup\{0\}$.
All the cubes $Q\subset\mathbb{R}^n$ in this article
are always assumed to have their edges parallel to the coordinate
axes. For any cube $Q\subset\mathbb{R}^n$, let $\ell(Q)$
be its \emph{edge length} and, for any $r\in(0,\infty)$,
let $rQ$ be the cube with the
same center as $Q$ and the edge length $r\ell(Q)$.
Let $\mathcal{D}:=\{Q_{j, k}\}_{j\in\mathbb{Z}, k\in{\mathbb{Z}}^n}
:=\{2^{-j}([0, 1)^n+k)\}_{j\in\mathbb{Z}, k\in{\mathbb{Z}}^n}$
be the set of all \emph{dyadic cubes} in $\mathbb{R}^n$.
Let $\mathbf{0}$ denote the \emph{origin}
of $\mathbb{R}^n$ or $\mathbb{C}^m$.
For any set $E\subset\mathbb{R}^n$,
let $\mathbf 1_E$ be its \emph{characteristic function}.
For any measurable set $E\subset\mathbb{R}^n$
with $|E|\in(0,\infty)$ and for any
positive measurable or integrable function $f$ on $E$, let
$$\fint_E f(x)\,dx:=\frac{1}{|E|}\int_{E}f(x)\,dx.$$
For any $p\in(1,\infty)$,
let $p':=\frac{p}{p-1}$ be its \emph{conjugate index}
(that is, $\frac{1}{p}+\frac{1}{p'}=1$) and,
for any $a\in\mathbb{R}$, let $\lceil a\rceil$ be the
smallest integer not less than $a$.
For any $x\in\mathbb{R}^n$ and $r\in(0,\infty)$, let
\begin{align*}
B(x,r):=\left\{y\in\mathbb{R}^n:\ \left|x-y\right|<r\right\}.
\end{align*}
For any $f\in L^1$, its \emph{Fourier transform}
$\widehat{f}$ is defined by setting,
for any $\xi\in\mathbb{R}^n$,
$$\widehat{f}(\xi):=\int_{\mathbb{R}^n}f(x)
e^{-i x\cdot \xi}\,dx,$$
where $i=\sqrt{-1}$.
The symbol $C$ denotes a positive constant which is independent
of the main parameters involved, but may vary from line to line.
The notation $A\lesssim B$ means that $A\leq CB$ for some positive constant $C$,
while $A\sim B$ means $A\lesssim B\lesssim A$.
Finally, in all subsequent proofs we shall retain
the notation introduced in the original theorem
(or the relevant statement).

\section{Local Matrix Weights}\label{s1}

In this section, we first recall the definition of
local weight classes $A^{\operatorname{loc}}_p(r)$
and present typical examples in these classes in terms of
exponential functions. Next, we introduce the local matrix
weight classes $\mathscr{A}^{\operatorname{loc}}_{p,q}(r)$
and establish their equivalent characterizations
in terms of reducing operators.
Finally, we show that the local matrix weight constant is at least 1.

We call a non-negative locally integrable
function on $\mathbb{R}^n$ a \emph{scalar weight} if it
takes values in $(0, \infty)$ almost everywhere
(see, for instance, \cite[p.\,499]{g14c}).
We also recall the concept of local weight classes
$A_{p}^{\mathrm{loc}}(r)$ introduced in
\cite[(1.1)]{ry01}.

\begin{definition}\label{def-w-locAp}
Let $r\in(0,\infty)$ and $p\in(1,\infty)$.
The \emph{local scalar Muckenhoupt class $A_{p}^{\mathrm{loc}}(r)$}
is defined to be the set of all scalar weights $w$ such that
\begin{align}\label{eq-w-locAp}
[w]_{A_{p}^{\mathrm{loc}}(r)}:=
\sup_{\{\operatorname{cube} Q:\,\ell(Q)\leq r\}}
\fint_Q w(x)\,dx\left\{\fint_Q \left[w(x)\right]^{-\frac{p^{\prime}}{p}}
\,dx\right\}^{\frac{p}{p^{\prime}}}<\infty.
\end{align}
\end{definition}

\begin{remark}
In \eqref{eq-w-locAp}, if the supremum is taken over
all cubes $Q\subset \mathbb{R}^n$,
then we obtain the well-known scalar Muckenhoupt class $A_p$
[see \eqref{eq-w-Ap}].
Obviously, for any $r\in(0,\infty)$ and $p\in(1,\infty)$,
$A_p \subset A_{p}^{\operatorname{loc}}(r)$. Moreover,
this inclusion is also proper. Indeed, by verifying the doubling
condition of the class $A_p$
(see, for instance, \cite[Proposition 7.1.5(9)]{g14c}),
it is easy to find that $e^{|\cdot|}\notin A_{p}$.
To see $e^{|\cdot|} \in A_{p}^{\operatorname{loc}}(r)$,
we give the following general proposition,
which determines when an exponential function
is a local Muckenhoupt weight.
\end{remark}

\begin{proposition}\label{prop-exp-locAp}
Let $r\in(0,\infty)$, $p\in(1,\infty)$, and $c,\alpha\in\mathbb{R}$.
Then $e^{c|\cdot|^{\alpha}}\in A^{\operatorname{loc}}_{p}(r)$
if and only if $c=0$ or $\alpha\in[0,1]$.
\end{proposition}

\begin{proof}
If $c=0$ or $\alpha=0$, then $e^{c|\cdot|^{\alpha}}$ is obviously
a constant function. Thus, $e^{c|\cdot|^{\alpha}}\in
A^{\operatorname{loc}}_{p}(r)$.
In what follows, we assume that $c\neq0$ and $\alpha\neq0$ and
consider three cases for the exponent $\alpha$.

\emph{Case (1)} $\alpha\in(-\infty, 0)$.
In this case, if $c\in(0,\infty)$, then, for any cube $Q\subset\mathbb{R}^n$
containing the origin $\mathbf{0}$,
$$
\int_{Q}e^{c|x|^{\alpha}}\,dx=\infty,
$$
which implies $[e^{c|\cdot|^{\alpha}}]_{A^{\operatorname{loc}}_{p}(r)}=\infty$.
On the other hand, if $c\in(-\infty,0)$,
then $-c\frac{p^{\prime}}{p} > 0$.
Thus, for any cube $Q\subset\mathbb{R}^n$
containing the origin $\mathbf{0}$,
$$
\int_{Q}e^{-c|x|^{\alpha}\frac{p^{\prime}}{p}}\,dx=\infty,
$$
and hence $[e^{c|\cdot|^{\alpha}}]_{A^{\operatorname{loc}}_{p}(r)}=\infty$.
This completes the proof of this case.

\emph{Case (2)} $\alpha\in(1, \infty)$.
In this case, we first consider the case $c\in(0,\infty)$.
Let $k\in\mathbb{N}$ and $Q_k:=[k, k+r]^n$
be a cube in $\mathbb{R}^n$. We next construct the
following two sub-cubes
$$Q^{(1)}_k:=\left[k, k+\frac{1}{3}r\right]^n \quad \text{and}
\quad Q^{(2)}_k:=\left[k+\frac{2}{3}r, k+r\right]^n.$$
Observe that $|Q^{(1)}_k|=|Q^{(2)}_k|=3^{-n}|Q_k|$.
Using \eqref{eq-w-locAp} and restricting the integration
domains to these two sub-cubes, we find that
\begin{align}\label{eq-exp-locAp}
[e^{c|\cdot|^{\alpha}}]_{A^{\operatorname{loc}}_{p}(r)}
&\geq\fint_{Q_k} e^{c|x|^{\alpha}}\,dx
\left(\fint_{Q_k} e^{-c|x|^{\alpha}\frac{p^{\prime}}{p}}
\,dx\right)^{\frac{p}{p^{\prime}}}\\
&\geq \frac{|Q^{(2)}_k|}{|Q_k|} \inf_{x \in Q^{(2)}_k}
e^{c|x|^{\alpha}}
\left(\frac{|Q^{(1)}_k|}{|Q_k|} \inf_{y \in Q^{(1)}_k} e^{-c|y|^{\alpha}\frac{p^{\prime}}{p}}
\right)^{\frac{p}{p^{\prime}}}\nonumber\\
&\gtrsim e^{c[\sqrt{n}(k+\frac{2}{3}r)]^\alpha-
c[\sqrt{n}(k+\frac{1}{3}r)]^\alpha}
=e^{cn^{\frac{\alpha}{2}}[(k+\frac{2}{3}r)^\alpha
-(k+\frac{1}{3}r)^\alpha]}.\nonumber
\end{align}
By the mean value theorem, there exists
$\xi \in (k+\frac{1}{3}r, k+\frac{2}{3}r)$ such that
\begin{align*}
\left(k+\frac{2}{3}r\right)^\alpha-\left(k+\frac{1}{3}r\right)^\alpha
=\alpha\xi^{\alpha-1} \frac{1}{3}r.
\end{align*}
Note that $\xi^{\alpha-1} \to \infty$ as $k\to\infty$
because $\alpha > 1$.
Letting $k\to\infty$ in \eqref{eq-exp-locAp},
we obtain $[e^{c|\cdot|^{\alpha}}]_{A^{\operatorname{loc}}_{p}(r)}=\infty$.

Next, if $c\in(-\infty,0)$, we can apply the same argument
as in \eqref{eq-exp-locAp} by simply swapping the
roles of $Q^{(1)}_k$ and $Q^{(2)}_k$ and then
conclude again that
$[e^{c|\cdot|^{\alpha}}]_{A^{\operatorname{loc}}_{p}(r)}=\infty$.
This completes the proof of this case.

\emph{Case (3)} $\alpha\in(0,1]$.
We first recall the following basic inequality:
for any $x,y\in\mathbb{R}^n$,
$$
\left||x|^{\alpha}-|y|^{\alpha}\right|\leq|x-y|^{\alpha}.
$$
By this inequality, we find that,
for any cube $Q\subset\mathbb{R}^n$ with
$\ell(Q)\leq r$ and for any $x,y\in Q$,
$|x-y| \leq \sqrt{n}r$ and
\begin{align}\label{eq-ex-ey}
\frac{e^{c|x|^{\alpha}}}{e^{c|y|^{\alpha}}}=
e^{c(|x|^{\alpha}-|y|^{\alpha})}\leq
e^{|c|||x|^{\alpha}-|y|^{\alpha}|}\leq
e^{|c||x-y|^{\alpha}}
\leq e^{|c|n^{\frac{\alpha}{2}}r^\alpha}.
\end{align}
It immediately follows from \eqref{eq-ex-ey} that
$$
\sup_{x\in Q}e^{c|x|^{\alpha}}
\leq e^{|c|n^{\frac{\alpha}{2}}r^\alpha}
\inf_{y\in Q}e^{c|y|^{\alpha}}.
$$
This, together with \eqref{eq-w-locAp}, further implies that,
for any cube $Q\subset\mathbb{R}^n$ with
$\ell(Q)\leq r$
\begin{align*}
\fint_Q e^{c|x|^{\alpha}}\,dx \left(\fint_Q
e^{-c|y|^{\alpha}\frac{p^{\prime}}{p}}\,dy
\right)^{\frac{p}{p^{\prime}}}
&\le\sup_{x\in Q}e^{c|x|^{\alpha}}
\left(\sup_{y \in Q} e^{-c|y|^{\alpha}
\frac{p^{\prime}}{p}}\right)^{\frac{p}{p^{\prime}}}\\
&=\sup_{x\in Q}e^{c|x|^{\alpha}}
\left( \inf_{y \in Q} e^{c|y|^{\alpha}} \right)^{-1}\\
&\leq e^{|c|n^{\frac{\alpha}{2}}r^\alpha}.
\end{align*}
Taking the supremum of its left-hand side over
all cubes $Q\subset\mathbb{R}^n$ with
$\ell(Q)\leq r$, we find that
$[e^{c|\cdot|^{\alpha}}]_{A_{p}^{\mathrm{loc}}(r)}<\infty$,
and hence $e^{c|\cdot|^{\alpha}}$ is a local Muckenhoupt $A_p$ weight.
This completes the proof of this case and hence
Proposition \ref{prop-exp-locAp}.
\end{proof}

Motivated by Definition \ref{def-w-locAp}, we next introduce the
local matrix Muckenhoupt weight classes. To this end,
we first recall some basic notation.
In what follows, we always use $m\in\mathbb{N}$ to
denote the dimension of vectors and the order of square matrices.
Let $M_m(\mathbb{C})$ be the set of all $m\times m$ complex matrices,
and let $I_m$ denote the identity matrix of order $m$.
For any $A\in M_m(\mathbb{C})$, its conjugate transpose is denoted by $A^*$
and its \emph{operator norm} $\|A\|$ is defined by setting
$$\|A\|:=\sup_{\vec z\in\mathbb{C}^m, |\vec z|=1}|A\vec{z}|.$$
Recall that $A\in M_m(\mathbb{C})$ is a \emph{unitary matrix} if $A^*A=I_m$
and is \emph{positive definite} (resp. \emph{positive semidefinite}) if
$\vec z^*A\vec z>0$ for any $\vec z\in\mathbb{C}^m\setminus\{\mathbf{0}\}$
(resp. $\vec{z}^*A\vec{z}\geq0$ for any $\vec z\in\mathbb{C}^m$)
(see, for instance, \cite[(7.1.1a) and (7.1.1b)]{hj13}).

By \cite[Theorems 2.5.6 and 7.2.1]{hj13},
for any given positive definite matrix $A\in M_m(\mathbb{C})$,
there exists a unitary matrix $U\in M_m(\mathbb{C})$ such that
\begin{align*}
A=U\operatorname{diag}\left(\lambda_1,\ldots,\lambda_m\right)U^{*},
\end{align*}
where $\{\lambda_i\}_{i=1}^m \subset (0,\infty)$ are the eigenvalues of $A$.
For any $\alpha\in\mathbb{R}$, let
$$A^{\alpha}:=U\operatorname{diag}
(\lambda^{\alpha}_1,\ldots,\lambda^{\alpha}_m)U^{*}.$$
It follows from \cite[p.\,408]{hj94} that $A^\alpha$
is independent of $U$ and hence well defined.

Let $D_m(\mathbb{C})$ denote the set of all $m\times m$
positive semidefinite matrices.
A matrix-valued function $W:\ \mathbb{R}^n\to D_m(\mathbb{C})$ is called
a \emph{matrix weight} if its entries are locally
integrable functions on $\mathbb{R}^n$
and $W(x)$ is positive definite for almost every $x\in\mathbb{R}^n$.
In this case, for any $\alpha\in\mathbb{R}$,
$W^{\alpha}$ is also a matrix-valued function
whose entries are all measurable functions on $\mathbb{R}^n$
(see, for instance, \cite[Lemma 2.3.5]{rs95}).

\begin{definition}\label{def-locAp}
Let $r\in(0,\infty)$ and $1<p\leq q<\infty$.
The \emph{local matrix Muckenhoupt class
$\mathscr{A}^{\operatorname{loc}}_{p,q}(r)$}
is defined to be the set of all matrix weights
$W:\ \mathbb{R}^n\to D_m(\mathbb{C})$ such that
\begin{align}\label{eq-locApq}
[W]_{\mathscr{A}^{\operatorname{loc}}_{p,q}(r)}
:=\sup_{\{\mathrm{cube}\,Q:\,\ell(Q)\leq r\}}
\fint_Q\left[\fint_Q\left\|W^{\frac{1}{q}}(x)
W^{-\frac{1}{q}}(y)\right\|^{p^{\prime}}
\,dy\right]^{\frac{q}{p^{\prime}}}\,dx<\infty.
\end{align}
\end{definition}

\begin{remark}
In \eqref{eq-locApq}, if
the supremum is taken over all cubes $Q\subset \mathbb{R}^n$,
then \eqref{eq-locApq} defines the
\emph{matrix Muckenhoupt class $\mathscr{A}_{p,q}$} introduced in
\cite[p.\,1331]{im19}. Furthermore, if $p=q$,
the class $\mathscr{A}_{p,q}$  is exactly the well-known
\emph{matrix Muckenhoupt class $\mathscr{A}_{p}$}.
When $m=1$, the local matrix Muckenhoupt class
$\mathscr{A}^{\operatorname{loc}}_{p,q}(r)$
reduces to the \emph{local scalar Muckenhoupt class
$A^{\operatorname{loc}}_{p,q}(r)$}.
\end{remark}

To characterize the condition \eqref{eq-locApq} of
matrix weights in terms of their associated averaging matrices,
we recall the concept of reducing operators,
which was originally introduced by Volberg \cite[(3.1)]{v97}.
Let $p\in(1,\infty)$ and $E\subset\mathbb{R}^n$ be a bounded measurable set
with positive measure. Suppose that $W$ is a matrix weight.
From \cite[Proposition 1.2]{g03},
it follows that there exists a positive definite matrix
$A_E$ such that, for any $\vec z\in\mathbb{C}^m$,
\begin{equation}\label{eq-reduce}
\left[\fint_E\left|W^{\frac{1}{p}}(x)\vec{z}
\right|^p\,dx\right]^{\frac{1}{p}}\leq
\left|A_E\vec z\right|\leq \sqrt{m}\left[\fint_E
\left|W^{\frac{1}{p}}(x)\vec z\right|^p\,dx\right]^{\frac{1}{p}}.
\end{equation}
We call $A_E$ the \emph{reducing operator}
of order $p$ for $W$. The next equivalence immediately follows from
\eqref{eq-reduce}; we omit the details.

\begin{lemma}\label{lem-W-dualW-dkps}
Let $r\in(0,\infty)$, $1<p\leq q<\infty$, and $W$ be a matrix weight
such that $W^{-\frac{p^{\prime}}{q}}$ is locally integrable.
If $\{A_Q\}_{\{\operatorname{cube}\,Q:\,\ell(Q)\leq r\}}$ and
$\{\widetilde{A}_Q\}_{\{\operatorname{cube}\,Q:\,\ell(Q)\leq r\}}$
are sequences of reducing operators respectively,
of order $q$ for $W$ and of order $p^{\prime}$ for $W^{-\frac{p^{\prime}}{q}}$,
then
\begin{align*}
[W]^{\frac{1}{q}}_{\mathscr{A}^{\operatorname{loc}}_{p,q}(r)}
&\sim\sup_{\{\operatorname{cube}\,Q:\,\ell(Q)\leq r\}}
\left\|A_Q\widetilde{A}_Q\right\|,
\end{align*}
where the positive equivalence constants depend only on $p$, $q$, and $m$.
\end{lemma}

It is well known that,
for any $p\in(1,\infty)$, a scalar weight $w\in A_p$
if and only if $w^{-\frac{p^{\prime}}{p}}\in A_{p^{\prime}}$.
The following corollary provides a corresponding
duality property of local matrix weights.

\begin{corollary}\label{cor-Ap-Ap'}
Let $r\in(0,\infty)$, $1<p\leq q<\infty$, and
$W$ be a matrix weight. Then
$W\in\mathscr{A}^{\operatorname{loc}}_{p,q}(r)$ if
and only if $\widetilde{W}:=W^{-\frac{p^{\prime}}{q}}\in
\mathscr{A}^{\operatorname{loc}}_{q^{\prime},p^{\prime}}(r)$.
Moreover,
\begin{align}\label{eq-W-dualW}
[W]^{\frac{1}{q}}_{\mathscr{A}^{\operatorname{loc}}_{p,q}(r)}
&\sim\left[\widetilde{W}\right]^{\frac{1}{p^{\prime}}}
_{\mathscr{A}^{\operatorname{loc}}_{q^{\prime},p^{\prime}}(r)},
\end{align}
where the positive equivalence constants depend only on $p$, $q$, and $m$.
\end{corollary}
\begin{proof}
By symmetry and Lemma \ref{lem-W-dualW-dkps}, it is enough to show that,
if $W\in\mathscr{A}^{\operatorname{loc}}_{p,q}(r)$,
then $\widetilde{W}$ is locally integrable.
For any cube $Q\subset\mathbb{R}^n$ with
$\ell(Q)\leq r$ and for almost every $x\in Q$,
\begin{align*}
\int_{Q}\left\|W^{-\frac{p^{\prime}}{q}}(y)\right\|\,dy
&=\int_{Q}\left\|W^{-\frac{1}{q}}(y)\right\|^{p^{\prime}}\,dy\\
&\leq\int_Q\left\|W^{\frac{1}{q}}(x)
W^{-\frac{1}{q}}(y)\right\|^{p^{\prime}}\,dy
\left\|W^{-\frac{1}{q}}(x)\right\|^{p^{\prime}}.
\end{align*}
Taking the integral on $Q$ with respect to the variable $x$ and using
the assumption that $W\in\mathscr{A}^{\operatorname{loc}}_{p,q}(r)$,
we conclude that
\begin{align*}
&\int_{Q}\left\|W^{-1}(x)\right\|^{-1}\,dx
\left[\int_{Q}\left\|W^{-\frac{p^{\prime}}{q}}(y)\right\|\,dy
\right]^{\frac{q}{p^{\prime}}}\\
&\quad\leq\int_{Q}\left[\int_{Q}\left\|W^{\frac{1}{q}}(x)
W^{-\frac{1}{q}}(y)\right\|^{p^{\prime}}
\,dy\right]^{\frac{q}{p^{\prime}}}\,dx<\infty.
\end{align*}
This, together with the fact that $W(x)$ is positive definite for
almost every $x\in\mathbb{R}^n$, further implies that
\begin{align*}
\int_{Q}\left\|W^{-\frac{p^{\prime}}{q}}(y)\right\|\,dy<\infty.
\end{align*}
Thus, it follows from the arbitrariness of cubes $Q$
that $\widetilde{W}$ is locally integrable.
This completes the proof of Corollary \ref{cor-Ap-Ap'}.
\end{proof}

The following lemma provides a uniform scalarization
of local matrix weights.

\begin{lemma}\label{lem-W-w}
If $r\in(0,\infty)$, $1<p\leq q<\infty$, and $W\in\mathscr{A}^{\operatorname{loc}}_{p,q}(r)$,
then, for any $\vec{z}\in\mathbb{C}^m\setminus\{\mathbf{0}\}$,
$w_{\vec{z}}:=|W^{\frac{1}{q}}\vec{z}|^q\in
A^{\operatorname{loc}}_{\gamma}(r)$ and
\begin{align}\label{eq-W-w-sup}
\sup_{\vec{z}\in\mathbb{C}^m\setminus\{\mathbf{0}\}}
\left[w_{\vec{z}}\right]_{A^{\operatorname{loc}}_{\gamma}(r)}
\leq \left[W\right]_{\mathscr{A}^{\operatorname{loc}}_{p,q}(r)},
\end{align}
where $\gamma:= 1 + \frac{q}{p^{\prime}}$.
\end{lemma}
\begin{proof}
By the definition of operator norms, we find that,
for almost every $x,y\in\mathbb{R}^n$,
\begin{align}\label{eq-W-w}
\left\|W^{\frac{1}{q}}(x)W^{-\frac{1}{q}}(y)\right\|
&=\sup_{\vec{z}\in\mathbb{C}^m\setminus\{\mathbf{0}\}}
\frac{|W^{\frac{1}{q}}(x)W^{-\frac{1}{q}}(y)\vec{z}|}{|\vec{z}|}\\
&=\sup_{\vec{z}\in\mathbb{C}^m\setminus\{\mathbf{0}\}}
\frac{|W^{\frac{1}{q}}(x)\vec{z}|}{|W^{\frac{1}{q}}(y)\vec{z}|}
=\sup_{\vec{z}\in\mathbb{C}^m\setminus\{\mathbf{0}\}}
\left[w_{\vec{z}}(x)\right]^{\frac{1}{q}}
\left[w_{\vec{z}}(y)\right]^{-\frac{1}{q}}.\nonumber
\end{align}
Note that $\frac{\gamma^{\prime}}{\gamma} =\frac{p^{\prime}}{q}$.
Applying this, \eqref{eq-W-w}, and Definitions \ref{def-w-locAp}
and \ref{def-locAp}, we obtain, for any $\vec{z}\in
\mathbb{C}^m\setminus\{\mathbf{0}\}$,
\begin{align*}
\left[w_{\vec{z}}\right]_{A^{\operatorname{loc}}_{\gamma}(r)}
&=\sup_{\{\mathrm{cube}\,Q:\,\ell(Q)\leq r\}}
\fint_Q w_{\vec{z}}(x)\,dx\left[\fint_Q
w_{\vec{z}}(x)^{-\frac{\gamma^{\prime}}{\gamma}}
\,dx\right]^{\frac{\gamma}{\gamma^{\prime}}}\\
&\leq\sup_{\{\mathrm{cube}\,Q:\,\ell(Q)\leq r\}}
\fint_Q\left[\fint_Q\left\|W^{\frac{1}{q}}(x)
W^{-\frac{1}{q}}(y)\right\|^{p^{\prime}}
\,dy\right]^{\frac{q}{p^{\prime}}}\,dx
=[W]_{\mathscr{A}^{\operatorname{loc}}_{p,q}(r)},
\end{align*}
which further implies \eqref{eq-W-w-sup}.
This completes the proof of Lemma \ref{lem-W-w}.
\end{proof}

As a corollary of \eqref{eq-W-w-sup},
we show that the $\mathscr{A}^{\operatorname{loc}}_{p,q}(r)$
constant is always greater than or equal to 1.

\begin{corollary}\label{coro-[W]geq1}
Let $r\in(0,\infty)$ and $1<p\leq q<\infty$.
For any matrix weight $W$,
$[W]_{\mathscr{A}^{\operatorname{loc}}_{p,q}(r)}\in[1,\infty]$.
\end{corollary}
\begin{proof}
If $W\notin \mathscr{A}^{\operatorname{loc}}_{p,q}(r)$,
then $[W]_{\mathscr{A}^{\operatorname{loc}}_{p,q}(r)}=\infty$
and the present corollary naturally holds.
In the remaining proof, we always assume that $W\in
\mathscr{A}^{\operatorname{loc}}_{p,q}(r)$.
It follows from \cite[Proposition 2.1.1]{ooi24}
and Lemma \ref{lem-W-w} that, for any
$\vec{z}\in\mathbb{C}^m\setminus\{\mathbf{0}\}$,
\begin{align*}
w_{\vec{z}}:=|W^{\frac{1}{q}}\vec{z}|^q\in A^{\operatorname{loc}}_{\gamma}(r)
\text{ and }[w_{\vec{z}}]_{A^{\operatorname{loc}}_{\gamma}(r)}\geq 1,
\end{align*}
where $\gamma:= 1 + \frac{q}{p^{\prime}}$. This, together with
\eqref{eq-W-w-sup}, further implies that
$[W]_{\mathscr{A}^{\operatorname{loc}}_{p,q}(r)}\geq 1$,
which completes the proof of Corollary \ref{coro-[W]geq1}.
\end{proof}

\section{Extension and Scaling Invariance of Local Matrix Weights}\label{s2}

In this section, we first establish the extension property
for the local matrix weight class $\mathscr{A}^{\operatorname{loc}}_{p,q}(r)$,
which provides a connection
between local matrix weights and their global counterparts.
Next, via a sharp quantitative estimate of
local matrix weight constants, we prove that
$\mathscr{A}^{\operatorname{loc}}_{p,q}(r)$ has a scale lifting property,
which further implies that $\mathscr{A}^{\operatorname{loc}}_{p,q}(r)$
is independent of the scale $r$.

The following theorem provides the aforementioned extension property
for local matrix  Muckenhoupt weights. Its proof is inspired by some ideas
from the proofs of \cite[Lemma 1.1]{ry01} and \cite[Theorem 2.2.1]{ooi24}.

\begin{theorem}\label{thm-extension}
Let $r\in(0,\infty)$,  $1<p\leq q<\infty$, and
$W\in \mathscr{A}^{\operatorname{loc}}_{p,q}(r)$.
For any given cube $Q$ with $\ell(Q)= r$, there
exists $V\in \mathscr{A}_{p,q}$ satisfying
$V=W$ on $Q$ and
\begin{align}\label{eq-norm-V-W}
\left[V\right]_{\mathscr{A}_{p,q}}\leq 3^{n\gamma}
\left[W\right]_{\mathscr{A}^{\operatorname{loc}}_{p,q}(r)},
\end{align}
where $\gamma:= 1 + \frac{q}{p^{\prime}}$.
\end{theorem}

\begin{proof}
Without loss of generality, we may assume that the lower-left
corner of $Q$ is the origin $\mathbf{0}$; otherwise translate
the coordinate system such that the lower-left corner of $Q$ is the origin.
To begin with, let $J$ be the cube centered at $\mathbf{0}$
with edge length $2r$.
For any $x:=(x_1,\dots,x_n)\in J$,
let $$V(x):=W(|x_1|,\dots,|x_n|).$$
This definition gives a mirror-symmetric extension of
$W$ from $Q$ to the larger cube $J$.
Next, we extend $V$ from $J$ to $\mathbb{R}^n$ periodically.
It immediately follows from the definition of $V$ that
$V=W$ on $Q$. Therefore, to complete the
present proof, it suffices to show that $V\in \mathscr{A}_{p,q}$
and \eqref{eq-norm-V-W} holds. To this end,
for any cube $R$ in $\mathbb{R}^n$ with edge length $t\in(0,\infty)$,
we consider the following two cases for $t$ and $r$.

\emph{Case (1)} $t\in(0,r)$.
Observe that there exists a cube $K\subset Q$ with edge length $t$
such that $R$ is covered by at most $2^n$ copies of
$K$, where each copy is generated by reflections across
the coordinate hyperplanes together with $2r$-periodic
translations along coordinate directions.
Let $E$ be the union of these copies.
From this, the construction of $V$, and changes of variables, we infer that
\begin{align}\label{eq-V-Case1}
&\fint_R\left[\fint_R\left\|V^{\frac{1}{q}}(x)
V^{-\frac{1}{q}}(y)\right\|^{p^{\prime}}
\,dy\right]^{\frac{q}{p^{\prime}}}\,dx\\
&\quad\leq\fint_R\left[\frac{1}{|R|}\int_{E}
\left\|V^{\frac{1}{q}}(x)V^{-\frac{1}{q}}(y)\right\|^{p^{\prime}}
\,dy\right]^{\frac{q}{p^{\prime}}}\,dx\nonumber\\
&\quad\leq\fint_R\left[\frac{2^n}{|R|}\int_{K}
\left\|V^{\frac{1}{q}}(x)W^{-\frac{1}{q}}(y)\right\|^{p^{\prime}}
\,dy\right]^{\frac{q}{p^{\prime}}}\,dx\nonumber\\
&\quad\leq2^{n\frac{q}{p^{\prime}}}\frac{1}{|R|}
\int_{E}\left[\fint_{K}
\left\|V^{\frac{1}{q}}(x)W^{-\frac{1}{q}}(y)\right\|^{p^{\prime}}
\,dy\right]^{\frac{q}{p^{\prime}}}\,dx\nonumber\\
&\quad\leq2^{n\frac{q}{p^{\prime}}}\frac{2^n}{|R|}
\int_{K}\left[\fint_{K}
\left\|W^{\frac{1}{q}}(x)W^{-\frac{1}{q}}(y)\right\|^{p^{\prime}}
\,dy\right]^{\frac{q}{p^{\prime}}}\,dx\nonumber\\
&\quad=2^{n\gamma}\fint_{K}
\left[\fint_{K}\left\|W^{\frac{1}{q}}(x)
W^{-\frac{1}{q}}(y)\right\|^{p^{\prime}}
\,dy\right]^{\frac{q}{p^{\prime}}}\,dx\leq 2^{n\gamma}
\left[W\right]_{\mathscr{A}^{\operatorname{loc}}_{p,q}(r)}.\nonumber
\end{align}

\emph{Case (2)} $t\in[r,\infty)$. In this case,
we can cover $R$ by at most $N:=(\lceil\frac{t}{r}\rceil+1)^n$
translated copies of $Q$.
Using this,  the definition of $V$, and the same calculation
as in Case (1), we conclude that
\begin{align}\label{eq-V-Case2}
&\fint_R\left[\fint_R\left\|V^{\frac{1}{q}}(x)
V^{-\frac{1}{q}}(y)\right\|^{p^{\prime}}\,dy
\right]^{\frac{q}{p^{\prime}}}\,dx\\
&\quad\leq 	\frac{N}{|R|}
\int_Q\left[\frac{N}{|R|}
\int_Q\left\|W^{\frac{1}{q}}(x)W^{-\frac{1}{q}}(y)\right\|^{p^{\prime}}
\,dy\right]^{\frac{q}{p^{\prime}}}\,dx\nonumber\\
&\quad\leq\left[N\left(\frac{r}{t}\right)^n\right]^{\gamma}
\fint_{Q}\left[\fint_{Q}\left\|W^{\frac{1}{q}}(x)
W^{-\frac{1}{q}}(y)\right\|^{p^{\prime}}\,dy
\right]^{\frac{q}{p^{\prime}}}\,dx
\leq 3^{n\gamma}
\left[W\right]_{\mathscr{A}^{\operatorname{loc}}_{p,q}(r)}.\nonumber
\end{align}
Applying \eqref{eq-V-Case1} and \eqref{eq-V-Case2}, we obtain
\begin{align*}
\left[V\right]_{\mathscr{A}_{p,q}}=\sup_{\mathrm{cube}\ R}
\fint_R\left[\fint_R\left\|V^{\frac{1}{q}}(x)
V^{-\frac{1}{q}}(y)\right\|^{p^{\prime}}\,dy\right]^{\frac{q}{p^{\prime}}}\,dx
\leq 3^{n\gamma}
\left[W\right]_{\mathscr{A}^{\operatorname{loc}}_{p,q}(r)},
\end{align*}
which further implies that
$V\in \mathscr{A}_{p,q}$ and \eqref{eq-norm-V-W} holds.
This completes the proof of Theorem \ref{thm-extension}.
\end{proof}

To show the scaling invariance of local matrix weights,
we first establish the following lemma on the
reducing operators generated by local matrix weights.

\begin{lemma}\label{lem:AQ1AQ2}
Let $r\in(0,\infty)$, $1<p\leq q<\infty$, and
$W\in \mathscr{A}^{\operatorname{loc}}_{p,q}(r)$.
Suppose that $\{A_Q\}_{\{\operatorname{cube}\,Q:\ \ell(Q)\leq r\}}$ and
$\{\widetilde{A}_Q\}_{\{\operatorname{cube}\,Q:\ \ell(Q)\leq r\}}$
are sequences of reducing operators, respectively,
of order $q$ for $W$ and of order $p^{\prime}$ for $W^{-\frac{p^{\prime}}{q}}$.
Assume that $c\in(0,1]$. Then, for any cubes $Q_1,Q_2,Q\subset\mathbb{R}^n$
satisfying $\ell(Q)\leq r$, $Q_1,Q_2\subset Q$, and
$\ell(Q_i)\geq c\ell(Q)$ for $i=1,2$,
\begin{equation*}
\left\|A_{Q_1}\widetilde A_{Q_2}\right\|\lesssim
[W]^{\frac{1}{q}}_{\mathscr{A}^{\operatorname{loc}}_{p,q}(r)},
\end{equation*}
where the implicit positive constant
depends only on $c$, $p$, $q$, $n$ and $m$.
\end{lemma}

\begin{proof}
By the definition of reducing operators \eqref{eq-reduce}
and the assumptions that $Q_1\subset Q$ and $\ell(Q_1)\sim \ell(Q)$,
we find that, for any vector $e\in\mathbb{C}^m$,
\begin{align}\label{eq:AQ1e<AQe}
\left|A_{Q_1}e\right|
\sim\left[\fint_{Q_1}\left|W^{\frac 1q}(x)e\right|^q\,dx\right]^{\frac1q}
\lesssim\left[\fint_{Q}\left|W^{\frac 1q}(x)e\right|^q\,dx\right]^{\frac1q}
\sim\left|A_{Q} e\right|.
\end{align}
Let $M\in M_m(\mathbb{C})$. Using \eqref{eq:AQ1e<AQe} with $Me$ in place of $e$,
we conclude that
\begin{equation}\label{eq:AQ1M<AQM}
\left\|A_{Q_1}M\right\|
=\sup_{\genfrac{}{}{0pt}{}{e\in\mathbb{C}^m}{|e|\leq 1}}
\left|A_{Q_1} M e\right|
\lesssim \sup_{\genfrac{}{}{0pt}{}{e\in\mathbb{C}^m}{|e|\leq 1}}
\left|A_{Q} M e\right|
=\left\|A_{Q} M\right\|.
\end{equation}
Similarly, we also have
\begin{equation}\label{eq:AQ2M<AQM}
\left\|\widetilde A_{Q_2}M\right\|\lesssim
\left\|\widetilde A_{Q} M\right\|.
\end{equation}
Applying \eqref{eq:AQ1M<AQM}, \eqref{eq:AQ2M<AQM},
the transpose invariance of the standard norm in $\mathbb{C}^m$
(see, for instance, \cite[Theorem 5.6.2(d)]{hj13}),
and Lemma \ref{lem-W-dualW-dkps}, we obtain
\begin{align*}
\left\|A_{Q_1}\widetilde A_{Q_2}\right\|
\lesssim\left\|A_{Q}\widetilde A_{Q_2}\right\|
=\left\|\widetilde A_{Q_2} A_{Q}\right\|
\lesssim\left\|\widetilde A_{Q} A_{Q}\right\|
=\left\|A_{Q}\widetilde A_{Q}\right\|
\lesssim[W]_{\mathscr{A}^{\operatorname{loc}}_{p,q}(r)}^{\frac1q}.
\end{align*}
This completes the proof of Lemma \ref{lem:AQ1AQ2}.
\end{proof}

Note that it immediately follows from \eqref{eq-locApq}
that the constant $[W]_{\mathscr{A}^{\operatorname{loc}}_{p,q}(r)}$
grows as $r$ increases. In the next theorem, we establish a delicate local
lifting property for the local matrix weight class $\mathscr{A}^{\operatorname{loc}}_{p,q}(r)$
to rigorously quantify this growth.

\begin{theorem}\label{thm:Apq(1+eta)}
Let $\eta\in(0,1)$ and
$1<p\leq q<\infty$. Then, for any $r\in(0,\infty)$ and
$W\in \mathscr{A}^{\operatorname{loc}}_{p,q}(r)$,
\begin{align}\label{eq-Apq(1+eta)}
[W]_{\mathscr{A}_{p,q}^{\operatorname{loc}}((1+\eta)r)}
\lesssim[W]_{\mathscr{A}^{\operatorname{loc}}_{p,q}(r)}^2,
\end{align}
where the implicit positive constant depends on $\eta$
but is independent of $r$ and $W$.
\end{theorem}

\begin{proof}
Let $\{A_Q\}_{\{\operatorname{cube}\,Q\}}$ and
$\{\widetilde{A}_Q\}_{\{\operatorname{cube}\,Q\}}$
be sequences of reducing operators, respectively,
of order $q$ for $W$ and of order $p^{\prime}$ for $W^{-\frac{p^{\prime}}{q}}$.
By Lemma \ref{lem-W-dualW-dkps}, we find that

\begin{align}\label{eq-sup-max}
[W]_{\mathscr{A}_{p,q}^{\operatorname{loc}}((1+\eta)r)}^{\frac1q}
\sim\max\left\{\sup_{\genfrac{}{}{0pt}{}{\text{ cube}\,Q}{\ell(Q)\leq r}}
\left\|A_Q\widetilde A_Q\right\|,
\sup_{\genfrac{}{}{0pt}{}{\text{ cube}\,Q}{r<\ell(Q)\leq (1+\eta)r}}
\left\|A_Q\widetilde A_Q\right\|\right\}
=:\max\{\operatorname{I}, \operatorname{II}\}.
\end{align}
Using Lemma \ref{lem-W-dualW-dkps},
we conclude that
\begin{equation*}
\operatorname{I}\sim [W]_{\mathscr{A}_{p,q}^{\operatorname{loc}}(r)}^{\frac1q}.
\end{equation*}
Next, we claim that
\begin{align}\label{eq-claimII}
\operatorname{II}\lesssim[W]_{
\mathscr{A}_{p,q}^{\operatorname{loc}}(r)}^{\frac2q}.
\end{align}
Since $[W]_{\mathscr{A}_{p,q}^{\operatorname{loc}}(r)}\geq 1$ by
Corollary \ref{coro-[W]geq1},
once \eqref{eq-claimII} is proved, then the present theorem
directly follows from \eqref{eq-sup-max} and \eqref{eq-claimII}.

Finally, we show \eqref{eq-claimII}. Let
$Q\subset\mathbb{R}^n$ be a cube with $r<\ell(Q)\leq(1+\eta)r$.
By bisecting each edge of $Q$, we obtain the subcubes
$\{Q_i\}_{i=1}^{2^n}$ of $Q$. Then, for any
$i\in\{1,\dots,2^n\}$,
\begin{equation*}
\ell(Q_i)=\tfrac12\ell(Q)\in\left(\frac12 r,\frac12\left(1+\eta\right)r\right]
\subset\left(\frac12 r,r\right).
\end{equation*}
For any $i\in\{1,\dots,2^n\}$, let $\widetilde{Q}_i$ be a cube
with edge length $\ell(\widetilde{Q}_i)=r$ and
$Q_i\subset \widetilde{Q}_i \subset Q$.
Let $\widetilde{Q}:=\bigcap_{i=1}^{2^n}\widetilde{Q}_i$.
It is easy to find that $\widetilde{Q}$ is
a cube concentric with $Q$ and $\ell(\widetilde{Q})=
2r-\ell(Q)$. Applying the construction of
$\{Q_i\}_{i=1}^{2^n}$ and applying \eqref{eq-reduce},
we conclude that, for any vector $e\in\mathbb{C}^m$,
\begin{equation*}
\left|A_Q e\right|
\sim\left[\fint_Q\left|W^{\frac 1q}(x)e\right|^q\,dx\right]^{\frac1q}
=\left[2^{-n}\sum_{i=1}^{2^n}\fint_{Q_i}\left|W^{\frac 1q}(x)e\right|^q\,dx\right]^{\frac1q}
\sim\left(2^{-n}\sum_{i=1}^{2^n}\left|A_{Q_i} e\right|^q\right)^{\frac1q},
\end{equation*}
which, together with the definition of the operator norm,
further implies that, for any $M\in M_m(\mathbb{C})$,
\begin{align}\label{eq:AQM<sum}
\left\|A_Q M\right\|
=\sup_{\genfrac{}{}{0pt}{}{e\in\mathbb{C}^m}{|e|\leq 1}}
\left|A_Q Me\right|
&\lesssim\left(2^{-n}\sum_{i=1}^{2^n} \sup_{\genfrac{}{}{0pt}{}{e\in\mathbb{C}^m}{|e|\leq 1}}
\left|A_{Q_i}Me\right|^q\right)^{\frac1q}\\
&=\left(2^{-n}\sum_{i=1}^{2^n}\left\|A_{Q_i} M\right\|^q\right)^{\frac1q}
\leq\max_{1\leq i\leq 2^n}\left\|A_{Q_i} M\right\|.\nonumber
\end{align}
Similarly, we also have, for any $M\in M_m(\mathbb{C})$,
\begin{align}\label{eq:tildeAQM<sum}
\left\|\widetilde A_Q M\right\|
\lesssim\max_{1\leq j\leq 2^n}\left\|\widetilde A_{Q_j}M\right\|.
\end{align}
Repeatedly using \eqref{eq:AQM<sum} and \eqref{eq:tildeAQM<sum},
we conclude that
\begin{align}\label{eq:AQtildeAQ<}
\left\|A_Q\widetilde A_Q\right\|
&\lesssim\max_{1\leq i\leq 2^n}\left\|A_{Q_i}\widetilde A_Q\right\|
=\max_{1\leq i\leq 2^n}\left\|\widetilde A_QA_{Q_i}\right\|\\
&\lesssim\max_{1\leq i,j\leq 2^n}\left\|\widetilde A_{Q_j}A_{Q_i}\right\|
=\max_{1\leq i,j\leq 2^n}\left\|A_{Q_i}\widetilde A_{Q_j}\right\|.\nonumber
\end{align}
For any $i,j\in\{1,\dots,2^n\}$, we continue with
\begin{align}\label{eq:AQiAQj<}
\left\|A_{Q_i}\widetilde A_{Q_j}\right\|
&=\left\|A_{Q_i}A_{\widetilde{Q}}^{-1}A_{\widetilde{Q}}\widetilde A_{Q_j}\right\|\\
&\lesssim\left\|A_{Q_i}A_{\widetilde{Q}}^{-1}\right\|
\left\|A_{\widetilde{Q}}\widetilde A_{Q_j}\right\|
=\left\|A_{\widetilde{Q}}^{-1}A_{Q_i}\right\|
\left\|A_{\widetilde{Q}}\widetilde A_{Q_j}\right\|.\nonumber
\end{align}
By \cite[Lemma 2.8]{BCHYY}, we find that
\begin{align*}
\left\|A_{\widetilde{Q}}^{-1}A_{Q_i}\right\|
\lesssim\left\|\widetilde A_{\widetilde{Q}}A_{Q_i}\right\|.
\end{align*}
Substituting this back to \eqref{eq:AQiAQj<}, we obtain
\begin{align}\label{eq-AQiAQj}
\left\|A_{Q_i}\widetilde A_{Q_j}\right\|
\lesssim\left\|\widetilde A_{\widetilde{Q}}A_{Q_i}\right\|
\left\|A_{\widetilde{Q}}\widetilde A_{Q_j}\right\|
=\left\|A_{Q_i} \widetilde A_{\widetilde{Q}}\right\|
\left\|A_{\widetilde{Q}}\widetilde A_{Q_j}\right\|.
\end{align}
Note that both factors on the right-hand side of
\eqref{eq-AQiAQj}
have the same form $\left\|A_{Q_1}\widetilde A_{Q_2}\right\|$
as considered in Lemma \ref{lem:AQ1AQ2}.
More precisely, in the first factor, $Q_1:=Q_i$ and
$Q_2:=\widetilde{Q}$ of lengths $\ell(Q_i)>\frac12r$ and
$\ell(\widetilde{Q})=2r-\ell(Q)
\geq(1-\eta)r$ are both contained in $Q:=\widetilde{Q}_i$ of length $r$.
Moreover, in the second factor, $Q_1:=\widetilde{Q}$ and $Q_2:=Q_j$ of
lengths $\ell(\widetilde{Q})\geq(1-\eta)r$ and
$\ell(Q_j)>\frac12r$ are both contained in $Q:=\widetilde{Q}_j$ of length $r$.
Thus, it follows from Lemma \ref{lem:AQ1AQ2}
that both factors on the right-hand side of
\eqref{eq-AQiAQj} are bounded by the value
$[W]_{\mathscr{A}^{\operatorname{loc}}_{p,q}(r)}^{\frac 1q}$
(with implicit positive constants depending on $\eta$).
As the product of two such factors, we obtain
\begin{equation*}
\left\|A_{Q_i}\widetilde A_{Q_j}\right\|
\lesssim [W]_{\mathscr{A}^{\operatorname{loc}}_{p,q}(r)}^{\frac2q},
\end{equation*}
which, together with \eqref{eq:AQtildeAQ<},
further implies \eqref{eq-claimII}. This completes the proof
of the claim \eqref{eq-claimII} and hence Theorem \ref{thm:Apq(1+eta)}.
\end{proof}

\begin{remark}
Let the notation be the same as in Theorem \ref{thm:Apq(1+eta)}.
As a direct corollary, we find that the class of weights
$\mathscr{A}_{p,q}^{\operatorname{loc}}(r)$ is indeed independent of the
choice of $r\in(0,\infty)$. Thus, in Theorems \ref{thm-bound-maxloc}
and \ref{thm-bound-fracintloc} below, even if we only assume that
$W\in\mathscr{A}_{p,q}^{\rm loc}(r)$, then $W$ also belongs to
$\mathscr{A}_{p,q}^{\rm loc}(r+\varepsilon)$
for any $\varepsilon\in(0,\infty)$.
\end{remark}

Finally, we present a one-dimensional scalar example to
show that the exponent $2$ in \eqref{eq-Apq(1+eta)} is optimal.
Since Theorem \ref{thm:Apq(1+eta)}, particularly this exponent,
plays a crucial role in establishing the sharp bounds for
(fractional) maximal operators [see \eqref{eq-bound-maxloc-1} and its proof],
this optimality implies that our current approach
cannot be used to relax the restriction
on $p$ and $q$ in \eqref{eq-bound-maxloc-1}.

\begin{example}\label{exam-sharpness-power-2}
Let $M\in[1,\infty)$, $\eta\in(0,1)$, and $1<p\leq q<\infty$.
For any $x\in\mathbb{R}$, let
\begin{align*}
w_M(x) :=
\begin{cases}
M &\displaystyle \text{if } x \in \left[-\frac{\eta}{2}r, 0\right]=:I_1,\\
M^{-1} &\displaystyle \text{if } x \in
\left[r, \left(1+\frac{\eta}{2}\right)r\right]=: I_2,\\
1 & \text{otherwise}.
\end{cases}
\end{align*}
Next, we prove that $[w_M]_{A^{\operatorname{loc}}_{p,q}(r)}\leq M$
and $[w_M]_{A^{\operatorname{loc}}_{p,q}((1+\eta)r)}\gtrsim M^2$,
where the implicit positive constant is independent of $M$.

To this end, for any interval $I\subset\mathbb{R}$, let
\begin{align*}
\mathscr{N}(I):=\fint_I w_M(x)\,dx\left\{\fint_I \left[w_M(x)\right]^{-\frac{p^{\prime}}{q}}
\,dx\right\}^{\frac{q}{p^{\prime}}}.
\end{align*}

To prove $[w_M]_{A^{\operatorname{loc}}_{p,q}(r)}\leq M$,
suppose that $I$ is an interval
in $\mathbb{R}$ with $\ell(I)\in(0,r]$.
If $|I\cap I_1|=0$ and $|I\cap I_2|=0$, then $w$ equals 1 a.e. on $I$
and hence $\mathscr{N}(I)=1\leq M$. If $|I\cap I_1|\neq0$, it follows from
the definition of $w_M$ that $|I\cap I_2|=0$. Thus,
\begin{align*}
\mathscr{N}(I)\leq\fint_I M\,dx\left[\fint_I 1
\,dx\right]^{\frac{q}{p^{\prime}}}=M.
\end{align*}
Applying the same argument as above, we also obtain, if $|I\cap I_2|\neq0$,
then $\mathscr{N}(I)\leq M$. Using these three cases and
the definition of $[w_M]_{A^{\operatorname{loc}}_{p,q}(r)}$,
we conclude that
\begin{align*}
\left[w_M\right]_{A^{\operatorname{loc}}_{p,q}(r)}=
\sup_{\genfrac{}{}{0pt}{}{\operatorname{interval}
\ I\subset\mathbb{R}}{\ell(I)\leq r}}
\mathscr{N}(I)\leq M.
\end{align*}

Finally, we show $[w_M]_{A^{\operatorname{loc}}_{p,q}((1+\eta)r)}\gtrsim M^2$.
Let $I_0:=[-\frac{\eta}{2}r, (1+\frac{\eta}{2})r]$. By the definition of
$w_M$, we find that
\begin{align*}
\mathscr{N}\left(I_0\right)&\geq\frac{1}{|I_0|}\int_{I_1} w_M(x)\,dx
\left\{\frac{1}{|I_0|}\int_{I_2} \left[w_M(x)\right]^{-\frac{p^{\prime}}{q}}
\,dx\right\}^{\frac{q}{p^{\prime}}}\\
&=\left[\frac{\eta}{(1+\eta)2}\right]^{1+\frac{q}{p^{\prime}}} M^2,
\end{align*}
which, combined with the definition of
$[w_M]_{A^{\operatorname{loc}}_{p,q}((1+\eta)r)}$,
further implies that
\begin{align*}
[w_M]_{A^{\operatorname{loc}}_{p,q}((1+\eta)r)}=
\sup_{\genfrac{}{}{0pt}{}{\operatorname{interval}
\ I\subset\mathbb{R}}{\ell(I)\leq (1+\eta)r}}
\mathscr{N}(I)\geq \mathscr{N}(I_0)\gtrsim M^2.
\end{align*}

Moreover, if the exponent 2 in \eqref{eq-Apq(1+eta)} could be improved
to some exponent $s\in[1,2)$, applying \eqref{eq-Apq(1+eta)} to $w_M$,
we obtain $M^2 \lesssim M^s$. Letting $M \to \infty$ immediately
yields a contradiction. Thus, the exponent 2 in \eqref{eq-Apq(1+eta)} is optimal.
\end{example}

\section{Quantitative Weighted Inequalities}\label{s3}

This section, consisting of four subsections, is devoted
to establishing the quantitative boundedness of various operators
on local matrix-weighted Lebesgue spaces. More precisely,
we obtain the quantitative weighted boundedness of
local fractional maximal operators, local fractional integral operators,
local Haar square functions, and Calder\'{o}n--Zygmund operators
with exponential decay, respectively, in
Subsections \ref{s3-1}, \ref{s3-2}, \ref{s3-3}, and \ref{s3-4}.
The main idea of our proofs is to use the extension property
from Theorem \ref{thm-extension} to establish the relationship between
local and global matrix-weighted boundedness,
combined with the sharp scale lifting property of local
matrix weights established in Theorem \ref{thm:Apq(1+eta)}.

\subsection{Local Fractional Maximal Operators}\label{s3-1}

This subsection is devoted to establishing the boundedness of
local fractional maximal operators on
local matrix-weighted Lebesgue spaces.
To this end, we first recall the definition of
matrix-weighted fractional maximal operators.
Let $\alpha\in[0,n)$, $q\in(1,\infty)$,
and $W$ be a matrix weight.
The \emph{matrix-weighted fractional maximal operator}
$M_{W,\alpha}$ is defined by setting,
for any measurable vector-valued function $\vec{f}$
on $\mathbb{R}^n$ and for any $x\in\mathbb{R}^n$,
\begin{align}\label{eq-def-fracmax}
M_{W,\alpha}\vec{f}(x):=\sup_{Q \ni x}
\frac{1}{|Q|^{1-\frac{\alpha}{n}}}
\int_Q\left|W^{\frac{1}{q}}(x)
W^{-\frac{1}{q}}(y) \vec{f}(y)\right|\,dy,
\end{align}
where the supremum is taken over all cubes $Q\subset\mathbb{R}^n$
that contain $x$ (see \cite[p.\,1331]{im19}).
If $\alpha=0$, the maximal operator $M_{W,0}=:M_{W}$
is exactly the Christ--Goldberg maximal function in \cite{g03}.
The following quantitative boundedness of $M_{W,\alpha}$ was established
by Isralowitz and Moen in \cite[Theorem 1.3]{im19}.
\begin{theorem}\label{thm-bound-max}
Let $\alpha\in[0,n)$, $p\in(1,\frac{n}{\alpha})$, and
$q\in(1,\infty)$ satisfy $\frac{1}{q}=
\frac{1}{p}-\frac{\alpha}{n}$. Then there exists
a positive constant $C$ such that, for any $W\in\mathscr{A}_{p,q}$,
\begin{align*}
\left\|M_{W,\alpha}\right\|_{L^p\to L^q}\leq C[W]^{\frac{p^{\prime}}{q}
(1-\frac{\alpha}{n})}_{\mathscr{A}_{p,q}},
\end{align*}
where $C$ is independent of $W$, and the exponent
$\frac{p^{\prime}}{q}(1-\frac{\alpha}{n})$ is sharp.
\end{theorem}

Let $r\in(0,\infty)$.
If we take the supremum in \eqref{eq-def-fracmax} over all
cubes in $\mathbb{R}^n$ with edge length less than
or equal to $r$, then we obtain the \emph{matrix-weighted
local fractional maximal operators} $M_{W,\alpha,r}$.
In the following theorem, we establish the quantitative boundedness
of $M_{W,\alpha,r}$ corresponding to Theorem \ref{thm-bound-max}.

\begin{theorem}\label{thm-bound-maxloc}
Let $r\in(0,\infty)$. Suppose that $\alpha\in[0,n)$,
$p\in(1,\frac{n}{\alpha})$,
and $q\in(1,\infty)$ satisfy $\frac{1}{q}=\frac{1}{p}-\frac{\alpha}{n}$.
Then, for any $\varepsilon\in(0, r]$ and
$W\in\mathscr{A}^{\operatorname{loc}}_{p,q}(r)$,
\begin{align}\label{eq-bound-maxloc}
\left\|M_{W,\alpha,r}\right\|_{L^p\to L^q}\lesssim
[W]^{\frac{p^{\prime}}{q}(1-\frac{\alpha}{n})}
_{\mathscr{A}^{\operatorname{loc}}_{p,q}(r)}
+\left(1+\frac{r}{\varepsilon}\right)^{\frac{n}{q}}
[W]^{\frac{1}{q}}
_{\mathscr{A}^{\operatorname{loc}}_{p,q}(r+\varepsilon)},
\end{align}
where the implicit positive constant is independent
of $r$, $\varepsilon$, and $W$. Moreover,
\begin{align}\label{eq-bound-maxloc-2}
\left\|M_{W,\alpha,r}\right\|_{L^p\to L^q}\lesssim
[W]^{\max\{\frac{p^{\prime}}{q}(1-\frac{\alpha}{n}), \frac{2}{q}\}}
_{\mathscr{A}^{\operatorname{loc}}_{p,q}(r)},
\end{align}
where the implicit positive constant is independent of $r$ and $W$.
If $q\leq p'$, the bound in \eqref{eq-bound-maxloc-2} can
be improved to
\begin{align}\label{eq-bound-maxloc-1}
\left\|M_{W,\alpha,r}\right\|_{L^p\to L^q}\lesssim
[W]^{\frac{p^{\prime}}{q}(1-\frac{\alpha}{n})}
_{\mathscr{A}^{\operatorname{loc}}_{p,q}(r)},
\end{align}
where the implicit positive constant is independent of
$r$ and $W$.
\end{theorem}

\begin{proof}
From the definition of $M_{W,\alpha,r}$, we first deduce that,
for any measurable vector-valued function
$\vec{f}$ on $\mathbb{R}^n$ and for any $x\in\mathbb{R}^n$,
\begin{align*}
M_{W,\alpha,r}\vec{f}(x)
&=\sup_{\genfrac{}{}{0pt}{}{Q\owns x}{\ell(Q)\leq r}}
\frac{1}{|Q|^{1-\frac{\alpha}{n}}}\int_Q\left|W^{\frac{1}{q}}(x)
W^{-\frac{1}{q}}(y) \vec{f}(y)\right|\, dy\\
&\leq\sup_{\genfrac{}{}{0pt}{}{Q\owns x}{\ell(Q)\leq r/2}}
\frac{1}{|Q|^{1-\frac{\alpha}{n}}}\int_Q\left|W^{\frac{1}{q}}(x)
W^{-\frac{1}{q}}(y) \vec{f}(y)\right|\, dy\\
&\quad+\sup_{\genfrac{}{}{0pt}{}{Q\owns x}{\ell(Q)\in(r/2,r]}}\cdots\\
&=:\operatorname{I}(x)+\operatorname{II}(x),
\end{align*}
where $\operatorname{I}(x)=M_{W,\alpha,\frac{r}{2}}\vec{f}(x)$.
On the other hand, if $Q\subset\mathbb{R}^n$ is a cube with
edge length $\ell(Q)\in(r/2,r]$ and $R\supset Q$
is a cube with edge length $\ell(R)=r$,
then $|R|\sim|Q|$, which further implies that,
for any $x\in\mathbb{R}^n$,
\begin{align}\label{eq-def-1scalemaxope}
\operatorname{II}(x)\sim
\sup_{\genfrac{}{}{0pt}{}{Q\owns x}{\ell(Q)=r}}
\frac{1}{|Q|^{1-\alpha/n}}
\int_Q\left|W^{\frac1q}(x)W^{-\frac1q}(y)\vec{f}(y)\right|\,dy
=:\widetilde{M}_{W,\alpha,r}\vec{f}(x).
\end{align}
Thus, to prove the boundedness of $M_{W,\alpha,r}$,
it suffices to show the boundedness of both
$M_{W,\alpha, \frac{r}{2}}$ and the one-scale maximal
operator $\widetilde{M}_{W,\alpha,r}$ from $L^p$ to $L^q$.

Next, we prove these two boundedness estimates in turn.
To prove the boundedness of $M_{W,\alpha, \frac{r}{2}}$,
we divide $\mathbb{R}^n$ into the collection $\Gamma$ of
cubes with pairwise disjoint interiors
and with edge length $\frac{r}{2}$.
For any $Q\in \Gamma$, by the structure of
$\mathbb{R}^n$, we can find a cube $\widetilde{Q}$ with
edge length $\frac{3}{2}r$ and with the same center as $Q$.
From the definition of $\widetilde Q$, we deduce that,
for any $x\in Q$ and any cube $R\ni x$ with $\ell(R)\le \frac{r}{2}$,
$R\subset\widetilde Q$, and hence
\begin{align*}
M_{W,\alpha,\frac{r}{2}}\vec{f}(x)=M_{W,\alpha,\frac{r}{2}}
\left(\mathbf{1}_{\widetilde{Q}}\vec{f}\right)(x)\leq
M_{W,\alpha}\left(\mathbf{1}_{\widetilde{Q}}\vec{f}\right)(x),
\end{align*}
where $M_{W,\alpha}\vec{f}$ is as in \eqref{eq-def-fracmax}.
Since there exist $2^n$ cubes $\{Q_i\}_{i=1}^{2^n}$
with edge length $r$ such that $\widetilde{Q}=
\cup_{i=1}^{2^n} Q_i$ and $Q\subset\cap_{i=1}^{2^n} Q_i$,
it follows that, for any $x\in Q$,
\begin{align}\label{eq-MW-MWa}
M_{W,\alpha}\left(\mathbf{1}_{\widetilde{Q}}\vec{f}\right)(x)
\leq\sum_{i=1}^{2^n}M_{W,\alpha}\left(\mathbf{1}_{Q_i}\vec{f}\right)(x).
\end{align}
Using Theorem \ref{thm-extension}, we conclude that,
for any $i\in\{1,\dots,2^n\}$,
there exists $W_i\in\mathscr{A}_{p,q}$ such that
\begin{align}\label{eq-Wi=W}
W_i=W\text{ on }Q_i\text{ and }\left[W_i\right]_{\mathscr{A}_{p,q}}\leq
3^{n(1+\frac{q}{p^{\prime}})}
\left[W\right]_{\mathscr{A}^{\operatorname{loc}}_{p,q}(r)}.
\end{align}
By \eqref{eq-MW-MWa} and \eqref{eq-Wi=W},
we find that, for any $x\in Q$,
\begin{align}\label{eq-W-i}
M_{W,\alpha,\frac{r}{2}}\vec{f}(x)&\leq
\sum_{i=1}^{2^n}M_{W,\alpha}\left(\mathbf{1}_{Q_i}\vec{f}\right)(x)
=\sum_{i=1}^{2^n}M_{W_i,\alpha}\left(\mathbf{1}_{Q_i}\vec{f}\right)(x).
\end{align}
Applying Theorem \ref{thm-bound-max} to each term
in the right-hand side of \eqref{eq-W-i}, we obtain
\begin{align}\label{eq-Q-Q}
&\int_{Q}\left[M_{W,\alpha,\frac{r}{2}}\vec{f}(x)\right]^q\,dx\\
&\quad\lesssim\sum_{i=1}^{2^n}\int_{Q}\left[M_{W_i,\alpha}
(\mathbf{1}_{Q_i}\vec{f})(x)\right]^q\,dx
\leq\sum_{i=1}^{2^n}\int_{\mathbb{R}^n}
\left[M_{W_i,\alpha}
(\mathbf{1}_{Q_i}\vec{f})(x)\right]^q\,dx\nonumber\\
&\quad\lesssim\sum_{i=1}^{2^n}[W_i]^{p^{\prime}
(1-\frac{\alpha}{n})}_{\mathscr{A}_{p,q}}
\left[\int_{Q_i}\left|\vec{f}(x)\right|^p\,dx\right]^{\frac{q}{p}}
\lesssim[W]^{p^{\prime}(1-\frac{\alpha}{n})}_{
\mathscr{A}^{\operatorname{loc}}_{p,q}(r)}
\left[\int_{\widetilde{Q}}\left|\vec{f}(x)\right|^p\,dx\right]^{\frac{q}{p}}.\nonumber
\end{align}
Using this, \eqref{eq-Q-Q}, and the embedding
$\ell^p\hookrightarrow\ell^q$ when $p\leq q$, we conclude that
\begin{align*}
&\int_{\mathbb{R}^n}\left[M_{W,\alpha,\frac{r}{2}}\vec{f}(x)\right]^q\,dx\\
&\quad=\sum_{Q\in\Gamma}\int_{Q}\left[M_{W,\alpha,\frac{r}{2}}\vec{f}(x)\right]^q\,dx
\lesssim[W]^{p^{\prime}(1-\frac{\alpha}{n})}_{
\mathscr{A}^{\operatorname{loc}}_{p,q}(r)}
\sum_{Q\in\Gamma}\left[\int_{\widetilde{Q}}
\left|\vec{f}(x)\right|^p\,dx\right]^{\frac{q}{p}}\\
&\quad\leq[W]^{p^{\prime}(1-\frac{\alpha}{n})}_{
\mathscr{A}^{\operatorname{loc}}_{p,q}(r)}
\left[\sum_{Q\in\Gamma}\int_{\widetilde{Q}}
\left|\vec{f}(x)\right|^p\,dx\right]^{\frac{q}{p}}
\lesssim[W]^{p^{\prime}(1-\frac{\alpha}{n})}_{
\mathscr{A}^{\operatorname{loc}}_{p,q}(r)}
\left[\int_{\mathbb{R}^n}\left|\vec{f}(x)\right|^p\,dx\right]^{\frac{q}{p}},
\end{align*}
where the last inequality follows from the fact that all
the cubes $\{\widetilde{Q}\}_{Q\in\Gamma}$ have bounded overlap.
Taking the $q$-th root on both sides, we obtain
\begin{align}\label{eq-bound-I}
\left\|M_{W,\alpha,\frac{r}{2}}\right\|_{L^p\to L^q}\lesssim
[W]^{\frac{p^{\prime}}{q}(1-\frac{\alpha}{n})}
_{\mathscr{A}^{\operatorname{loc}}_{p,q}(r)},
\end{align}
which completes the proof of the boundedness of $M_{W,\alpha,\frac{r}{2}}$.

Finally, we show the boundedness of $\widetilde{M}_{W,\alpha,r}$.
To this end, for any $\varepsilon\in(0,r]$,
let $\mathscr{R}$ be the set of all cubes
of length $r+\varepsilon$ with lower left-corners
at $\varepsilon\mathbb{Z}^n$.
Then, for any cube $Q$ of edge length $r$,
there exists $R\in\mathscr{R}$ containing $Q$, and
each $x\in\mathbb{R}^n$ is contained in at most
$N:=(\lceil1+r/\varepsilon\rceil+1)^n$ cubes in $\mathscr{R}$.
Next, we split $\mathscr{R}$ into $N$
subcollections $\mathscr R_j$, each of which
consists of pairwise disjoint cubes.
Applying these observations and
\eqref{eq-def-1scalemaxope}, we obtain, for any $x\in\mathbb{R}^n$,
\begin{align*}
\widetilde M_{W,\alpha,r}\vec{f}(x)
&\lesssim\sup_{Q\in\mathscr{R}}\frac{\mathbf{1}_{Q}(x)}{|Q|^{1-\alpha/n}}
\int_{Q}\left|W^{\frac1q}(x)W^{-\frac1q}(y)\vec{f}(y)\right|\,dy \\
&\leq\left(\sum_{j=1}^{N}\left[\sum_{Q\in\mathscr R_j} \frac{\mathbf{1}_{Q}(x)}{|Q|^{1-\alpha/n}}
\int_{Q}\cdots\right]^q\right)^{\frac 1q}
=:\left(\sum_{j=1}^{N}\left[M_{W,\alpha,r;j}
\vec{f}(x)\right]^q\right)^{\frac 1q}.
\end{align*}
This, together with H\"{o}lder's inequality, further implies that
\begin{align}\label{eq-est-1scalemax}
\left\|\widetilde M_{W,\alpha,r}\vec{f}\right\|_{L^q}
&\lesssim\left(\sum_{j=1}^{N}\left\|M_{W,\alpha,r;j}
\vec{f}\right\|_{L^q}^q\right)^{\frac1q}
\leq\left(\sum_{j=1}^{N} 1\right)^{\frac1q}
\max_{j\in\{1,\dots,N\}}
\left\|M_{W,\alpha,r;j}\vec{f}\right\|_{L^q}\\
&\sim\left(1+\frac{r}{\varepsilon}\right)^{\frac nq}
\max_{j\in\{1,\dots,N\}}
\left\|M_{W,\alpha,r;j}\vec{f}\right\|_{L^q}.\nonumber
\end{align}
Since, for any $j\in\{1,\dots,N\}$, the cubes in
$\mathscr{R}_j$ are pairwise disjoint,
by H\"{o}lder's inequality and the definition
\eqref{eq-locApq} of matrix weight constants,
we find that
\begin{align*}
\left\|M_{W,\alpha,r;j}\vec{f}\right\|_{L^q}^q
&=\sum_{Q\in\mathscr R_j}\int_{Q}
\left[|Q|^{\alpha/n}\fint_{Q}
\left|W^{\frac1q}(x)W^{-\frac1q}(y)\vec{f}(y)\right|\,dy\right]^q\,dx \\
&\leq\sum_{Q\in\mathscr R_j}|Q|^{\frac{\alpha}{n}q}
\int_{Q}\left[\fint_{Q}
\left\|W^{\frac1q}(x)W^{-\frac 1q}(y)\right\|^{p'}\,dy
\right]^{\frac{q}{p^{\prime}}}
\left[\fint_{Q}\left|\vec{f}(z)\right|^p\, dz\right]^{\frac{q}{p}}\,dx \\
&=\sum_{Q\in\mathscr R_j}|Q|^{\frac{\alpha}{n}q+1}
\fint_{Q}\left[\fint_{Q}
\left\|W^{\frac1q}(x)W^{-\frac 1q}(y)\right\|^{p'}\,dy
\right]^{\frac{q}{p^{\prime}}}\,dx
\left[\fint_{Q}\left|\vec{f}(z)\right|^p\,dz\right]^{\frac{q}{p}} \\
&\leq[W]_{\mathscr{A}_{p,q}^{\operatorname{loc}}(r+\varepsilon)}
\sum_{Q\in\mathscr R_j}
\left[\int_{Q}\left|\vec{f}(z)\right|^p\,dz\right]^{\frac{q}{p}},
\end{align*}
where in the last inequality we used the identity
\begin{align*}
\frac{\alpha}{n}q+1=\frac{q}{p}
\end{align*}
to eliminate the power of $|Q|$ against the average over $Q$ in the last step.

From the embedding $\ell^p\hookrightarrow\ell^q$
when $p\leq q$, we deduce that
\begin{align*}
\left\|M_{W,\alpha,r;j}\vec{f}\right\|_{L^q}^q
\leq [W]_{\mathscr{A}_{p,q}^{\operatorname{loc}}(r+\varepsilon)}
\left(\sum_{Q\in\mathscr R_j}\int_{Q}
\left|\vec{f}(z)\right|^p\,dz\right)^{\frac qp}
\leq [W]_{\mathscr{A}_{p,q}^{\operatorname{loc}}(r+\varepsilon)}
\left\|\vec{f}\right\|_{L^p}^q.
\end{align*}
This, together with \eqref{eq-est-1scalemax}, further implies that
\begin{align}\label{eq-bound-II}
\left\|\widetilde M_{W,\alpha,r}\right\|_{L^p\to L^q}
\lesssim\left(1+\frac{r}{\varepsilon}\right)^{\frac nq}
[W]_{\mathscr{A}_{p,q}^{\operatorname{loc}}(r+\varepsilon)}^{\frac 1q}.
\end{align}
Combining \eqref{eq-bound-I} and \eqref{eq-bound-II},
we obtain
\begin{align}\label{eq:2terms}
\left\|M_{W,\alpha,r}\right\|_{L^p\to L^q}
\lesssim[W]_{\mathscr{A}^{\operatorname{loc}}_{p,q}(r)}^{\frac{p'}{q}
(1-\frac{\alpha}{n})}+\left(1+\frac{r}{\varepsilon}\right)^{\frac nq}[W]_{\mathscr{A}_{p,q}^{\operatorname{loc}}(r+\varepsilon)}^{\frac 1q},
\end{align}
which is exactly the desired estimate \eqref{eq-bound-maxloc}.

By \eqref{eq:2terms} with $\varepsilon:=\frac{r}{2}$,
Theorem \ref{thm:Apq(1+eta)}
with $\eta:=\frac{1}{2}$, and the fact that
$[W]_{\mathscr{A}^{\operatorname{loc}}_{p,q}(r)}\geq 1$
established in Corollary \ref{coro-[W]geq1}, we find that
\begin{align*}
\left\|M_{W,\alpha,r}\right\|_{L^p\to L^q}\lesssim
[W]^{\frac{p^{\prime}}{q}(1-\frac{\alpha}{n})}
_{\mathscr{A}^{\operatorname{loc}}_{p,q}(r)}+[W]^{\frac{2}{q}}
_{\mathscr{A}^{\operatorname{loc}}_{p,q}(r)}\lesssim
[W]^{\max\{\frac{p^{\prime}}{q}(1-\frac{\alpha}{n}), \frac{2}{q}\}}
_{\mathscr{A}^{\operatorname{loc}}_{p,q}(r)}.
\end{align*}
Thus, \eqref{eq-bound-maxloc-2} holds.
Note that $\frac1q=\frac1p-\frac{\alpha}{n}$
and hence
\begin{align*}
p'\left(1-\frac{\alpha}{n}\right)=p'\left(\frac1q
+\frac{1}{p'}\right)=1+\frac{p'}{q}.
\end{align*}
In particular, if $q\leq p'$, namely
$1+\frac{p'}{q}\geq 2$, it immediately follows from
\eqref{eq-bound-maxloc-2} that \eqref{eq-bound-maxloc-1} holds.
This completes the proof of \eqref{eq-bound-maxloc-1}
and hence Theorem \ref{thm-bound-maxloc}.
\end{proof}

To show that the exponent of the weight constant
in \eqref{eq-bound-maxloc-1} of Theorem \ref{thm-bound-maxloc}
is sharp, we need to introduce the scalar local fractional
maximal operator $M_{\alpha,r}$.
Let $r\in(0,\infty)$ and $\alpha\in[0,n)$.
The \emph{scalar local fractional maximal operator}
$M_{\alpha,r}$ is defined by setting,
for any measurable function $f$
on $\mathbb{R}^n$ and for any $x\in\mathbb{R}^n$,
\begin{align}\label{eq-def-locfracmax}
M_{\alpha,r} f(x):=\sup_{
\genfrac{}{}{0pt}{}{Q \ni x}{\ell(Q)\leq r}}
\frac{1}{|Q|^{1-\frac{\alpha}{n}}}
\int_Q\left|f(y)\right|\, dy,
\end{align}
where the supremum is taken over all cubes $Q\subset\mathbb{R}^n$
with $\ell(Q)\leq r$ that contain $x$.
In the next proposition, via a scaling argument, we show
that the exponent of the weight constant
in \eqref{eq-bound-maxloc-1} cannot be improved to be strictly smaller
than the corresponding one in \eqref{eq-quantativebound-fracmax}
and hence is sharp.

\begin{proposition}\label{prop-sharpness-fractmax}
Suppose that $\alpha\in[0,n)$,
$p\in(1,\frac{n}{\alpha})$,
and $q\in(1,\infty)$ satisfy $\frac{1}{q}=\frac{1}{p}-\frac{\alpha}{n}$.
If there exist $r,s\in(0,\infty)$ and a constant $C\in (0,\infty)$ such that,
for any $w\in A^{\operatorname{loc}}_{p,q}(r)$ and
$f\in L^p(w^\frac{p}{q})$,
\begin{align}\label{eq-exponent-s}
\left\|M_{\alpha,r}f\right\|_{L^q(w)}
\le C[w]^s_{A^{\operatorname{loc}}_{p,q}(r)}
\left\|f\right\|_{L^p(w^\frac{p}{q})},
\end{align}
then $s\geq (1-\frac{\alpha}{n})\frac{p'}{q}$.
\end{proposition}

\begin{proof}
Suppose that $r,s\in(0,\infty)$ satisfy the assumption of
the present proposition.
Let $w \in A_{p,q}$ and $f \in L^p(w^\frac{p}{q})$. For any
$\lambda\in(0,\infty)$ and $x\in\mathbb{R}^n$, let
\begin{align}\label{eq-flambda}
f_\lambda(x) := f(\lambda x)\text{ and }w_\lambda(x) := w(\lambda x).
\end{align}
By the definition of $M_{\alpha,r}$ and
a change of variables, we find that,
for any $\lambda\in(0,\infty)$ and $x\in\mathbb{R}^n$,
\begin{align*}
M_{\alpha, r} f_\lambda(x)
&= \sup_{\genfrac{}{}{0pt}{}{Q \ni x}{\ell(Q) \le r}}
\frac{1}{|Q|^{1-\frac{\alpha}{n}}} \int_Q |f(\lambda y)| \, dy \\
&= \lambda^{-\alpha}
\sup_{\genfrac{}{}{0pt}{}{\lambda Q \ni \lambda x}{\ell(\lambda Q) \le \lambda r}}
\frac{1}{|\lambda Q|^{1-\frac{\alpha}{n}}} \int_{\lambda Q} |f(y)| \, dy
= \lambda^{-\alpha} M_{\alpha, \lambda r} f(\lambda x).
\end{align*}
Applying this dilation equality and a change of variables,
we obtain, for any $\lambda\in(0,\infty)$,
\begin{align}\label{eq-Lq(w)}
\left\|M_{\alpha, \lambda r} f\right\|_{L^q(w)}
&=\lambda^{\frac{n}{q}}\left\{\int_{\mathbb{R}^n}
\left[M_{\alpha, \lambda r}
f(\lambda x)\right]^q w(\lambda x)\,dx\right\}^{\frac{1}{q}}\\
&=\lambda^{\frac{n}{q}+\alpha}\left\{\int_{\mathbb{R}^n}\left[M_{\alpha, r}
f_\lambda(x)\right]^q w_{\lambda}(x)\,dx\right\}^{\frac{1}{q}}\nonumber\\
&=\lambda^{\frac{n}{q}+\alpha}\left\|M_{\alpha, r}
f_\lambda\right\|_{L^q(w_\lambda)}.\nonumber
\end{align}
Observe that, for any $\lambda\in(0,\infty)$, a change of variables also yields
\begin{align*}
[w_\lambda]_{A_{p,q}^{\operatorname{loc}}(r)}=
[w]_{A_{p,q}^{\operatorname{loc}}(\lambda r)}
\leq[w]_{A_{p,q}}<\infty.
\end{align*}
Using this observation, \eqref{eq-Lq(w)}, and the assumption \eqref{eq-exponent-s},
we conclude that
\begin{align}\label{eq-M}
\left\|M_{\alpha, \lambda r} f\right\|_{L^q(w)}
&=\lambda^{\frac{n}{q}+\alpha}\left\|M_{\alpha, r} f_\lambda\right\|_{L^q(w_\lambda)}\lesssim\lambda^{\frac{n}{q}+\alpha}
[w_\lambda]^s_{A_{p,q}^{\operatorname{loc}}(r)}
\left\|f_\lambda\right\|_{L^p(w_\lambda^\frac{p}{q})}\\
&=\lambda^{\frac{n}{q}+\alpha}
[w]^s_{A_{p,q}^{\operatorname{loc}}(\lambda r)}
\left\|f_\lambda\right\|_{L^p(w_\lambda^\frac{p}{q})},\nonumber
\end{align}
where the implicit positive constant is independent of
$f$, $w$, and $\lambda$.
By the scaling property of $L^p$, we find that, for any $\lambda\in(0,\infty)$,
\begin{align*}
\left\|f_\lambda\right\|_{L^p(w_\lambda^\frac{p}{q})}= \lambda^{-\frac{n}{p}}
\left\|f\right\|_{L^p(w^\frac{p}{q})}.
\end{align*}
Applying this, \eqref{eq-M}, and the assumption that
$\frac{1}{q} = \frac{1}{p} - \frac{\alpha}{n}$, we obtain,
for any $\lambda\in(0,\infty)$,
\begin{align}\label{eq-G}
\left\|M_{\alpha, \lambda r}f\right\|_{L^q(w)}
\lesssim[w]^s_{A_{p,q}^{\operatorname{loc}}(\lambda r)}
\left\|f\right\|_{L^p(w^\frac{p}{q})}.
\end{align}
Note that, as $\lambda\to\infty$,
$M_{\alpha, \lambda r} f$
increases pointwise to $M_{\alpha}f$ and
$[w]_{A_{p,q}^{\operatorname{loc}}(\lambda r)}$
increases to $[w]_{A_{p,q}}$. Using this and the
monotone convergence theorem and letting $\lambda\to\infty$
in \eqref{eq-G}, we conclude that
\begin{align*}
\|M_\alpha f\|_{L^q(w)} \lesssim [w]_{A_{p,q}}^s \|f\|_{L^p(w^\frac{p}{q})},
\end{align*}
which, together with the arbitrariness of $w \in A_{p,q}$ and
$f\in L^p(w^\frac{p}{q})$, further implies that
\begin{align*}
\left\|M_\alpha\right\|_{L^p(w^\frac{p}{q})\to L^q(w)}
\lesssim [w]^{s}_{A_{p,q}},
\end{align*}
where the implicit positive constant is independent of $w$.
It follows from the sharpness of \eqref{eq-quantativebound-fracmax}
that $s\geq (1-\frac{\alpha}{n})\frac{p'}{q}$.
This completes the proof of Proposition \ref{prop-sharpness-fractmax}.
\end{proof}

\begin{remark}\label{rmk-max}
Let all the assumptions be as in Theorem \ref{thm-bound-maxloc}.
\begin{itemize}
\item[{\rm (i)}] The exponent of
$[W]_{\mathscr{A}^{\operatorname{loc}}_{p,q}(r)}$
in \eqref{eq-bound-maxloc-1} is sharp.
Indeed, if this exponent could be improved to
a strictly smaller one, then \eqref{eq-bound-maxloc-1}
would imply that the corresponding inequality
in the scalar-valued setting also holds
with this smaller exponent, which is
impossible by Proposition \ref{prop-sharpness-fractmax}.

However, to obtain \eqref{eq-bound-maxloc-1}, we need an
additional assumption that $q\leq p'$. In the case
$\alpha=0$, this assumption is equivalent to $p\in(1,2]$.
It remains unknown whether the
inequality \eqref{eq-bound-maxloc-1} holds
without the additional assumption $q\leq p'$.

\item[{\rm (ii)}]
It is worth pointing out that, in the particular case $n=m=1$ and $\alpha=0$
(the one-dimensional and scalar-valued setting), \eqref{eq-bound-maxloc-1}
can be established for all $p \in (1,\infty)$ (see Theorem \ref{thm:strong}).
In this setting, the weights and local maximal operators
can be separated. Thus, we can use the interpolation methods.
Furthermore, the proof of Theorem \ref{thm:strong}
also relies heavily on the total ordering of $\mathbb{R}$.
Since these properties do not hold for matrix weights
in $\mathbb{R}^n$ with $n>1$,
the proof of Theorem \ref{thm:strong} seems inapplicable to
the general case of \eqref{eq-bound-maxloc-1}.

\item[{\rm (iii)}] If $m=1$ and $\alpha=0$
(the local maximal operator in the scalar-valued setting),
Theorem \ref{thm-bound-maxloc} also improves \cite[Theorem 2.3.7]{ooi24}.
In \cite[Theorem 2.3.7]{ooi24}, Ooi used the centered local maximal function
to obtain the weighted boundedness of the uncentered local maximal function,
yielding a quantitative dependence on
$[w]^{\frac{1}{p-1}}_{A_{p}^{\mathrm{loc}}(6r)}$.
This bound is obviously larger than the one
given in \eqref{eq-bound-maxloc}.
\end{itemize}
\end{remark}

Borrowing some ideas from the proofs of
\cite[Proposition 3.1 and Corollary 3.2]{im19},
we finally give a characterization of the class
$\mathscr{A}^{\operatorname{loc}}_{p,q}(r)$ of
local Muckenhoupt matrix weights
via matrix-weighted maximal operator $M_{W,\alpha,r}$. For this purpose,
we recall the following averaging operators. Let $\alpha\in[0,\infty)$.
For any cube $Q\subset\mathbb{R}^n$,
the \emph{averaging operator} $\operatorname{Aver}^{(\alpha)}_Q$
is defined by setting, for any integrable vector-valued function
$\vec{f}$ on $Q$ and for any $x\in\mathbb{R}^n$,
\begin{align*}
\operatorname{Aver}^{(\alpha)}_Q\vec{f}(x):=
\frac{1}{|Q|^{1-\frac{\alpha}{n}}}
\int_Q\vec{f}(y)\,dy\mathbf{1}_Q(x).
\end{align*}

\begin{corollary}\label{cor-Apq-Mloc}
Let $r\in(0,\infty)$ and $W$ be a matrix weight.
Suppose that $\alpha\in[0,n)$, $p\in(1,\frac{n}{\alpha})$,
and $q\in(1,\infty)$ satisfy $\frac{1}{q}=\frac{1}{p}-\frac{\alpha}{n}$.
Then the following three statements are mutually equivalent.
\begin{itemize}
\item[{\rm (i)}] $W\in\mathscr{A}^{\operatorname{loc}}_{p,q}(r)$.
\item[{\rm (ii)}] $M_{W,\alpha,r}$ is bounded from $L^p$ to $L^q$.
\item[{\rm (iii)}]
The operators $\operatorname{Aver}^{(\alpha)}_Q$,
for cubes $Q\subset\mathbb{R}^n$ with $\ell(Q)\leq r$,
are bounded from $L^p(W^{\frac{p}{q}})$ to $L^q(W)$
uniformly on $Q$.
\end{itemize}
\end{corollary}
\begin{proof}
The implication (i) $\Longrightarrow$ (ii) immediately follows
from Theorem \ref{thm-bound-maxloc}.
The implications (ii) $\Longrightarrow$ (iii)
and (iii) $\Longrightarrow$ (i) can be proved by
repeating the arguments used in the proofs of
\cite[Proposition 3.1 and Corollary 3.2]{im19}
restricted to cubes of edge length not greater than $r$; we omit the details.
This completes the proof of Corollary \ref{cor-Apq-Mloc}.
\end{proof}

\subsection{Local Fractional Integral Operators}\label{s3-2}

In this subsection, we establish the boundedness of local
fractional integral operators on local matrix-weighted Lebesgue spaces.
Let $\alpha\in(0,n)$ and $r\in(0,\infty)$.
The \emph{local fractional integral operator}
$I_{\alpha, r}$ is defined by setting, for any bounded function $f$
on $\mathbb{R}^n$ with compact support and for any $x\in\mathbb{R}^n$,
\begin{align}\label{eq-def-fracintloc}
I_{\alpha, r}f(x)=\int_{Q(x,r)} \frac{f(y)}{|x-y|^{n-\alpha}} \,dy,
\end{align}
where $Q(x,r)$ denotes the cube in $\mathbb{R}^n$ centered at $x$
and the edge length $2r$. To establish the boundedness of
$I_{\alpha, r}$, we first recall the boundedness of $I_\alpha$ on
matrix-weighted Lebesgue spaces, which was established
by Isralowitz and Moen in \cite[Theorem 1.4]{im19}.

\begin{theorem}\label{thm-bound-fracint}
Suppose that $\alpha\in(0,n)$, $p\in(1,\frac{n}{\alpha})$,
and $q\in(1,\infty)$ satisfy $\frac{1}{q}=
\frac{1}{p}-\frac{\alpha}{n}$.
Then, for any $W\in\mathscr{A}_{p,q}$,
\begin{align*}
\left\|I_\alpha\right\|_{L^p(W^{\frac{p}{q}}) \rightarrow L^q(W)}
\lesssim[W]_{\mathscr{A}_{p,q}}^{\frac{p^{\prime}}{q}(1-\frac{\alpha}{n})
+\frac{1}{q^{\prime}}},
\end{align*}
where the implicit positive constant is independent of $W$
and $I_\alpha$ is the same as in \eqref{eq-fract}.
\end{theorem}
Inspired by the proof of Theorem \ref{thm-bound-fracint}
in \cite{im19}, we obtain the following boundedness of $I_{\alpha, r}$.
\begin{theorem}\label{thm-bound-fracintloc}
Let $r\in(0,\infty)$. Suppose that $\alpha\in(0,n)$, $p\in(1,\frac{n}{\alpha})$,
and $q\in(1,\infty)$ satisfy $\frac{1}{q}=
\frac{1}{p}-\frac{\alpha}{n}$.
Then, for any $W\in\mathscr{A}^{{\rm loc}}_{p,q}(r)$,
\begin{align}\label{eq-bound-fracintloc}
\left\|I_{\alpha,r}\right\|_{L^p(W^{\frac{p}{q}}) \rightarrow L^q(W)}
\lesssim[W]_{\mathscr{A}^{{\rm loc}}_{p,q}(r)
}^{\frac{p^{\prime}}{q}(1-\frac{\alpha}{n})+\frac{1}{q^{\prime}}},
\end{align}
where the implicit positive constant is independent of
$r$ and $W$.
\end{theorem}
\begin{proof}
It was proved in \cite[Proposition 3.6]{cmr16} that
the set of bounded functions on $\mathbb{R}^n$
with compact support is dense in $L^p(W)$.
Thus, in what follows, we may assume that
$\vec{f}$ is a bounded function on $\mathbb{R}^n$
with compact support. Observe that, for any $x\in\mathbb{R}^n$,
\begin{align}\label{eq-fracint-I-II}
W^{\frac{1}{q}}(x)I_{\alpha, r}\left(W^{-\frac{1}{q}}\vec{f}\right)(x)
&=\int_{Q(x,r)} \frac{W^{\frac{1}{q}}(x)W^{-\frac{1}{q}}(y)
\vec{f}(y)}{|x-y|^{n-\alpha}} \,dy\\
&=\int_{Q(x,\frac{r}{2})}\frac{W^{\frac{1}{q}}(x)W^{-\frac{1}{q}}(y)
\vec{f}(y)}{|x-y|^{n-\alpha}} \,dy+\int_{Q(x,r)\setminus Q(x,\frac{r}{2})}\cdots\nonumber\\
&=:\operatorname{I}(x)+\operatorname{II}(x).\nonumber
\end{align}
Thus, to show \eqref{eq-bound-fracintloc}, it suffices
to estimate the $L^q$-norms of
$\operatorname{I}$ and $\operatorname{II}$.

We first deal with $\operatorname{II}$ in \eqref{eq-fracint-I-II}.
Since, for any $x\in\mathbb{R}^n$ and $y\in Q(x,r)\setminus Q(x,\frac{r}{2})$,
$|x-y|\gtrsim r$ and $Q(x,r)\setminus Q(x,\frac{r}{2})$ can be covered
by $2^n$ cubes containing $x$ with edge length $r$, it follows that
\begin{align*}
\left|\operatorname{II}(x)\right|\lesssim M_{W,\alpha,r}\vec{f}(x).
\end{align*}
Using Theorem \ref{thm-bound-maxloc} with $\varepsilon:=\frac{r}{2}$
therein and using Theorem \ref{thm:Apq(1+eta)}, we conclude that
\begin{align}\label{eq-est-II-1}
\left\|\operatorname{II}\right\|_{L^q}\lesssim
\left\{[W]^{\frac{p^{\prime}}{q}(1-\frac{\alpha}{n})}
_{\mathscr{A}^{\operatorname{loc}}_{p,q}(r)}+
[W]^{\frac{2}{q}}_{\mathscr{A}^{\operatorname{loc}}_{p,q}(r)}\right\}
\left\|\vec{f}\right\|_{L^p}.
\end{align}
The assumption that $\frac1q=\frac1p-\frac{\alpha}{n}$ and
hence $q>p$ yield that
\begin{align*}
\frac{p^{\prime}}{q}\left(1-\frac{\alpha}{n}\right)
+\frac{1}{q^{\prime}}
=\frac{1}{q}\left(1+\frac{p'}{q}\right)+\frac{1}{q'}
>\frac{1}{q}\left(1+\frac{q'}{q}+\frac{q}{q'}\right)\geq\frac{3}{q}>
\frac{2}{q}.
\end{align*}
Applying this, \eqref{eq-est-II-1}, and the fact
$[W]_{\mathscr{A}^{\operatorname{loc}}_{p,q}(r)}\geq1$
in Corollary \ref{coro-[W]geq1}, we obtain
\begin{align}\label{eq-est-II}
\left\|\operatorname{II}\right\|_{L^q}\lesssim
[W]_{\mathscr{A}^{{\rm loc}}_{p,q}(r)
}^{\frac{p^{\prime}}{q}(1-\frac{\alpha}{n})+
\frac{1}{q^{\prime}}}\left\|\vec{f}\right\|_{L^p},
\end{align}
which is the desired estimate of $\operatorname{II}$.

Next, we estimate $\operatorname{I}$ in \eqref{eq-fracint-I-II}.
As in the proof of Theorem \ref{thm-bound-maxloc},
we divide $\mathbb{R}^n$ into a collection $\Gamma$ of
cubes with pairwise disjoint interiors and with edge length $\frac{r}{2}$.
Let $Q\in \Gamma$ and let $\widetilde{Q}$ be the cube with
edge length $\frac{3}{2}r$ and with the same center as $Q$.
Suppose that $\vec{g}$ is a bounded function supported in $Q$.
Since $Q(x,\frac{r}{2})\subset\widetilde Q$ for all $x\in Q$,
using the definition of local fractional integral operators
in \eqref{eq-def-fracintloc}, we conclude that
\begin{align}\label{eq-locfrac-1}
\left|\int_{Q}\left(\operatorname{I}(x),\vec{g}(x)\right)\,dx\right|
&=\left|\int_{Q}\left(W^{\frac{1}{q}}(x)I_{\alpha,\frac{r}{2}}
\left(W^{-\frac{1}{q}}\vec{f}\right)(x),\vec{g}(x)\right)\,dx\right|\\
&=\left|\int_{Q}\left(I_{\alpha,\frac{r}{2}}\left(W^{-\frac{1}{q}}
\vec{f}\right)(x),W^{\frac{1}{q}}(x)\vec{g}(x)\right)\,dx\right|\nonumber\\
&\leq\int_{Q}\int_{Q(x,\frac{r}{2})}\frac{|(W^{-\frac{1}{q}}(y)
\vec{f}(y),W^{\frac{1}{q}}(x)\vec{g}(x))|}{|x-y|^{n-\alpha}}\,dy\,dx\nonumber\\
&\leq\int_{Q}\int_{\widetilde{Q}}\frac{|(W^{-\frac{1}{q}}(y)
\vec{f}(y),W^{\frac{1}{q}}(x)\vec{g}(x))|}{|x-y|^{n-\alpha}}\,dy\,dx
=:\Omega.\nonumber
\end{align}
As in the proof of Theorem \ref{thm-bound-maxloc},
there exist cubes $\{Q_i\}_{i=1}^{2^n}$
with edge length $r$ and matrix weights
$\{W_i\}_{i=1}^{2^n}\subset\mathscr{A}_{p,q}$ such that $\widetilde{Q}=
\cup_{i=1}^{2^n} Q_i$, $Q\subset\cap_{i=1}^{2^n} Q_i$, and,
for any $i\in\{1,\dots,2^n\}$, $W_i=W$ on $Q_i$ and
\begin{align*}
\left[W_i\right]_{\mathscr{A}_{p,q}}\leq
3^{n(1+\frac{q}{p^{\prime}})}\left[W\right]_{
\mathscr{A}^{\operatorname{loc}}_{p,q}(r)}.
\end{align*}
Since $\{W_i\}_{i=1}^{2^n}\subset\mathscr{A}_{p,q}$,
to estimate $\Omega$ in \eqref{eq-locfrac-1}, applying this extension property
and the proof of \cite[Theorem 1.4]{im19}, we conclude that
\begin{align*}
\Omega&\leq\sum_{i=1}^{2^n}\int_{Q}\int_{Q_i}
\frac{|(W_i^{-\frac{1}{q}}(y)\vec{f}(y),
W_i^{\frac{1}{q}}(x)\vec{g}(x))|}{|x-y|^{n-\alpha}}\,dy\,dx\\
&=\sum_{i=1}^{2^n}\int_{\mathbb{R}^n}\int_{\mathbb{R}^n}
\frac{|(W_i^{-\frac{1}{q}}(y)\vec{f}(y)\mathbf{1}_{Q_i}(y),
W_i^{\frac{1}{q}}(x)\vec{g}(x)\mathbf{1}_{Q}
(x))|}{|x-y|^{n-\alpha}}\,dy\,dx\\
&\lesssim\sum_{i=1}^{2^n}[W_i]_{\mathscr{A}_{p,q}
}^{\frac{p^{\prime}}{q}(1-\frac{\alpha}{n})
+\frac{1}{q^{\prime}}}\left\|\vec{f}\mathbf{1}_{Q_i}\right\|_{L^p}
\left\|\vec{g}\mathbf{1}_Q\right\|_{L^{q^{\prime}}}\\
&\lesssim[W]_{\mathscr{A}^{\operatorname{loc}}_{p,q}(r)
}^{\frac{p^{\prime}}{q}(1-\frac{\alpha}{n})
+\frac{1}{q^{\prime}}}\left\|\vec{f}\mathbf{1}_{\widetilde{Q}}
\right\|_{L^p}\left\|\vec{g}\mathbf{1}_Q\right\|_{L^{q^{\prime}}}
\end{align*}
and hence
\begin{align*}
\left\|\operatorname{I}\mathbf{1}_{Q}\right\|_{L^q}
\lesssim[W]_{\mathscr{A}^{\operatorname{loc}}_{p,q}(r)
}^{\frac{p^{\prime}}{q}(1-\frac{\alpha}{n})
+\frac{1}{q^{\prime}}}\left\|\vec{f}\mathbf{1}_{\widetilde{Q}}
\right\|_{L^p}.
\end{align*}
Taking the $q$-th power on both sides
and summing over all $Q\in\Gamma$, we obtain
\begin{align}\label{eq-last-est}
\left\|\operatorname{I}\right\|^q_{L^q}
&=\int_{\mathbb{R}^n}\left|
\operatorname{I}(x)\right|^q\,dx
=\sum_{Q\in\Gamma}\int_{Q}\left|
\operatorname{I}(x)\right|^q\,dx
\lesssim[W]_{\mathscr{A}^{\operatorname{loc}}_{p,q}(r)
}^{p^{\prime}(1-\frac{\alpha}{n})+\frac{q}{q^{\prime}}}\sum_{Q\in\Gamma}
\left[\int_{\widetilde{Q}}\left|\vec{f}(x)\right|^p\,dx
\right]^{\frac{q}{p}}\\
&\lesssim[W]_{\mathscr{A}^{\operatorname{loc}}_{p,q}(r)
}^{p^{\prime}(1-\frac{\alpha}{n})+\frac{q}{q^{\prime}}}
\left[\int_{\mathbb{R}^n}\left|\vec{f}(x)\right|^p\,dx\right]^{\frac{q}{p}}
=[W]_{\mathscr{A}^{\operatorname{loc}}_{p,q}(r)
}^{p^{\prime}(1-\frac{\alpha}{n})+\frac{q}{q^{\prime}}}
\left\|\vec{f}\right\|^q_{L^p},\nonumber
\end{align}
where the last inequality follows from the embedding
$\ell^p\hookrightarrow\ell^q$ when $p\leq q$ and
the bounded overlap of the cubes $\{\widetilde Q\}_{Q\in\Gamma}$.
Taking the $q$-th root on both sides of \eqref{eq-last-est}, we obtain
\begin{align}\label{eq-est-I-1}
\left\|\operatorname{I}\right\|_{L^q}
\lesssim[W]_{\mathscr{A}^{{\rm loc}}_{p,q}(r)
}^{\frac{p^{\prime}}{q}(1-\frac{\alpha}{n})+\frac{1}{q^{\prime}}}
\left\|\vec{f}\right\|_{L^p}.
\end{align}
Combining \eqref{eq-fracint-I-II}, \eqref{eq-est-II},
and \eqref{eq-est-I-1}, we finally obtain
the desired estimate \eqref{eq-bound-fracintloc}.
This completes the proof of Theorem \ref{thm-bound-fracintloc}.
\end{proof}

\begin{remark}
Note that the exponent of the weight constant
$\frac{p^{\prime}}{q}(1-\frac{\alpha}{n})+\frac{1}{q^{\prime}}$
in \eqref{eq-bound-fracintloc} does not match the sharp exponent
in \eqref{eq-quantativebound-fracint} when reducing to the scalar-valued case.
However, as observed in the remark after \cite[Theorem 1.4]{im19},
by formally letting $\alpha=0$ and $p=q=2$,
this exponent becomes exactly $\frac{3}{2}$.
This coincides with the matrix-weighted $L^2$
bound of Calder\'{o}n--Zygmund operators
established in \cite{nptv}. Furthermore,
it has been proved in \cite{dptv24} that the exponent
$\frac{3}{2}$ is sharp, which is strictly greater
than the exponent $1$ in the scalar setting.
Therefore, it remains unknown whether
the bound in \eqref{eq-bound-fracintloc}
is sharp in the matrix-weighted setting.
\end{remark}

Finally, we present a corollary to illustrate that the exponent $\frac{p^{\prime}}{q}(1-\frac{\alpha}{n})+\frac{1}{q^{\prime}}$
in \eqref{eq-bound-fracintloc} arises directly from
the method used to prove Theorem \ref{thm-bound-fracint}.
In the scalar case, combining
the argument used in the proof of Theorem \ref{thm-bound-fracintloc} with
\eqref{eq-quantativebound-fracint}, we establish the following
quantitative weighted estimate of $I_{\alpha, r}$ whose
weight constant exponent coincides with the one
in \eqref{eq-quantativebound-fracint}.

\begin{corollary}\label{cor-bound-fracintloc-scalar}
Let $r\in(0,\infty)$. Suppose that $\alpha\in(0,n)$, $p\in(1,\frac{n}{\alpha})$,
and $q\in(1,\infty)$ satisfy $\frac{1}{q}=
\frac{1}{p}-\frac{\alpha}{n}$. 	Then, for any $w\in  A^{{\rm loc}}_{p,q}(r)$,
\begin{align}\label{eq-bound-fracintloc-scalar}
\left\|I_{\alpha,r}\right\|_{L^p(w^{\frac{p}{q}}) \rightarrow L^q(w)}
\lesssim[w]_{A^{{\rm loc}}_{p,q}(r)
}^{(1-\frac{\alpha}{n})\max(1,\frac{p'}{q})},
\end{align}
where the implicit positive constant is independent of
$r$ and $w$.
\end{corollary}
\begin{proof}
Since $I_{\alpha,r}$ is a positive operator, without loss of
generality, we may assume that $f$ is a non-negative bounded function
on $\mathbb{R}^n$ with compact support. Based on the idea used in
\eqref{eq-fracint-I-II}, for any $x\in\mathbb{R}^n$, we have
\begin{align}\label{eq-fracint-I-II-scalar}
I_{\alpha, r} f(x)
=\int_{Q(x,\frac{r}{2})}\frac{f(y)}{|x-y|^{n-\alpha}} \,dy
+\int_{Q(x,r)\setminus Q(x,\frac{r}{2})}\cdots
=:\operatorname{I}(x)+\operatorname{II}(x).
\end{align}
To estimate $\operatorname{II}$ in \eqref{eq-fracint-I-II-scalar},
as in \eqref{eq-est-II-1}, we use \eqref{eq-bound-maxloc}
with $\varepsilon:=\frac{r}{2}$ and Theorem \ref{thm:Apq(1+eta)}
to conclude that
\begin{align}\label{eq-II-scalar}
\left\|\operatorname{II}\right\|_{L^q(w)}\lesssim
\left\{[w]^{\frac{p^{\prime}}{q}(1-\frac{\alpha}{n})}
_{A^{\operatorname{loc}}_{p,q}(r)}+
[w]^{\frac{2}{q}}_{A^{\operatorname{loc}}_{p,q}(r)}\right\}
\left\|f\right\|_{L^p(w^{\frac{p}{q}})}.
\end{align}
The assumption that $\frac1q=\frac1p-\frac{\alpha}{n}$ and
hence $q>p$ imply that
\begin{align*}
\left(\frac1q
+\frac{1}{p'}\right)\max\left(1,\frac{p'}{q}\right)\geq\frac{2}{q}.
\end{align*}
This, together with \eqref{eq-II-scalar} and the known fact that
$[w]_{A^{\operatorname{loc}}_{p,q}(r)}\geq1$ (see Corollary \ref{coro-[W]geq1}),
further implies the desired estimate of $\operatorname{II}$ in \eqref{eq-fracint-I-II-scalar}.

Finally, to estimate $\operatorname{I}$ in \eqref{eq-fracint-I-II-scalar},
it suffices to repeat the argument used in the proof of Theorem \ref{thm-bound-fracintloc}
by replacing the application of Theorem \ref{thm-bound-fracint} therein
with the sharp bound \eqref{eq-quantativebound-fracint} here to directly
estimate $\|I_{\alpha,\frac{r}{2}}f\mathbf{1}_Q\|_{L^q(w)}$ for any
$Q\in\Gamma$; we omit the details.
This completes the proof of Corollary \ref{cor-bound-fracintloc-scalar}.
\end{proof}

\begin{remark}
Let the notation be the same as in Corollary \ref{cor-bound-fracintloc-scalar}.
We observe that the exponent of the weight constant in
\eqref{eq-bound-fracintloc-scalar} is sharp.
Indeed, suppose that $f$ is a non-negative function on
$\mathbb{R}^n$ with compact support. For any $\lambda\in(0,\infty)$,
let $f_\lambda$ be as in \eqref{eq-flambda}.
By the definition
\eqref{eq-def-fracintloc} of the local fractional integral
operator and a change of variables, we find that,
for any $\lambda\in(0,\infty)$ and $x\in\mathbb{R}^n$,
\begin{align*}
I_{\alpha, r} f_\lambda(x)
&= \int_{Q(x, r)} \frac{f(\lambda y)}{|x-y|^{n-\alpha}} \, dy
= \lambda^{-n}\int_{Q(\lambda x, \lambda r)}
\frac{f(y)}{|x - \lambda^{-1}y|^{n-\alpha}}  \, dy \\
&= \lambda^{-\alpha}\int_{Q(\lambda x, \lambda r)}
\frac{f(y)}{|\lambda x - y|^{n-\alpha}} \, dy
= \lambda^{-\alpha} I_{\alpha, \lambda r} f(\lambda x).
\end{align*}
Using this scaling equality and \eqref{eq-quantativebound-fracint}
and repeating the proof of
Proposition \ref{prop-sharpness-fractmax} with
$M_{\alpha, r}$ and $M_{\alpha}$ replaced, respectively, by
$I_{\alpha, r}$ and $I_{\alpha}$, we conclude that
the exponent of the weight constant in
\eqref{eq-bound-fracintloc-scalar} is sharp.
\end{remark}

\subsection{Local Haar Square Functions}\label{s3-3}

In this subsection, we aim to characterize the
$L^p(W)$-norm for any $p\in(1,\infty)$ and $W\in\mathscr{A}^{\operatorname{loc}}_{p}$
in terms of local Haar square functions.
To this end, we first recall the
concept of Haar systems. For any dyadic interval $I$ in $\mathbb{R}$, let
\begin{align}\label{eq-Haar-1}
h_I^{(1)}=|I|^{-\frac{1}{2}} \mathbf{1}_I \text{ and } h_I^{(0)}=|I|^{-\frac{1}{2}}\left(\mathbf{1}_{I_{l}}-\mathbf{1}_{I_r}\right),
\end{align}
where $I_{l}$ and $I_{r}$ denote, respectively,
the left and the right half parts of $I$.
Let $I:=I_1\times \cdots \times I_n$ be a dyadic cube
in $\mathbb{R}^n$ and $k:=(k_1,\dots,k_n)\in\{0,1\}^n$,
where $I_1,\dots,I_n$ are dyadic intervals in $\mathbb{R}$ .
For any $x:=(x_1,\dots,x_n)\in\mathbb{R}^n$, let
\begin{align}\label{eq-Haar-n}
h_I^{(k)}(x)=\prod_{i=1}^n h_{I_i}^{(k_i)}(x_i).
\end{align}
Let
$$\operatorname{Sig}_n := \{0,1\}^n \setminus \{(1,\dots,1)\}.$$
It is well known that the set
\[
\left\{h_I^{(k)} : I\in\mathcal{D},\ k\in\operatorname{Sig}_n\right\}
\]
forms an orthonormal basis of $L^2$.
\begin{remark}\label{rmk-Haar}
By the definition of the Haar systems \eqref{eq-Haar-1}
and \eqref{eq-Haar-n}, it is easy to verify that,
for any $I\in\mathcal{D}$ and $k\in\operatorname{Sig}_n$,
\begin{itemize}
\item[{\rm (i)}] $h_I^{(k)}$ is supported in $I$;
\item[{\rm (ii)}] $\int_I h_I^{(k)}(x)\,dx = 0$;
\item[{\rm (iii)}] $h_I^{(k)}$ is constant on each dyadic subcube of $I$.
\end{itemize}
\end{remark}

Let $\vec{f}:\mathbb{R}^n\longrightarrow \mathbb{C}^m$
be a locally integrable function. For any
$I\in\mathcal{D}$ and $k\in\operatorname{Sig}_n$,
its \emph{Haar coefficient} $\vec{f}_I^{(k)}$ is defined by setting
\begin{align}\label{eq-haarcof}
\vec{f}_I^{(k)} := \int_I \vec{f}(x)h_I^{(k)}(x)\, dx,
\end{align}
where $h_I^{(k)}$ is the same as in \eqref{eq-Haar-n}.
Here, and thereafter, for any $j\in\mathbb{Z}$, let
$\mathcal{D}_j:=\{Q\in\mathcal{D}:\ \ell(Q)=2^{-j}\}$
be the set of all \emph{dyadic cubes in $\mathbb{R}^n$
at the $j$-th level}.
Based on Haar coefficients, we next
recall the concept of Haar square functions.
Suppose that $p\in(1,\infty)$, $W$ is a matrix weight,
and $j\in\mathbb{Z}$.
For any locally integrable vector-valued function
$\vec{f}:\mathbb{R}^n\longrightarrow \mathbb{C}^m$,
its \emph{local Haar square function}  $S_{W,p,j}\vec{f}$ is defined by
setting, for any $x\in\mathbb{R}^n$,
\begin{align}\label{eq-def-locHaarSquare}
S_{W,p,j}\vec{f}(x):=
\left[\sum_{l=j}^{\infty}\sum_{I \in \mathcal{D}_l}\sum_{k\in\operatorname{Sig}_n}
\frac{|W^{\frac{1}{p}}(x)\vec{f}_I^{(k)}|^2}{|I|}
\mathbf{1}_I(x)\right]^{\frac{1}{2}}.
\end{align}

In \eqref{eq-def-locHaarSquare}, if we take the first summation
over all $l\in\mathbb{Z}$, we obtain the corresponding
well-known \emph{Haar square function} $S_{W,p}$.
The Haar square function $S_{W,p}$ and its variants
have been extensively investigated in the literature
(see, for example, \cite{DKPS,HPV,isr20,v97}).
Before presenting our main results, we first recall
the quantitative weighted norm estimates for $S_{W,p}$ in the global setting.
More precisely, the particular case $p=2$ of \eqref{eq-HAAR-right}
was established by Hyt\"onen et al. in \cite[Theorem 1]{HPV}.
Later, Isralowitz in \cite[Theorem 1]{isr20} proved
the quantitative $L^p(V)$ bounds of $S_{W,p}$ and
showed the sharpness of this bound in the range $p\in(1,2]$.
Moreover, \eqref{eq-HAAR-right} is exactly a
particular case of \cite[Theorem 1]{isr20}.
Using \eqref{eq-HAAR-right} and duality,
\eqref{eq-HAAR-left} was proved by
Domelevo et al. in \cite[Corollary 5.6]{DKPS}.

\begin{theorem}\label{thm-HAAR-Ap}
Let $p\in(1,\infty)$ and $W\in\mathscr{A}_p$.
Then, for any $\vec{f} \in L^p(W)$,
\begin{align}\label{eq-HAAR-right}
\left\|S_{W,p}\vec{f}\right\|_{L^p}
\lesssim\left[W\right]_{\mathscr{A}_p}^{\gamma(p)}
\left\|\vec{f}\right\|_{L^p(W)}
\end{align}
and
\begin{align}\label{eq-HAAR-left}
\left\|\vec{f}\right\|_{L^p(W)}\lesssim
\left[W\right]_{\mathscr{A}_p}^{\frac{\gamma(p')}{p-1}}
\left\|S_{W,p}\vec{f}\right\|_{L^p}
\end{align}
hold, where all the implicit positive constants
are independent of $\vec{f}$ and $W$ and
\begin{align}\label{eq-gamma-p}
\gamma(p):=
\begin{cases}
\displaystyle\frac{1}{p-1}
& \text{if } p\in(1,2],\\
\displaystyle\frac{1}{2}+\frac{1}{p(p-1)}
& \text{if } p\in(2,\infty).
\end{cases}
\end{align}
\end{theorem}

To establish the variant of Theorem \ref{thm-HAAR-Ap} for
local matrix weights, we also need the following concept.
For any $j\in\mathbb{Z}$ and any locally integrable function
$\vec{f}:\mathbb{R}^n\longrightarrow \mathbb{C}^m$, let
\begin{align}\label{eq-Ej}
E_j\vec{f}:=\sum_{Q \in \mathcal{D}_j}
\mathbf{1}_Q\fint_Q \vec{f}(x)\,dx.
\end{align}
The following theorem provides the local variant of
Theorem \ref{thm-HAAR-Ap}. Its proof follows a
localization argument based on extending
local matrix weights to global ones and
applying the global Haar square function estimates.

\begin{theorem}\label{thm-HAAR-Aploc}
Let $j\in\mathbb{Z}$, $p\in(1,\infty)$, and
$W\in\mathscr{A}^{\operatorname{loc}}_p(2^{-j})$.
Then, for any  $\vec{f} \in L^p(W)$,
\begin{align}\label{eq-HAAR-Aploc-2}
\left\|S_{W,p,j}\vec{f}\right\|_{L^p}
+\left\|E_j \vec{f}\right\|_{L^p(W)}
\lesssim[W]_{\mathscr{A}^{\operatorname{loc}}
_p(2^{-j})}^{\gamma(p)}
\left\|\vec{f}\right\|_{L^p(W)}
\end{align}
and
\begin{align}\label{eq-HAAR-Aploc-1}
\left\|\vec{f}\right\|_{L^p(W)}
\lesssim[W]_{\mathscr{A}^{\operatorname{loc}}
_p(2^{-j})}^{\frac{\gamma(p')}{p-1}}
\left[\left\|S_{W,p,j}\vec{f}\right\|_{L^p}
+\left\|E_j \vec{f}\right\|_{L^p(W)}\right]
\end{align}
hold, where $E_j$ is as in \eqref{eq-Ej},
$\gamma(p)$ is as in \eqref{eq-gamma-p},
and all the implicit positive constants
are independent of $\vec{f}$, $W$, and $j$.
\end{theorem}

\begin{proof}
We first prove \eqref{eq-HAAR-Aploc-1}.
Without loss of generality, we may assume that $\vec{f}$
is a bounded function on $\mathbb{R}^n$ with compact support
because the set of bounded functions on $\mathbb{R}^n$
with compact support is dense in $L^p(W)$
(see \cite[Proposition 3.6]{cmr16}).
Let $Q\in\mathcal{D}_j$. By (i) and (iii) of Remark \ref{rmk-Haar},
we find that the Haar coefficients of $\vec{f}\mathbf{1}_Q$ can be computed
as follows, for any $I\in\mathcal{D}$ and $k\in\operatorname{Sig}_n$,
\begin{align}\label{eq-coff-f1Q}
\left(\vec{f}\mathbf{1}_Q\right)_I^{(k)}=
\begin{cases}
\vec{f}_I^{(k)}
& \text{if } I\subset Q,\\
0
& \text{if } I\cap Q=\emptyset,\\
h_I^{(k)}(x_Q)\int_{Q}\vec{f}(x)\,dx
& \text{if } Q\subsetneqq I,
\end{cases}
\end{align}
where $x_Q$ is the lower-left corner of $Q$.
From Theorem \ref{thm-extension}, it follows that there
exists $V\in\mathscr{A}_p$ such that
\begin{align}\label{eq-V=W}
V=W\text{ on }Q\text{ and }[V]_{\mathscr{A}_p}
\leq  3^{np}[W]_{\mathscr{A}^{\operatorname{loc}}_p(2^{-j})}.
\end{align}
Applying Theorem \ref{thm-HAAR-Ap} to $\vec{f}\mathbf{1}_Q$ and $V$ and
applying \eqref{eq-coff-f1Q} and \eqref{eq-V=W}, we obtain
\begin{align}\label{eq-f1Q-1}
\left\|\vec{f}\mathbf{1}_Q\right\|^p_{L^p(W)}
&=\left\|\vec{f}\mathbf{1}_Q\right\|^p_{L^p(V)}\\
&\lesssim\left[V\right]_{\mathscr{A}_p}^{\frac{\gamma(p')p}{p-1}}
\int_{\mathbb{R}^n}\left[\sum_{
\genfrac{}{}{0pt}{}{I\in\mathcal{D}}{I\subset Q}}
\sum_{k\in\operatorname{Sig}_n}
\frac{|V^{\frac{1}{p}}(x)\vec{f}_I^{(k)}|^2}{|I|}
\mathbf{1}_I(x)\right]^{\frac{p}{2}}\,dx\nonumber\\
&\quad+\left[V\right]_{\mathscr{A}_p}^{\frac{\gamma(p')p}{p-1}}
\int_{\mathbb{R}^n}\left[\sum_{
\genfrac{}{}{0pt}{}{I\in\mathcal{D}}{Q  \subsetneqq I}}
\sum_{k\in\operatorname{Sig}_n}
\frac{|V^{\frac{1}{p}}(x) h_I^{(k)}(x_Q)\int_{Q}\vec{f}(y)\,dy|^2}{|I|}
\mathbf{1}_I(x)\right]^{\frac{p}{2}}\,dx\nonumber\\
&=\left[V\right]_{\mathscr{A}_p}^{\frac{\gamma(p')p}{p-1}}
\int_{Q}\left[S_{V,p,j}\vec{f}(x)\right]^p\,dx\nonumber\\
&\quad+\left[V\right]_{\mathscr{A}_p}^{\frac{\gamma(p')p}{p-1}}
\int_{\mathbb{R}^n}\left[\sum_{
\genfrac{}{}{0pt}{}{I\in\mathcal{D}}{Q  \subsetneqq I}}
\sum_{k\in\operatorname{Sig}_n} \frac{|V^{\frac{1}{p}}(x)
\int_{Q}\vec{f}(y)\,dy|^2}{|I|^2}
\mathbf{1}_I(x)\right]^{\frac{p}{2}}\,dx.\nonumber
\end{align}
Next, we estimate the
last integral in \eqref{eq-f1Q-1}.
To this end, for any $h\in\mathbb{Z}_+$,
let $I_h$ be  the unique dyadic cube in
$\mathcal D_{j-h}$ containing $Q$.
Since $\{I_h\}_{h\in\mathbb{Z}_+}$ is a nested family
and $|I_h|=2^{hn}|Q|$,
we deduce that, for any $h\in\mathbb{Z}_+$ and
$x\in I_{h}\setminus I_{h-1}$ (where $I_{-1}:=\emptyset$),
\begin{align}\label{eq-Sum}
&\sum_{\genfrac{}{}{0pt}{}{I\in\mathcal{D}}{Q  \subsetneqq I}}
\sum_{k\in\operatorname{Sig}_n}
\frac{|V^{\frac{1}{p}}(x)\int_{Q}\vec{f}(y)\,dy|^2}{|I|^2}
\mathbf{1}_I(x)\\
&\quad=(2^n-1)\left|V^{\frac{1}{p}}(x)
\int_{Q}\vec{f}(y)\,dy\right|^2
\sum_{l=h}^{\infty}\frac{1}{|I_l|^2}\nonumber\\
&\quad\sim\frac{2^n-1}{2^{2hn}|Q|^2}
\left|V^{\frac{1}{p}}(x)
\int_{Q}\vec{f}(y)\,dy\right|^2.\nonumber
\end{align}

Recall that the proof of Theorem \ref{thm-extension}
guarantees that the new matrix weight $V$ satisfies,
for any $R\in\mathcal{D}_j$ and $\vec{z}\in\mathbb{C}^m$,
\begin{align}\label{eq-int-R-Q}
\int_{R}\left|V^{\frac{1}{p}}(x)\vec{z}\right|^p\,dx
=\int_{Q}\left|W^{\frac{1}{p}}(x)\vec{z}\right|^p\,dx.
\end{align}
By \eqref{eq-Sum} and \eqref{eq-int-R-Q}, we find that the
last integral in \eqref{eq-f1Q-1} can be estimated as follows
\begin{align*}
&\int_{\mathbb{R}^n}\left[\sum_{\genfrac{}{}{0pt}{}{I\in\mathcal{D}}{Q\subsetneqq I}}
\sum_{k\in\operatorname{Sig}_n} \frac{|V^{\frac{1}{p}}(x)
\int_{Q}\vec{f}(y)\,dy|^2}{|I|^2}
\mathbf{1}_I(x)\right]^{\frac{p}{2}}\,dx\\
&\quad=\sum_{h\in\mathbb{Z}_+}\int_{I_h\setminus I_{h-1}}
\left[\sum_{\genfrac{}{}{0pt}{}{I\in\mathcal{D}}{Q\subsetneqq I}}
\sum_{k\in\operatorname{Sig}_n} \frac{|V^{\frac{1}{p}}(x)
\int_{Q}\vec{f}(y)\,dy|^2}{|I|^2}
\mathbf{1}_I(x)\right]^{\frac{p}{2}}\,dx\\
&\quad\sim\sum_{h\in\mathbb{Z}_+}\left(\frac{2^n-1}{2^{2hn}|Q|^2}
\right)^{\frac{p}{2}}\int_{I_h\setminus I_{h-1}}
\left|V^{\frac{1}{p}}(x)
\int_{Q}\vec{f}(y)\,dy\right|^p\,dx\\
&\quad\sim\sum_{h\in\mathbb{Z}_+}
\left(\frac{1}{2^{2hn}|Q|^2}\right)^{\frac p2}
\frac{|I_h\setminus I_{h-1}|}{|Q|}\int_{Q}
\left|W^{\frac{1}{p}}(x)\int_{Q}\vec{f}(y)\,dy\right|^p\,dx\\
&\quad\sim\int_{Q}
\left|W^{\frac{1}{p}}(x)\fint_{Q}\vec{f}(y)\,dy\right|^p\,dx
=\int_{Q}\left|W^{\frac{1}{p}}(x)E_j\vec{f}(x)\right|^p\,dx.
\end{align*}
This, together with \eqref{eq-f1Q-1} and \eqref{eq-V=W},
further implies that $S_{W,p,j}
\vec{f}=S_{V,p,j}\vec{f}$ on $Q$ and
\begin{align*}
\left\|\vec{f}\mathbf{1}_Q\right\|^p_{L^p(W)}
\lesssim[W]_{\mathscr{A}^{\operatorname{loc}}
_p(2^{-j})}^{\frac{\gamma(p')p}{p-1}}\left[\int_{Q}\left[S_{W,p,j}
\vec{f}(x)\right]^p\,dx+\int_{Q}\left|W^{\frac{1}{p}}(x)
E_j\vec{f}(x)\right|^p\,dx\right].
\end{align*}
Summing over all dyadic cubes in $\mathcal{D}_j$ on both sides,
we conclude that
\begin{align*}
\left\|\vec{f}\right\|^p_{L^p(W)}
&=\sum_{Q\in\mathcal{D}_j}\left\|\vec{f}\mathbf{1}_Q\right\|^p_{L^p(W)}\\
&\lesssim[W]^{\frac{\gamma(p')p}{p-1}}_{\mathscr{A}^{\operatorname{loc}}_p(2^{-j})}
\left\{\int_{\mathbb{R}^n}\left[S_{W,p,j}\vec{f}(x)\right]^p\,dx+
\int_{\mathbb{R}^n}\left|W^{\frac{1}{p}}(x)
E_j\vec{f}(x)\right|^p\,dx\right\}
\end{align*}
and hence \eqref{eq-HAAR-Aploc-1} holds.

Finally, we prove \eqref{eq-HAAR-Aploc-2}.
Let $Q\in\mathcal{D}_j$ and $V\in\mathscr{A}_p$
be the same as in \eqref{eq-V=W}.
From \eqref{eq-coff-f1Q} and the definition of the local Haar square function
in \eqref{eq-def-locHaarSquare},
we deduce that, for any $x\in Q$,
\begin{align*}
\left[S_{W,p,j}\vec{f}(x)\right]^2&=
\sum_{\genfrac{}{}{0pt}{}{I\in\mathcal{D}}{I\subset Q}}
\sum_{k\in\operatorname{Sig}_n}
\frac{|W^{\frac{1}{p}}(x)\vec{f}_I^{(k)}|^2}{|I|}
\mathbf{1}_I(x)\\
&\leq\sum_{\genfrac{}{}{0pt}{}{I\in\mathcal{D}}{I\subset Q}}\sum_{k\in\operatorname{Sig}_n}
\frac{|W^{\frac{1}{p}}(x)(\vec{f}\mathbf{1}_Q)_I^{(k)}|^2}{|I|} \mathbf{1}_I(x)
\leq\left[S_{V,p}\left(\vec{f}\mathbf{1}_Q\right)(x)\right]^2.
\end{align*}
From this, \eqref{eq-HAAR-right}, and \eqref{eq-V=W}, we deduce that
\begin{align*}
\int_{Q}\left[S_{W,p,j}\vec{f}(x)\right]^p\,dx
&\leq\int_{Q}\left[S_{V,p}\left(\vec{f}\mathbf{1}_Q\right)(x)\right]^p\,dx\\
&\lesssim[V]_{\mathscr{A}_p}^{\gamma(p)p}
\left\|\vec{f}\mathbf{1}_Q\right\|^p_{L^p(V)}
\lesssim[W]_{\mathscr{A}^{\operatorname{loc}}_p(2^{-j})}^{\gamma(p)p}
\left\|\vec{f}\mathbf{1}_Q\right\|^p_{L^p(W)}.
\end{align*}
Summing over all dyadic cubes in $\mathcal{D}_j$
and taking the $p$-th root on both sides, we obtain
\begin{align*}
\left\|S_{W,p,j}\vec{f}\right\|_{L^p}
\lesssim[W]_{\mathscr{A}^{\operatorname{loc}}
_p(2^{-j})}^{\gamma(p)}
\left\|\vec{f}\right\|_{L^p(W)}.
\end{align*}
To finish the present proof, it suffices to show
\begin{align}\label{eq-Ejf}
\left\|E_j \vec{f}\right\|_{L^p(W)}\lesssim
[W]_{\mathscr{A}^{\operatorname{loc}}_p(2^{-j})}^{\gamma(p)}
\left\|\vec{f}\right\|_{L^p(W)}.
\end{align}
For this purpose, using H\"{o}lder's inequality and \eqref{eq-locApq},
we conclude that
\begin{align*}
\int_{Q}\left|W^{\frac{1}{p}}(x)E_j\vec{f}(x)\right|^p\,dx
&=\int_{Q}\left|W^{\frac{1}{p}}(x)\fint_Q\vec{f}(y)\,dy\right|^p\,dx
\leq\int_{Q}\left[\fint_Q\left|W^{\frac{1}{p}}(x)
\vec{f}(y)\right|\,dy\right]^p\,dx\\
&\leq\int_{Q}\left[\fint_Q\left\|W^{\frac{1}{p}}(x)W^{-\frac{1}{p}}(y)\right\|
\left|W^{\frac{1}{p}}(y)\vec{f}(y)\right|\,dy\right]^p\,dx\\
&\leq\fint_{Q}\left[\fint_Q\left\|W^{\frac{1}{p}}(x)W^{-\frac{1}{p}}(y)
\right\|^{p^{\prime}}\,dy\right]^{\frac{p}{p^{\prime}}}\,dx
\int_{Q}\left|W^{\frac{1}{p}}(y)\vec{f}(y)\right|^p\,dy\\
&\leq[W]_{\mathscr{A}^{\operatorname{loc}}_p(2^{-j})}
\int_{Q}\left|W^{\frac{1}{p}}(y)\vec{f}(y)\right|^p\,dy.
\end{align*}
Summing over all dyadic cubes $Q\in\mathcal{D}_j$, we obtain
\begin{align*}
\left\|E_j \vec{f}\right\|^p_{L^p(W)}&=\sum_{Q\in\mathcal{D}_j}
\int_{Q}\left|W^{\frac{1}{p}}(x)E_j\vec{f}(x)\right|^p\,dx\\
&\leq[W]_{\mathscr{A}^{\operatorname{loc}}_p(2^{-j})}
\sum_{Q\in\mathcal{D}_j}\int_{Q}\left|W^{\frac{1}{p}}(y)\vec{f}(y)\right|^p\,dy\\
&=[W]_{\mathscr{A}^{\operatorname{loc}}_p(2^{-j})}
\int_{\mathbb{R}^n}\left|W^{\frac{1}{p}}(y)\vec{f}(y)\right|^p\,dy,
\end{align*}
which further implies that
\begin{align*}
\left\|E_j \vec{f}\right\|_{L^p(W)}\leq
[W]_{\mathscr{A}^{\operatorname{loc}}_p(2^{-j})}^{\frac{1}{p}}
\left\|\vec{f}\right\|_{L^p(W)}\leq
[W]_{\mathscr{A}^{\operatorname{loc}}_p(2^{-j})}^{\gamma(p)}
\left\|\vec{f}\right\|_{L^p(W)}
\end{align*}
and hence \eqref{eq-Ejf} holds.
This completes the proof of Theorem \ref{thm-HAAR-Aploc}.
\end{proof}

In the study of matrix-weighted function spaces,
the averaging spaces associated with reducing operators
are easier to deal with and therefore particularly useful,
as the reducing operators satisfy doubling properties
(see \cite{BHYY:I, fr21, ro03, v97}
for applications of this approach).
Thus, the Haar square function associated with
reducing operators also attracts
considerable interest (see, for example, \cite{DKPS,isr21,nt96,v97}).
Suppose that $p\in(1,\infty)$, $W$ is a matrix weight,
and $\{A_I\}_{I\in\mathcal{D}}$ is a sequence of reducing operators
of order $p$ for $W$.  For any locally integrable vector-valued function
$\vec{f}:\mathbb{R}^n\longrightarrow \mathbb{C}^m$ ,
its \emph{Haar square function}  $\widetilde{S}_{W,p}\vec{f}$ is defined by
setting, for any $x\in\mathbb{R}^n$,
\begin{align}\label{eq-def-HaarSquareTilde}
\widetilde{S}_{W,p}\vec{f}(x):=
\left[\sum_{I \in \mathcal{D}}\sum_{k\in\operatorname{Sig}_n}
\frac{|A_I \vec{f}_I^{(k)}|^2}{|I|}
\mathbf{1}_I(x)\right]^{\frac{1}{2}},
\end{align}
where $\vec{f}_I^{(k)}$ is as in \eqref{eq-haarcof}.
Combining \cite[Theorem 1.1]{isr21} and \cite[Corollaries 5.4 and 5.6]{DKPS},
we obtain the following quantitative weighted norm estimate for
the Haar square function in \eqref{eq-def-HaarSquareTilde}.

\begin{theorem}\label{thm-Haar-Ap}
If $p\in(1,\infty)$, $W\in\mathscr{A}_p$, and $\{A_I\}_{I\in\mathcal{D}}$
is a sequence of reducing operators of order $p$ for $W$, then, for any
$\vec{f} \in L^p(W)$,
\begin{align*}
\left\|\widetilde{S}_{W,p}\vec{f}\right\|_{L^p}
\lesssim\left[W\right]_{\mathscr{A}_p}^{
\min\{\frac{2}{p}+\frac{\lceil p^{\prime}\rceil}{p},
\frac{1}{p}+\gamma(p)\}}
\left\|\vec{f}\right\|_{L^p(W)}
\end{align*}
and
\begin{align*}
\left\|\vec{f}\right\|_{L^p(W)} \lesssim
\left[W\right]_{\mathscr{A}_p}^{\min\{\frac{1}{p}
+\frac{\lceil p \rceil}{p}, \frac{1}{p}+\frac{\gamma(p^{\prime})}{p-1}\}}
\left\|\widetilde{S}_{W,p}\vec{f}\right\|_{L^p},
\end{align*}
where all the implicit positive constants
are independent of $\vec{f}$ and $W$
and $\gamma(p)$ is as in \eqref{eq-gamma-p}.
\end{theorem}

If we take the summation in \eqref{eq-def-HaarSquareTilde}
on all dyadic cubes with the edge length not greater
than $2^{-j}$ for some $j\in\mathbb{Z}$
as in \eqref{eq-def-locHaarSquare},
we obtain the corresponding \emph{local Haar square
function} $\widetilde{S}_{W,p,j}$.
Following a similar spirit to that used in
the proof of Theorem \ref{thm-HAAR-Aploc},
we establish the local version of
Theorem \ref{thm-Haar-Ap} as follows.

\begin{theorem}\label{thm-Haar-Aploc}
Let $j\in\mathbb{Z}$ and $p\in(1,\infty)$, and let
$W\in\mathscr{A}^{\operatorname{loc}}_p(2^{-j})$.
Assume that $\{A_I\}_{I\in\mathcal{D}}$
is a sequence of reducing operators of order $p$ for $W$.
Then, for any  $\vec{f} \in L^p(W)$,
\begin{align}\label{eq-Haar-Aploc-2}
\left\|\widetilde{S}_{W,p,j}\vec{f}\right\|_{L^p}
+\left\|E_j \vec{f}\right\|_{L^p(W)}
\lesssim[W]_{\mathscr{A}^{\operatorname{loc}}_p(2^{-j})}^{
\min\{\frac{2}{p}+\frac{\lceil p^{\prime}\rceil}{p},
\frac{1}{p}+\gamma(p)\}}
\left\|\vec{f}\right\|_{L^p(W)}
\end{align}
and
\begin{align}\label{eq-Haar-Aploc-1}
\left\|\vec{f}\right\|_{L^p(W)}
\lesssim[W]_{\mathscr{A}^{\operatorname{loc}}_p(2^{-j})}^{\min\{\frac{1}{p}
+\frac{\lceil p \rceil}{p}, \frac{1}{p}+\frac{\gamma(p^{\prime})}{p-1}\}}
\left[\left\|\widetilde{S}_{W,p,j}\vec{f}\right\|_{L^p}
+\left\|E_j \vec{f}\right\|_{L^p(W)}\right],
\end{align}
where all the implicit positive constants are independent of
$\vec{f}$, $W$, and $j$ and
$\gamma(p)$ is as in \eqref{eq-gamma-p}.
\end{theorem}
\begin{proof}
The proof of \eqref{eq-Haar-Aploc-2} is similar to that
of \eqref{eq-HAAR-Aploc-2} with $S_{W,p,j}$ therein
replaced by $\widetilde{S}_{W,p,j}$; we omit the details.

Now, we show \eqref{eq-Haar-Aploc-1}.
Without loss of generality, we may assume that $\vec{f}$
is a bounded function on $\mathbb{R}^n$ with compact support
because the set of bounded functions on $\mathbb{R}^n$
with compact support is dense in $L^p(W)$
(see \cite[Proposition 3.6]{cmr16}).
Let $Q\in\mathcal{D}_j$ and $V\in\mathscr{A}_p$ be as in \eqref{eq-V=W}.
Note that, in the proof of Theorem \ref{thm-extension},
the new matrix weight $V$ is constructed by
reflections and translations of $W$ on $Q$.
Therefore, for any $I\in\mathcal{D}$ with $Q\subset I$
and for any $\vec{z}\in\mathbb{C}^m$,
\begin{align}\label{eq-V_R=A_Q}
\left[\fint_{I}\left|V^{\frac{1}{p}}(x)\vec{z}
\right|^p\,dx\right]^{\frac{1}{p}}=\left[\frac{1}{|I|}
\frac{|I|}{|Q|}\int_{Q}
\left|W^{\frac{1}{p}}(x)\vec{z}
\right|^p\,dx\right]^{\frac{1}{p}}\sim\left|A_Q\vec{z}\right|,
\end{align}
where the positive equivalence constants
are independent of $Q$, $I$, and $\vec{z}$.
Hence $A_Q$ serves as a reducing operator
of order $p$ for $V$ on $I$
(with constants independent of $I$).
Let $\beta:=\min\{1
+\lceil p \rceil, 1+\frac{p\gamma(p^{\prime})}{p-1}\}$.
Combining \eqref{eq-coff-f1Q}, \eqref{eq-V=W},
and \eqref{eq-V_R=A_Q} and applying
Theorem \ref{thm-Haar-Ap} to $\vec{f}\mathbf{1}_Q$
and $V$, we obtain
\begin{align}\label{eq-f1Q}
\left\|\vec{f}\mathbf{1}_Q\right\|^p_{L^p(W)}
&=\left\|\vec{f}\mathbf{1}_Q\right\|^p_{L^p(V)}\\
&\lesssim[V]_{\mathscr{A}_p}^{\beta}
\int_{\mathbb{R}^n}\left[\sum_{\genfrac{}{}{0pt}{}{I\in\mathcal{D}}{I\subset Q}}\sum_{k\in\operatorname{Sig}_n}
\frac{|A_I \vec{f}_I^{(k)}|^2}{|I|}
\mathbf{1}_I(x)\right]^{\frac{p}{2}}\,dx\nonumber\\
&\quad+[V]_{\mathscr{A}_p}^{\beta}
\int_{\mathbb{R}^n}\left[\sum_{\genfrac{}{}{0pt}{}{I\in\mathcal{D}}{Q\subsetneqq I}}
\sum_{k\in\operatorname{Sig}_n}
\frac{|A_Qh_I^{(k)}(x_Q)\int_{Q}\vec{f}(y)\,dy|^2}{|I|}
\mathbf{1}_I(x)\right]^{\frac{p}{2}}\,dx\nonumber\\
&=[V]_{\mathscr{A}_p}^{\beta}
\int_{Q}\left[\widetilde{S}_{W,p,j}
\vec{f}(x)\right]^p\,dx\nonumber\\
&\quad+[V]_{\mathscr{A}_p}^{\beta}(2^{n}-1)^{\frac{p}{2}}\left|A_Q\int_{Q}
\vec{f}(y)\,dy\right|^p\int_{\mathbb{R}^n}
\left[\sum_{\genfrac{}{}{0pt}{}{I\in\mathcal{D}}{Q\subsetneqq I}}
\frac{1}{|I|^2} \mathbf{1}_I(x)\right]^{\frac{p}{2}}\,dx.\nonumber
\end{align}
Next, we estimate the last integral in \eqref{eq-f1Q}.
To this end, for any $h\in\mathbb{Z}_+$,
let $I_h$ be  the unique dyadic cube in
$\mathcal D_{j-h}$ containing $Q$.
Since $\{I_h\}_{h\in\mathbb{Z}_+}$ is a nested family
and $|I_h|=2^{hn}|Q|$,
we deduce that, for any $h\in\mathbb{Z}_+$ and
$x\in I_{h}\setminus I_{h-1}$ (where $I_{-1}:=\emptyset$),
\[
\sum_{\genfrac{}{}{0pt}{}{I\in\mathcal{D}}{Q\subsetneqq I}}
\frac{1}{|I|^2}\mathbf 1_I(x)
=\sum_{l=h}^{\infty}\frac{1}{|I_l|^2}
\sim\frac{1}{2^{2hn}|Q|^2}
\]
and hence
\begin{align*}
\int_{\mathbb R^n}\left[\sum_{\genfrac{}{}{0pt}{}{I\in\mathcal{D}}{Q\subsetneqq I}}
\frac{1}{|I|^2} \mathbf{1}_I(x)\right]^{\frac{p}{2}}\, dx
&=\sum_{h\in\mathbb{Z}_+}\int_{I_h\setminus I_{h-1}}
\left[\sum_{\genfrac{}{}{0pt}{}{I\in\mathcal{D}}{Q\subsetneqq I}}
\frac{1}{|I|^2}\mathbf{1}_I(x)\right]^{\frac{p}{2}}\,dx \\
&\lesssim\sum_{h\in\mathbb{Z}_+}|I_h|
\left(\frac{1}{2^{2hn}|Q|^2}\right)^{\frac p2}\\
&=|Q|^{1-p}\sum_{h\in\mathbb{Z}_+}
\frac{1}{2^{hn(p-1)}}\sim|Q|^{1-p}.
\end{align*}
Inserting this estimate into \eqref{eq-f1Q} and using \eqref{eq-Ej},
\eqref{eq-V=W}, and \eqref{eq-reduce}, we conclude that
\begin{align}\label{eq-f-Q}
\left\|\vec{f}\mathbf{1}_Q\right\|^p_{L^p(W)}
&\lesssim[V]_{\mathscr{A}_p}^{\beta}\\
&\quad\times\left\{\int_{Q}\left[\widetilde{S}_{W,p,j}
\vec{f}(x)\right]^p\,dx+\int_{Q}
\left|W^{\frac{1}{p}}(x)
E_j\vec{f}(x)\right|^p\,dx\right\}\nonumber\\
&\lesssim[W]_{\mathscr{A}^{\operatorname{loc}}_p(2^{-j})}^{\beta}\nonumber\\
&\quad\times\left\{\int_{Q}\left[\widetilde{S}_{W,p,j}
\vec{f}(x)\right]^p\,dx+\int_{Q}
\left|W^{\frac{1}{p}}(x)E_j\vec{f}(x)\right|^p\,dx\right\}.\nonumber
\end{align}
Summing over all dyadic cubes in $\mathcal{D}_j$ on both sides of
\eqref{eq-f-Q}, we obtain
\begin{align*}
\left\|\vec{f}\right\|^p_{L^p(W)}
&=\sum_{Q\in\mathcal{D}_j}\left\|\vec{f}\mathbf{1}_Q\right\|^p_{L^p(W)}\\
&\lesssim[W]_{\mathscr{A}^{\operatorname{loc}}_p(2^{-j})}^{\beta}\\
&\quad\times\left\{\sum_{Q\in\mathcal{D}_j}\int_{Q}
\left[\widetilde{S}_{W,p,j}
\vec{f}(x)\right]^p\,dx+\sum_{Q\in\mathcal{D}_j}\int_{Q}
\left|W^{\frac{1}{p}}(x)E_j\vec{f}(x)\right|^p\,dx\right\}\\
&=[W]_{\mathscr{A}^{\operatorname{loc}}_p(2^{-j})}^{\beta}
\left(\left\|\widetilde{S}_{W,p,j}\vec{f}
\right\|^p_{L^p}+\left\|E_j \vec{f}\right\|^p_{L^p(W)}\right).
\end{align*}
This completes the proof of \eqref{eq-Haar-Aploc-1}
and hence Theorem \ref{thm-Haar-Aploc}.
\end{proof}

\begin{remark}
Note that the exponents of matrix weight constants in
Theorems \ref{thm-HAAR-Aploc} and \ref{thm-Haar-Aploc}
respectively match the corresponding global ones
in Theorems \ref{thm-HAAR-Ap} and \ref{thm-Haar-Ap}.
\end{remark}

\subsection{Calder\'{o}n--Zygmund Operators With Exponential Decay}\label{s3-4}

In this subsection, we aim to establish the boundedness of
Calder\'{o}n--Zygmund operators with exponential decay on
local matrix-weighted Lebesgue spaces.
To this end, we first recall the definition of such operators
(see \cite[Definition 2.7]{dlt25} for similar operators with exponential decay
and a critical radius function).
Let $\delta\in(0,1]$ and $t\in[0,\infty)$.
A linear operator $T$ is called a
\emph{$(\delta,t)$-Calder\'{o}n--Zygmund operator}
if $T$ is bounded on $L^2$ and,
for any $f \in C_{\operatorname{c}}^{\infty}$
(the \emph{set of all infinitely differentiable functions on $\mathbb{R}^n$
with compact support}) and $x \notin \operatorname{supp} f$,
$$
T f(x):=\int_{\mathbb{R}^n} K(x, y) f(y)\,dy,
$$
where the kernel function $K:(\mathbb{R}^n\times\mathbb{R}^n)\setminus
\{(x,x):x\in\mathbb{R}^n\}\to\mathbb{C}$ satisfies
that there exists a positive constant $C$ such that
\begin{itemize}
\item[{\rm (i)}] for any $x,y\in\mathbb{R}^n$ satisfying $x\neq y$,
\begin{align}\label{eq-decay-condition}
|K(x, y)|\leq C\frac{e^{-t|x-y|}}{|x-y|^{n}},
\end{align}
\item[{\rm (ii)}] for any $x,y,h\in\mathbb{R}^n$
satisfying $|x-y|>2|h|$,
\begin{align}\label{eq-smooth-condition}
|K(x, y)-K(x, y+h)|+|K(x, y)-K(x+h, y)|\leq
C\frac{|h|^\delta}{|x-y|^{n+\delta}}.
\end{align}
\end{itemize}
\begin{remark}
It immediately follows from \eqref{eq-decay-condition}
that every $(\delta,t)$-Calder\'{o}n--Zygmund operator is
also a standard $\delta$-Calder\'{o}n--Zygmund operator
(see, for instance, \cite[p.\,36]{cim18}).
\end{remark}

We first point out that the exponential decay condition
in \eqref{eq-decay-condition} cannot be improved in general.
This is because, even in the scalar case,
the class $A_{p}^{\mathrm{loc}}(r)$ contains the weight
$e^{|\cdot|}$, which has exponential growth.
The following proposition precisely indicates this point.
Suppose that $\rho$ is a positive function defined on $[0,\infty)$.
For any bounded function $f$ on $\mathbb{R}^n$ with compact support
and for any $x\in\mathbb{R}^n$, let
\begin{align}\label{eq-def-Jrho}
J_\rho f(x):=\int_{\mathbb{R}^n}f(y)
\rho(|x-y|)\,dy.
\end{align}
\begin{proposition}\label{prop-exp-decay}
Let $p\in(1,\infty)$ and $\rho$ be a positive
measurable function defined on $[0,\infty)$.
If $J_{\rho}$ is bounded on $L^p(e^{|\cdot|})$,
then
\begin{align*}
\int_{ \mathbb{R}^n}\left[\int_{B(x, 1)}
\rho(|y|)\,dy\right]^p e^{|x|}\,dx<\infty,
\end{align*}
where $J_{\rho}$ is as in \eqref{eq-def-Jrho}.
Moreover, if $\rho$ is decreasing, then
\begin{align}\label{eq-decay-e1}
\int_{1}^{\infty}\rho^p(r)e^r r^{n-1}\,dr<\infty.
\end{align}
\end{proposition}

\begin{proof}
From \eqref{eq-def-Jrho}, we deduce that, for any $x\in \mathbb{R}^n$,
\begin{align*}
\left|J_{\rho} \mathbf{1}_{B(\mathbf{0}, 1)}(x)\right|=
\int_{B(\mathbf{0}, 1)}\rho\left(|x-y|\right)\,dy
=\int_{B(x, 1)}\rho\left(|y|\right)\,dy.
\end{align*}
This, combined with the assumption that $J_{\rho}$
is bounded on $L^p(e^{|\cdot|})$,
further implies that
\begin{align}\label{eq-decay-e2}
\int_{ \mathbb{R}^n}\left[\int_{B(x, 1)}\rho(|y|)
\,dy\right]^p e^{|x|}\,dx
&=\int_{\mathbb{R}^n}|J_{\rho} \mathbf{1}_{B(\mathbf{0}, 1)}(x)|^p e^{|x|}\,dx\\
&=\left\|J_{\rho} \mathbf{1}_{B(\mathbf{0}, 1)}\right\|^p_{L^p(e^{|\cdot|})}
\lesssim\left\|\mathbf{1}_{B(\mathbf{0}, 1)}\right\|^p_{L^p(e^{|\cdot|})}\sim 1.\nonumber
\end{align}
Moreover, if $\rho$ is decreasing, then,
for any $x\in\mathbb{R}^n$ with $|x|\geq1$
and $y\in B(x-\frac{x}{2|x|}, \frac{1}{2})$,
we have $|y|\leq|x-\frac{x}{2|x|}|+\frac{1}{2}=|x|$,
and hence $\rho(|x|)\leq\rho(|y|)$.
By this fact, we conclude that
\begin{align*}
\int_{1}^{\infty}[\rho(r)]^pe^r r^{n-1}\,dr&\sim
\int_{B(\mathbf{0},1)^\complement}[\rho(|x|)]^pe^{|x|}\,dx\\
&\lesssim\int_{B(\mathbf{0},1)^\complement}
\left[\int_{B(x-\frac{x}{2|x|}, \frac{1}{2})}
\rho(|y|)\,dy\right]^p e^{|x|}\,dx.
\end{align*}
This, together with \eqref{eq-decay-e2} and the fact that,
for any $x\in\mathbb{R}^n$ with $|x|\geq1$,
$B(x-\frac{x}{2|x|}, \frac{1}{2})\subset B(x,1)$,
further implies \eqref{eq-decay-e1}.
This completes the proof of Proposition \ref{prop-exp-decay}.
\end{proof}

The following theorem is the main result of this subsection,
which establishes the boundedness of
$(\delta,t)$-Calder\'{o}n--Zygmund operators on $L^p(W)$ with local matrix
$\mathscr{A}^{\operatorname{loc}}_{p}(r)$ weights.

\begin{theorem}\label{thm-CZloc}
Let $r\in(0,\infty)$, $\delta\in(0,1]$, $p\in(1,\infty)$, and
$W\in\mathscr{A}^{\operatorname{loc}}_{p}(r)$.
If
\begin{align}\label{eq-t}
t\in\left(\frac{3}{p r}\ln\left(2^{np}
[W]_{\mathscr{A}^{\operatorname{loc}}_{p}(r)}\right),\infty\right)
\end{align}
and $T$ is a $(\delta,t)$-Calder\'{o}n--Zygmund operator,
then there exists a positive constant $C$
such that
\begin{align}\label{e4.52}
\|T\|_{L^p(W)\to L^p(W)}\leq
C[W]_{\mathscr{A}^{\operatorname{loc}
}_{p}(r)}^{1+\frac{1}{p-1}-\frac{1}{p}},
\end{align}
where the positive constant $C$ depends on
$T$, $n$, $m$, $r$, $\delta$, and $p$.
\end{theorem}

Note that, in \eqref{eq-decay-condition}, when $x$ is close to
$y$, then $e^{-t|x-y|}\sim1$. Thus,
a $(\delta,t)$-Calder\'{o}n--Zygmund operator locally behaves
like a $\delta$-Calder\'{o}n--Zygmund operator.
Motivated by this observation, we recall the boundedness of
$\delta$-Calder\'{o}n--Zygmund operators
on $L^p(W)$ with matrix $\mathscr{A}_p$ weights established
by Nazarov et al. in \cite[Theorem 1.1]{nptv} for $p=2$ and
by Cruz-Uribe et al. in \cite[Corollary 1.16]{cim18}
for all $p\in(1,\infty)$.

\begin{theorem}\label{thm-CZ}
Let $\delta\in(0,1]$, $p\in(1,\infty)$, and
$T$ be a $\delta$-Calder\'{o}n--Zygmund operator.
Then there exists a positive constant $C$ such that,
for any $W\in\mathscr{A}_{p}$,
\begin{align}\label{e4.53}
\|T\|_{L^p(W)\to L^p(W)} \leq C
[W]_{\mathscr{A}_p}^{1+\frac{1}{p-1}-\frac{1}{p}},
\end{align}
where $C$ is independent of $W$.
\end{theorem}

\begin{remark}
Observe that the exponent of the matrix weight constant
in \eqref{e4.52} matches the one in \eqref{e4.53}. In particular,
if $p=2$, then the bound $[W]_{\mathscr{A}^{
\operatorname{loc}}_{p}(r)}^{1+\frac{1}{p-1}-\frac{1}{p}}$
in \eqref{e4.52} becomes precisely
$[W]_{\mathscr{A}^{\operatorname{loc}}_{p}(r)}^{\frac 32}$.
However, it is still unknown whether this bound
in \eqref{e4.52} is sharp.
\end{remark}

To prove Theorem \ref{thm-CZloc}, we need the
following boundedness of integral operators
with exponential decay, whose proof
borrows some ideas from the proof of \cite[(2.20)]{ry01}.

\begin{lemma}\label{lem-vector-K_a}
Let $r\in(0,\infty)$, $p\in(1,\infty)$, and
$W\in\mathscr{A}^{\operatorname{loc}}_{p}(r)$.
If $t$ satisfies \eqref{eq-t}, then there exists a positive
constant $C$, depending on $n$, $m$, $p$, and $r$,
such that, for any $\vec{f}\in L^p$,
\begin{align*}
\left\|W^{\frac{1}{p}}J_{e^{-t|\cdot|}}
\left(W^{-\frac{1}{p}}\vec{f}\right)\right\|_{L^p}\leq
C[W]^{\frac{2}{p}}_{
\mathscr{A}^{\operatorname{loc}}_{p}(r)}
\left\|\vec{f}\right\|_{L^p},
\end{align*}
where $J_{e^{-t|\cdot|}}$ is the operator defined in
\eqref{eq-def-Jrho} with $\rho:=e^{-t|\cdot|}$.
\end{lemma}

\begin{proof}
Let $\vec{f}\in L^p$.
First, we partition $\mathbb{R}^n$ into a grid $\Gamma$ of cubes with
edge length $r$ and disjoint interiors.
Applying the definition of the operator $J_{e^{-t|\cdot|}}$ and 	
the grid structure of $\mathbb{R}^n$, we obtain,
for any $x\in\mathbb{R}^n$,
\begin{align}\label{eq-1}
\left|W^{\frac{1}{p}}(x)J_{e^{-t|\cdot|}}
\left(W^{-\frac{1}{p}}\vec{f}\right)(x)\right|
&\leq\int_{\mathbb{R}^n}\left|W^{\frac{1}{p}}(x)
W^{-\frac{1}{p}}(y)\vec{f}(y)\right|e^{-t|x-y|}\,dy\\
&=\sum_{Q\in\Gamma}\int_{Q}
\left|W^{\frac{1}{p}}(x)
W^{-\frac{1}{p}}(y)\vec{f}(y)\right|
e^{-t|x-y|}\,dy\nonumber\\
&\leq\sum_{Q\in\Gamma}e^{-t d(x, Q)}
\int_{Q}\left|W^{\frac{1}{p}}(x)
W^{-\frac{1}{p}}(y)\vec{f}(y)\right|\,dy,\nonumber
\end{align}
where $d(x, Q):=\inf_{y\in Q}|x-y|$.

Let $M:=2^{np}[W]_{\mathscr{A}^{\operatorname{loc}}_{p}(r)}$
and $\{A_Q\}_{\operatorname{cube}\,Q\subset\mathbb{R}^n}$
be a family of reducing operators of order $p$ for $W$.
Next, we claim that, if
\begin{align}\label{eq-alpha}
\alpha=\frac{5}{2pr}\ln M,
\end{align}
then, for any $Q\in\Gamma$,
\begin{align}\label{eq-intRn<1}
\int_{\mathbb{R}^n}
\left\|W^{\frac{1}{p}}(x)A^{-1}_{Q}\right\|^p
e^{-\alpha p d(x, Q)}\,dx\lesssim
[W]_{\mathscr{A}^{\operatorname{loc}}_{p}(r)}|Q|.
\end{align}

To show \eqref{eq-intRn<1}, fix $Q\in\Gamma$.
Let $\{e_i\}_{i=1}^m$ be an orthonormal basis of $\mathbb{C}^m$.
Applying the equivalence norm theorem in finite dimensional spaces,
we conclude that, for any $k\in\mathbb{N}$,
\begin{align}\label{eq-equi-kQ}
\int_{(k+1)Q\setminus k Q}
\left\|W^{\frac{1}{p}}(x)A^{-1}_{Q}\right\|^p\,dx
\sim\sum_{i=1}^{m}\int_{(k+1)Q\setminus k Q}
\left|W^{\frac{1}{p}}(x)A^{-1}_{Q}e_i\right|^p\,dx.
\end{align}
For any $k\in\mathbb{N}$, dividing each edge of $(k+1)Q$
into $k+1$ equal parts, we obtain $(k+1)^n$ cubes of
edge length $r$. Remove those cubes contained in $(k-1)Q$
and denote the remaining cubes by $\{R_j\}_{j=1}^{N_k}$,
where $0Q:=\emptyset$ and $N_k:=(k+1)^n-(k-1)^n$.
By this construction, we find that, for any $k\in\mathbb{N}$,
the cubes $\{R_j\}_{j=1}^{N_k}$
have pairwise disjoint interiors,
$$(k+1)Q\setminus (k-1)Q= \bigcup_{j=1}^{N_k} R_j,$$
and, for any $j\in\{1,\dots,N_k\}$,
$$|R_j\cap kQ|\geq \left(\frac{r}{2}\right)^n.$$
Since the edge length of $R_j$ is exactly
$r$, the local doubling property of
local $A^{\operatorname{loc}}_p(r)$ holds
(see, for instance, \cite[Proposition 2.2.4]{ooi24}).
From this, the properties of $\{R_j\}_{j=1}^{N_k}$ above,
and Lemma \ref{lem-W-w}, we deduce that, for any
$k\in\mathbb{N}$ and $i\in\{1,\dots,m\}$,
\begin{align}\label{eq-kQtok-1Q}
\int_{(k+1)Q\setminus k Q}
\left|W^{\frac{1}{p}}(x)A^{-1}_{Q}e_i\right|^p\,dx
&\leq\sum_{j=1}^{N_k}\int_{R_j}
\left|W^{\frac{1}{p}}(x)A^{-1}_{Q}e_i\right|^p\,dx\\
&\leq2^{np}[W]_{\mathscr{A}^{\operatorname{loc}}_p(r)}
\sum_{j=1}^{N_k}\int_{R_j\cap kQ}
\left|W^{\frac{1}{p}}(x)A^{-1}_{Q}e_i\right|^p\,dx\nonumber\\
&=M\int_{kQ\setminus (k-1) Q}
\left|W^{\frac{1}{p}}(x)A^{-1}_{Q}e_i\right|^p\,dx.\nonumber
\end{align}
Iterating \eqref{eq-kQtok-1Q} $k$ times and applying
\eqref{eq-equi-kQ}, we conclude that, for any $k\in\mathbb{N}$,
\begin{align*}
\int_{(k+1)Q\setminus k Q}
\left|W^{\frac{1}{p}}(x)A^{-1}_{Q}e_i\right|^p\,dx
&\leq M^k
\int_{Q}\left|W^{\frac{1}{p}}(x)A^{-1}_{Q}e_i\right|^p\,dx,
\end{align*}
and hence
\begin{align}\label{eq-intk-Mk}
\int_{(k+1)Q\setminus k Q}
\left\|W^{\frac{1}{p}}(x)A^{-1}_{Q}\right\|^p\,dx
&\sim\sum_{i=1}^{m}\int_{(k+1)Q\setminus k Q}
\left|W^{\frac{1}{p}}(x)A^{-1}_{Q}e_i\right|^p\,dx\\
&\leq M^k
\sum_{i=1}^{m}\int_{Q}\left|W^{\frac{1}{p}}(x)A^{-1}_{Q}e_i\right|^p\,dx\nonumber\\
&\sim M^k
\int_{Q}\left\|W^{\frac{1}{p}}(x)A^{-1}_{Q}\right\|^p\,dx,\nonumber
\end{align}
where the last equivalence follows from the equivalence
norm theorem in finite dimensional spaces.
Note that, for any $k\in\mathbb{N}$ and $x\in(k+1)Q \setminus kQ$,
$d(x,Q)\geq\frac{(k-1)r}{2}$. This, together with
\eqref{eq-intk-Mk} and \eqref{eq-reduce}, further implies that
\begin{align*}
&\int_{\mathbb{R}^n} \left\|W^{\frac{1}{p}}(x)A^{-1}_{Q}\right\|^p
e^{-\alpha p d(x, Q)}\,dx \\
&\quad\leq \int_{Q}\left\|W^{\frac{1}{p}}(x)A^{-1}_{Q}\right\|^p \,dx +
\sum_{k\in\mathbb{N}} e^{-\alpha p \frac{k-1}{2}r} \int_{(k+1)Q \setminus kQ} \left\|W^{\frac{1}{p}}(x)A^{-1}_{Q}\right\|^p \,dx \\
&\quad\lesssim \int_{Q}\left\|W^{\frac{1}{p}}(x)A^{-1}_{Q}\right\|^p \,dx +
\sum_{k\in\mathbb{N}} e^{-\alpha p \frac{k-1}{2}r} M^k \int_{Q}\left\|W^{\frac{1}{p}}(x)A^{-1}_{Q}\right\|^p \,dx \\
&\quad\sim\left[1+M\sum_{k\in\mathbb{N}}
\left(M e^{-\frac{\alpha p r}{2}}\right)^{k-1} \right]
|Q|=: \Omega.
\end{align*}
If $\alpha$ satisfies \eqref{eq-alpha}, then
$M e^{-\frac{\alpha p r}{2}}=M^{-\frac{1}{4}}$.
Due to $[W]_{\mathscr{A}^{\operatorname{loc}}_{p}(r)}\geq1$
(see Corollary \ref{coro-[W]geq1}),
we have $M\geq 2^{np}\geq 2$. Thus, the series in $\Omega$ converges
and hence
$$ \Omega \lesssim (1 + M) |Q| \lesssim
[W]_{\mathscr{A}^{\operatorname{loc}}_{p}(r)}|Q|. $$
This completes the proof of \eqref{eq-intRn<1}.

Applying H\"{o}lder's inequality to \eqref{eq-1},
we conclude that
\begin{align}\label{eq-2}
&\left\|W^{\frac{1}{p}}J_{e^{-t|\cdot|}}
\left(W^{-\frac{1}{p}}\vec{f}\right)\right\|^p_{L^p}\\
&\quad\leq\int_{\mathbb{R}^n}
\left[\sum_{Q\in\Gamma}e^{-t d(x, Q)}
\int_{Q}\left|W^{\frac{1}{p}}(x)
W^{-\frac{1}{p}}(y)\vec{f}(y)\right|\,dy\right]^p\,dx\nonumber\\
&\quad\leq\int_{\mathbb{R}^n}\sum_{Q\in\Gamma}
\left\|W^{\frac{1}{p}}(x)A^{-1}_{Q}\right\|^p
e^{-\alpha p d(x, Q)}\left[\int_{Q}
\left|A_{Q}W^{-\frac{1}{p}}(y)\vec{f}(y)\right|\,dy\right]^p\nonumber\\
&\quad\quad\times
\left[\sum_{Q\in\Gamma}e^{-(t-\alpha)p^{\prime}
d(x, Q)}\right]^{\frac{p}{p^{\prime}}}\,dx.\nonumber
\end{align}
We now estimate the last sum in \eqref{eq-2}.
Observe that, for any $x\in\mathbb{R}^n$,
there exists a unique cube $Q_x\in\Gamma$
(up to boundaries) containing $x$.
For any $k\in\mathbb{N}$, let
\begin{align*}
R_k := \left\{ Q \in \Gamma : Q \subset (2k+1)Q_x
\setminus (2k-1)Q_x \right\}.
\end{align*}
The cardinality of $R_k$ is exactly $(2k+1)^n-(2k-1)^n$ and,
for any $Q \in R_k$, $d(x, Q) \ge (k-1)r$. Therefore, we have
\begin{align*}
\sum_{Q\in\Gamma}e^{-(t-\alpha)p^{\prime} d(x, Q)}
&= e^{-(t-\alpha)p^{\prime} d(x, Q_x)} +
\sum_{k\in\mathbb{N}}\sum_{Q\in R_k}
e^{-(t-\alpha)p^{\prime} d(x, Q)}\\
&\leq 1+\sum_{k\in\mathbb{N}}
e^{-(t-\alpha)p^{\prime}(k-1)r}[(2k+1)^n-(2k-1)^n]\\
&\lesssim 1+\sum_{k\in\mathbb{N}}k^{n-1}
e^{-(t-\alpha)p^{\prime}(k-1)r} <\infty.
\end{align*}
Inserting this uniform estimate into \eqref{eq-2}
and using \eqref{eq-intRn<1}, we conclude that
\begin{align}\label{eq-3}
\left\|W^{\frac{1}{p}}J_{e^{-t|\cdot|}}
\left(W^{-\frac{1}{p}}\vec{f}\right)\right\|^p_{L^p}
&\lesssim\sum_{Q\in\Gamma}\int_{\mathbb{R}^n}
\left\|W^{\frac{1}{p}}(x)A^{-1}_{Q}\right\|^p
e^{-\alpha p d(x, Q)}\,dx\\
&\quad\times\left[\int_{Q}\left|A_{Q}W^{-\frac{1}{p}}(y)
\vec{f}(y)\right|\,dy\right]^p\nonumber\\
&\lesssim[W]_{\mathscr{A}^{\operatorname{loc}}_{p}(r)}
\sum_{Q\in\Gamma}|Q|\left[\int_{Q}\left|A_{Q}W^{-\frac{1}{p}}(y)
\vec{f}(y)\right|\,dy\right]^p.\nonumber
\end{align}
From H\"{o}lder's inequality, \eqref{eq-locApq}, and
\eqref{eq-W-dualW}, we deduce that, for any $Q\in\Gamma$,
\begin{align*}
\int_{Q}\left|A_{Q}W^{-\frac{1}{p}}(y)
\vec{f}(y)\right|\,dy
&\leq
\left[\int_{Q}\left\|A_{Q}W^{-\frac{1}{p}}(y)
\right\|^{p^{\prime}}\,dy\right]^{\frac{1}{p^{\prime}}}
\left[\int_{Q}\left|\vec{f}(y)\right|^p\,dy\right]^{\frac{1}{p}}\\
&\sim\left[\int_{Q}
\left[\fint_{Q}\left\|W^{\frac{1}{p}}(x)W^{-\frac{1}{p}}(y)
\right\|^p\,dx\right]^\frac{p^{\prime}}{p}\,dy\right]^{\frac{1}{p^{\prime}}}
\left[\int_{Q}\left|\vec{f}(y)\right|^p\,dy\right]^{\frac{1}{p}}\\
&\leq|Q|^{\frac{1}{p^{\prime}}}
\left[W^{-\frac{p^{\prime}}{p}}\right]^{\frac{1}{p^{\prime}}}
_{\mathscr{A}^{\operatorname{loc}}_{p^{\prime}}(r)}
\left[\int_{Q}\left|\vec{f}(y)\right|^p\,dy\right]^{\frac{1}{p}}\\
&\sim|Q|^{\frac{1}{p^{\prime}}}
[W]^{\frac{1}{p}}
_{\mathscr{A}^{\operatorname{loc}}_{p}(r)}
\left[\int_{Q}\left|\vec{f}(y)\right|^p\,dy\right]^{\frac{1}{p}}.
\end{align*}
This, together with \eqref{eq-3}, further implies that
\begin{align*}
\left\|W^{\frac{1}{p}}J_{e^{-t|\cdot|}}
\left(W^{-\frac{1}{p}}\vec{f}\right)\right\|^p_{L^p}
\lesssim[W]^{2}_{\mathscr{A}^{\operatorname{loc}}_{p}(r)}
\sum_{Q\in\Gamma}|Q|^p\int_{Q}\left|\vec{f}(y)\right|^p\,dy
\sim [W]^{2}_{\mathscr{A}^{\operatorname{loc}}_{p}(r)}
\int_{\mathbb{R}^n}\left|\vec{f}(y)\right|^p\,dy.
\end{align*}
Taking the $p$-th root on both sides yields the desired estimate
\begin{align*}
\left\|W^{\frac{1}{p}}J_{e^{-t|\cdot|}}
\left(W^{-\frac{1}{p}}\vec{f}\right)\right\|_{L^p}
\leq C [W]^{\frac{2}{p}}
_{\mathscr{A}^{\operatorname{loc}}_{p}(r)}
\left\|\vec{f}\right\|_{L^p}.
\end{align*}
This completes the proof of Lemma \ref{lem-vector-K_a}.
\end{proof}

We now prove Theorem \ref{thm-CZloc}.

\begin{proof}[Proof of Theorem \ref{thm-CZloc}]
Without loss of generality, we may assume that $\vec{f}$
is a vector-valued function on $\mathbb{R}^n$ such that
$W^{-\frac{1}{p}}\vec{f}$ is bounded and has compact support
because the set of bounded functions on $\mathbb{R}^n$
with compact support is dense in $L^p(W)$
(see \cite[Proposition 3.6]{cmr16}).
Next, we partition $\mathbb{R}^n$ into a grid $\Gamma$ of cubes with
edge length $\frac{r}{3}$ and disjoint interiors.
For any $Q\in \Gamma$, let $\widetilde{Q}$ be the cube with
edge length $r$ and the same center as $Q$.
Note that, for any $Q\in \Gamma$,
$\vec{f} = \vec{f}\mathbf{1}_{\widetilde{Q}}+
\vec{f}\mathbf{1}_{{\widetilde{Q}}^\complement}$ and hence
\begin{align*}
\int_{Q}\left|W^{\frac{1}{p}}(x)T
\left(W^{-\frac{1}{p}}\vec{f}\right)(x)\right|^p\,dx
&\lesssim\int_{Q}\left|W^{\frac{1}{p}}(x)T
\left(W^{-\frac{1}{p}}\vec{f}\mathbf{1}_{\widetilde{Q}}\right)(x)
\right|^p\,dx\\
&\quad+	\int_{Q}\left|W^{\frac{1}{p}}(x)T
\left(W^{-\frac{1}{p}}\vec{f}\mathbf{1}_{{\widetilde{Q}}^\complement}\right)(x)
\right|^p\,dx\\
&=:\operatorname{I}_Q+\operatorname{II}_Q.
\end{align*}

To estimate  $\operatorname{I}_Q$, by Theorem \ref{thm-extension},
we find that there exists $V\in\mathscr{A}_p$
such that $V=W$ on $\widetilde{Q}$ and
\begin{align}\label{eq-V-W}
[V]_{\mathscr{A}_{p}}\leq
3^{np} [W]_{\mathscr{A}^{\operatorname{loc}}_{p}(r)}.
\end{align}
Since $T$ is also a $\delta$-Calder\'{o}n--Zygmund operator,
from \eqref{eq-V-W} and Theorem \ref{thm-CZ}, we
deduce that, for any $Q\in \Gamma$,
\begin{align}\label{eq-est-I-CZ}
\operatorname{I}_Q
&=\int_{Q}\left|V^{\frac{1}{p}}(x)T
\left(V^{-\frac{1}{p}}\vec{f}\mathbf{1}_{\widetilde{Q}}
\right)(x)\right|^p\,dx
\leq \int_{\mathbb{R}^n}\left|V^{\frac{1}{p}}(x)T
\left(V^{-\frac{1}{p}}\vec{f}\mathbf{1}_{\widetilde{Q}}
\right)(x)\right|^p\,dx\\
&\lesssim[V]_{\mathscr{A}_p}^{p+\frac{p}{p-1}-1}
\int_{\widetilde{Q}}\left|\vec{f}(x)\right|^p\,dx
\lesssim[W]_{\mathscr{A}^{\operatorname{loc}}_{p}(r)}^{p+\frac{p}{p-1}-1}
\int_{\widetilde{Q}}\left|\vec{f}(x)\right|^p\,dx.\nonumber
\end{align}

Next, we estimate $\operatorname{II}_Q$. Using the
exponential decay assumption \eqref{eq-decay-condition}
of $T$, we conclude that, for any $x\in Q$,
\begin{align*}
\left|W^{\frac{1}{p}}(x)T\left(W^{-\frac{1}{p}}\vec{f}
\mathbf{1}_{{\widetilde{Q}}^\complement}\right)(x)\right|
&\leq\int_{{\widetilde{Q}}^\complement}\left|K(x,y)\right| \left|W^{\frac{1}{p}}(x)
W^{-\frac{1}{p}}(y)\vec{f}(y)\right|\,dy\\
&\lesssim\int_{\mathbb{R}^n} e^{-t|x-y|}\left|W^{\frac{1}{p}}(x)
W^{-\frac{1}{p}}(y)\vec{f}(y)\right|\,dy
\end{align*}
and hence
\begin{align}\label{eq-est-II-decay}
\operatorname{II}_Q \lesssim\int_{Q}
\left[\int_{\mathbb{R}^n}
e^{-t|x-y|}\left|W^{\frac{1}{p}}(x)
W^{-\frac{1}{p}}(y)\vec{f}(y)\right|\,dy\right]^p
\,dx.
\end{align}
Observe that
\begin{align}\label{eq-sum}
\int_{\mathbb{R}^n}\left|W^{\frac{1}{p}}(x)T
\left(W^{-\frac{1}{p}}\vec{f}\right)(x)\right|^p\,dx
&\lesssim\sum_{Q\in\Gamma} \operatorname{I}_Q
+\sum_{Q\in\Gamma} \operatorname{II}_Q.
\end{align}
Applying \eqref{eq-est-I-CZ} and the fact that all
the cubes $\{\widetilde{Q}\}_{Q\in\Gamma}$ have bounded overlap,
we obtain
\begin{align}\label{eq-sumI}
\sum_{Q\in\Gamma} \operatorname{I}_Q
&\lesssim [W]_{\mathscr{A}^{\operatorname{loc}}_{p}(r)}^{p+\frac{p}{p-1}-1}
\sum_{Q\in\Gamma} \int_{\widetilde{Q}}\left|\vec{f}(x)\right|^p\,dx
\lesssim [W]_{\mathscr{A}^{\operatorname{loc}}_{p}(r)}^{p+\frac{p}{p-1}-1}
\int_{\mathbb{R}^n}\left|\vec{f}(x)\right|^p\,dx.
\end{align}
By \eqref{eq-est-II-decay} and Lemma \ref{lem-vector-K_a}
with $t$ satisfying \eqref{eq-t}, we find that
\begin{align}\label{eq-sumII}
\sum_{Q\in\Gamma} \operatorname{II}_Q
&\lesssim \sum_{Q\in\Gamma} \int_{Q}
\left[\int_{\mathbb{R}^n}
e^{-t|x-y|}\left|W^{\frac{1}{p}}(x)
W^{-\frac{1}{p}}(y)\vec{f}(y)\right|\,dy\right]^p\,dx\\
&= \int_{\mathbb{R}^n}
\left[\int_{\mathbb{R}^n}
e^{-t|x-y|}\left|W^{\frac{1}{p}}(x)
W^{-\frac{1}{p}}(y)\vec{f}(y)\right|\,dy\right]^p\,dx\nonumber\\
&\lesssim [W]_{\mathscr{A}^{\operatorname{loc}}_{p}(r)}^2
\int_{\mathbb{R}^n}\left|\vec{f}(x)\right|^p\,dx.\nonumber
\end{align}
Since $[W]_{\mathscr{A}^{\operatorname{loc}}_{p}(r)} \geq 1$
(see Corollary \ref{coro-[W]geq1}) and, for any $p \in (1, \infty)$,
$p + \frac{p}{p-1} - 1> 2$, inserting estimates \eqref{eq-sumII}
and \eqref{eq-sumI} into \eqref{eq-sum} yields
\begin{align*}
\int_{\mathbb{R}^n}\left|W^{\frac{1}{p}}(x)T
\left(W^{-\frac{1}{p}}\vec{f}\right)(x)\right|^p\,dx
&\lesssim \left\{ [W]_{\mathscr{A}^{\operatorname{loc}}_{p}(r)}^{p+\frac{p}{p-1}-1} + [W]_{\mathscr{A}^{\operatorname{loc}}_{p}(r)}^2\right\} \int_{\mathbb{R}^n}\left|\vec{f}(x)\right|^p\,dx \\
&\lesssim [W]_{\mathscr{A}^{\operatorname{loc}}_{p}(r)}^{p+\frac{p}{p-1}-1}
\int_{\mathbb{R}^n}\left|\vec{f}(x)\right|^p\,dx.
\end{align*}
Taking the $p$-th root on both sides, we obtain
\begin{align*}
\left\|W^{\frac{1}{p}}T
\left(W^{-\frac{1}{p}}\vec{f}\right)\right\|_{L^p}
\lesssim[W]_{\mathscr{A}^{\operatorname{loc}}_{p}(r)}^{1+\frac{1}{p-1}
-\frac{1}{p}}\left\|\vec{f}\right\|_{L^p}.
\end{align*}
This completes the proof of Theorem \ref{thm-CZloc}.
\end{proof}

As an application of Theorem \ref{thm-CZloc},
we show the boundedness of the Riesz transform associated with
Schr\"{o}dinger operators
$-\Delta+m^2I$ on local matrix-weighted Lebesgue spaces $L^p(W)$,
where $m\in(0,\infty)$, $\Delta$ is the Laplacian,
and $I$ is the identity operator.

\begin{example}\label{exam-Riesz-bessel}
Let $j \in \{1, \dots, n\}$ and $m \in (0,\infty)$.
The \emph{Riesz transform $\mathcal{R}_{j,m}$
associated with the Schr\"{o}dinger operator $-\Delta+m^2I$}
is defined by setting, 	for any $f \in L^2$
and $\xi:=(\xi_1,\dots,\xi_n)\in\mathbb{R}^n$,
$$\widehat{\mathcal{R}_{j,m} f}(\xi)
:=\frac{i\xi_j}{(|\xi|^2+m^2)^{\frac{1}{2}}}
\widehat{f}(\xi).$$
Next, we verify that $\mathcal{R}_{j,m}$ falls into our class of
Calder\'{o}n--Zygmund operators with exponential decay. Obviously,
it follows from Plancherel's theorem and
$$\left|\frac{i\xi_j}{(|\xi|^2+m^2)^{\frac{1}{2}}}\right|
\leq \frac{|\xi|}{|\xi|} = 1$$
that
$\mathcal{R}_{j,m}$ is bounded on $L^2$.
Let
$$G_{m}:=\left[\left(|\cdot|^2+m^2\right)^{-\frac{1}{2}}\right]^\vee.$$
By the proof of \cite[Proposition 1.2.5]{gra14b},
we find that, for any
$x \in \mathbb{R}^n \setminus \{\mathbf{0}\}$,
$$G_{m}(x) = C_n \int_0^\infty
e^{-\frac{|x|^2}{4s}} e^{-m^2 s} s^{\frac{1-n}{2}}\,\frac{ds}{s},$$
where $C_n$ is a positive constant depending on $n$.
Consequently, the integral kernel $K_{j,m}$ of $\mathcal{R}_{j,m}$
can be computed by taking the partial derivative $\partial_{x_j}$
inside the integral
and, for any $x,y\in\mathbb{R}^n$ with $x\neq y$,
\begin{align}\label{eq-K_j,m}
K_{j,m}(x,y) = -\frac{C_n}{2} (x-y)_j \int_0^\infty
e^{-(\frac{|x-y|^2}{4s} + m^2 s)} s^{-\frac{n+1}{2}}\,\frac{ds}{s}.
\end{align}
To verify the size condition for $K_{j,m}$,
applying the basic inequality that
$a^2+b^2\geq 2|a||b|$ for any $a,b\in\mathbb{R}$,
we obtain $\frac{|x-y|^2}{4s} + m^2 s \ge m|x-y|$ and hence
\begin{align}\label{eq-e}
e^{-(\frac{|x-y|^2}{4s} + m^2 s)}
= e^{-\frac{1}{2}(\frac{|x-y|^2}{4s} + m^2 s)}
e^{-\frac{1}{2}(\frac{|x-y|^2}{4s} + m^2 s)}
\le e^{-\frac{m}{2}|x-y|} e^{-\frac{|x-y|^2}{8s}}.
\end{align}
Inserting this estimate into \eqref{eq-K_j,m} yields
$$\left|K_{j,m}(x,y)\right| \lesssim |x-y| e^{-\frac{m}{2}|x-y|}
\int_0^\infty e^{-\frac{|x-y|^2}{8s}} s^{-\frac{n+1}{2}}\,\frac{ds}{s},$$
which, together with the change of variables
$u = \frac{|x-y|^2}{8s}$, further implies that
\begin{align}\label{eq-int-change}
\int_0^\infty e^{-\frac{|x-y|^2}{8s}} s^{-\frac{n+1}{2}}\,\frac{ds}{s}
&=\int_0^\infty e^{-u} \left(\frac{8u}{|x-y|^2}
\right)^{\frac{n+1}{2}}\,\frac{du}{u}\\
&\sim\int_{0}^{\infty}e^{-u}u^{\frac{n-1}{2}}\,du
\frac{1}{|x-y|^{n+1}}\sim\frac{1}{|x-y|^{n+1}}\nonumber
\end{align}
and
$$\left|K_{j,m}(x,y)\right|\lesssim
\frac{e^{-\frac{m}{2}|x-y|}}{|x-y|^n}.$$

To verify the smoothness condition \eqref{eq-smooth-condition}
for $K_{j,m}$, using \eqref{eq-K_j,m}, \eqref{eq-e}, and
the same change of variables $u=\frac{|x-y|^2}{8s}$ as in \eqref{eq-int-change},
we conclude that, for any $x,y\in\mathbb{R}^n$ with $x\neq y$,
\begin{align*}
|\nabla_x K_{j,m}(x,y)|&\lesssim \int_0^\infty
e^{-(\frac{|x-y|^2}{4s} + m^2 s)}
s^{-\frac{n+1}{2}} \frac{ds}{s}+|x-y|^2 \int_0^\infty
e^{-(\frac{|x-y|^2}{4s} + m^2 s)} s^{-\frac{n+3}{2}} \frac{ds}{s}\\
&\leq e^{-\frac{m}{2}|x-y|}\Big(\int_0^\infty
e^{-\frac{|x-y|^2}{8s}}
s^{-\frac{n+1}{2}} \frac{ds}{s}+|x-y|^2 \int_0^\infty
e^{-\frac{|x-y|^2}{8s}}s^{-\frac{n+3}{2}} \frac{ds}{s}\Big) \\
&\sim\frac{e^{-\frac{m}{2}|x-y|}}{|x-y|^{n+1}}.
\end{align*}
By symmetry, the same estimate for $\nabla_y K_{j,m}$ also holds.
Then it follows from the mean-value theorem that
$K_{j,m}$ satisfies the smoothness condition \eqref{eq-smooth-condition}.
Thus, $\mathcal{R}_{j,m}$ is a $(1, \frac{m}{2})$-Calder\'{o}n--Zygmund operator.
Moreover, let the notation be the same as in Theorem \ref{thm-CZloc}.
If $$m>\frac{6}{p r}\ln\left(2^{np}
[W]_{\mathscr{A}^{\operatorname{loc}}_{p}(r)}\right),$$
then $\mathcal{R}_{j,m}$ is bounded on $L^p(W)$.
\end{example}

\section{One-Dimensional Scalar Case:
Sharp Weighted Quantitative \\
Bounds of Local Maximal Operators}\label{s4}

In this section, we work in the special setting where $n=1$
and $m=1$ (the one-dimensional scalar case) and
show \eqref{eq-bound-maxloc-1} for all $p=q\in(1,\infty)$.
Thus, throughout this section, except Lemma \ref{lem:rhi}
and its Corollary \ref{cor-Aq}, the underlying space is $\mathbb{R}$.

For any $r\in(0,\infty)$, let the \emph{notation} $M_r:=M_{0,r}$ be
the \emph{local maximal operator}, where $M_{0,r}$ is
as in \eqref{eq-def-locfracmax} with $\alpha:=0$.
We next present the main result of this subsection,
which is established on $\mathbb{R}$.
\begin{theorem}\label{thm:strong}
Let $r\in(0,\infty)$, $p\in(1,\infty)$,
and $w \in A_p^{\operatorname{loc}}(r)$.
There exists a positive constant $C$, depending on
$p$, such that, for any $f\in L^p(w)$,
\begin{align}\label{eq-bound-locmax-n=1}
\left\|M_{r} f\right\|_{L^p(w)} \le C
[w]_{A_p^{\operatorname{loc}}(r)}^{\frac{1}{p-1}} \|f\|_{L^p(w)}.
\end{align}
\end{theorem}

Its sharpness is presented in the following remark.

\begin{remark}
It immediately follows from Proposition \ref{prop-sharpness-fractmax}
that the exponent $\frac{1}{p-1}$ of the weight constant
in \eqref{eq-bound-locmax-n=1} is sharp. Furthermore,
since, in this setting, Theorem \ref{thm:strong} holds for
all $p\in(1,\infty)$, it is also natural to expect
that \eqref{eq-bound-maxloc-1} holds for the full ranges
of indices considered in Theorem \ref{thm-bound-maxloc},
which, even for $M_r$, is still unknown when $p\in (2,\infty)$
and $n>1$ or $m>1$.
\end{remark}

The proof of Theorem \ref{thm:strong} relies on the
Marcinkiewicz interpolation and the quantitative
reverse H\"{o}lder inequality for local weight classes.
To this end, we first establish the
weak $(p,p)$ bound for $M_{r}$.
In what follows, for any scalar weight $w$ on $\mathbb{R}$
and for any measurable set $E\subset\mathbb{R}$, let
$$w(E):=\int_{E}w(x)\,dx.$$

\begin{theorem}\label{thm:weak}
Let $r\in(0,\infty)$, $p\in(1,\infty)$,
and $w \in A_p^{\operatorname{loc}}(r)$.
Then, for any $f \in L^p(w)$ and $\lambda\in(0,\infty)$,
$$w\left(\left\{x \in \mathbb{R} : M_{r}
f(x) > \lambda\right\}\right) \le
2\frac{[w]_{A_p^{\operatorname{loc}}(r)}}{\lambda^p}
\int_{\mathbb{R}} |f(y)|^p w(y) \,dy.$$
\end{theorem}

\begin{proof}
For any $\lambda\in(0,\infty)$, let
$$E_\lambda:=\left\{x \in \mathbb{R} :
M_{r} f(x) > \lambda\right\}.$$
Since $E_\lambda$ is an open set in $\mathbb{R}$,
$E_\lambda$ can be decomposed into an
at most countable union of mutually
disjoint open intervals as
$$E_\lambda = \bigcup_{j\in\mathbb{N}}(u_j, v_j),$$
where $u_j, v_j\in\mathbb{R}\cup\{-\infty, \infty\}$.
For any fixed component $(u_j, v_j)$, let $K := [A, B]$ be a
bounded and closed interval such that $K \subset (u_j, v_j)$.
For any $x \in K$, by the definition of $M_{r}$,
there exists an open interval $I_x$ with $\ell(I_x)\leq r$
such that $x\in I_x$ and
\begin{align}\label{eq-fint}
\frac{1}{|I_x|} \int_{I_x} |f(y)|\,dy > \lambda.
\end{align}
Note that, for any $y\in I_x$, $M_{r}f(y)
>\lambda$. Thus, $I_x\subset E_\lambda$. Since $I_x$ is a connected
interval containing $x \in (u_j, v_j)$ and $\{(u_j,v_j)\}_{j\in\mathbb{N}}$
are mutually disjoint,
\begin{align}\label{eq-interval}
I_x\subset (u_j, v_j)
\end{align}
holds. The collection $\{I_x\}_{x \in K}$ forms an
open cover of the compact set $K$. Applying the Heine--Borel theorem,
we obtain a finite subcollection $\mathcal{F}$ of
the collection $\{I_x\}_{x \in K}$ such that
$\mathcal{F}$ is also an open cover of $K$.
Next, we claim that, via refining the subset $\mathcal{F}$,
we can find a finite subcollection $\mathcal{G} \subset \mathcal{F}$
such that:
\begin{enumerate}
\item[{\rm (i)}] $K \subset \bigcup_{I\in \mathcal{G}} I$.
\item[{\rm (ii)}] For any $y\in\mathbb{R}$, there are at most
two intervals in $\mathcal{G}$ containing $y$,
i.e., $\sum_{I\in \mathcal{G}} \mathbf{1}_{I}(y)\le 2$.
\end{enumerate}

We now use the following greedy algorithm to construct $\mathcal{G}$.
First, select $I_1 = (a_1, b_1) \in \mathcal{F}$ that contains $A$
and has the maximal right endpoint $b_1$. If $b_1 > B$,
then $\mathcal{G}:=\{I_1\}$ is sufficient.
If $b_1 \le B$, select $I_2 = (a_2, b_2) \in \mathcal{F}$
that contains $b_1$ and has the maximal right endpoint $b_2$.
If $b_2 > B$, then $\mathcal{G}:=\{I_1, I_2\}$ is sufficient.
If $b_2 \le B$ also holds, select $I_3 = (a_3, b_3) \in \mathcal{F}$
that contains $b_2$ and has the maximal right endpoint $b_3$.
Proceeding with this algorithm, since $\mathcal{F}$ is
a finite collection and $b_k < b_{k+1}$,
the process must terminate in a finite number of steps $N$, yielding
$\mathcal{G} = \{I_1, I_2, \dots, I_N\}$. Moreover, $b_N > B$
ensures (i) holds;
otherwise $b_N\leq B$ contradicts the termination of this process.

To prove the bounded overlap property (ii), it suffices to show that
$I_k \cap I_{k+2} = \emptyset$ for all $k \in \{1, \dots, N-2\}$.
Suppose for contradiction that $I_k \cap I_{k+2} \neq \emptyset$.
Then $b_k \in I_{k+2}$. By our algorithm, we necessarily
have $b_{k+2} > b_{k+1}$. However, this contradicts
the choice of $I_{k+1}$. In our algorithm, $I_{k+1}$ was selected
specifically to maximize the right endpoint among all intervals
in $\mathcal{F}$ covering $b_k$. Since
$b_{k+2} > b_{k+1}$, $I_{k+2} \in \mathcal{F}$ covers $b_k$ and
provides a strictly larger right endpoint,
$I_{k+2}$ would have been chosen
instead of $I_{k+1}$, which is a contradiction.
Thus, the property (ii) holds.

For any $I \in \mathcal{G}$, applying \eqref{eq-fint},
H\"{o}lder's inequality, and the $A_p^{\operatorname{loc}}(r)$ condition,
we obtain
\begin{align*}
\lambda^p &< \left[\fint_{I} |f(y)|\,dy\right]^p
\leq \fint_{I} |f(y)|^p w(y)\,dy \left[\fint_{I}
w^{-\frac{p^{\prime}}{p}}(y)\,dy \right]^{\frac{p}{p^{\prime}}}\\
&\leq \frac{1}{w(I)}[w]_{A_p^{\operatorname{loc}}(r)}
\int_{I} |f(y)|^p w(y)\,dy \, .
\end{align*}
Summing over all $I \in \mathcal{G}$ and using
the aforementioned property (ii), we conclude that
$$w(K) \le \sum_{I \in \mathcal{G}} w(I) \le
\frac{[w]_{A_p^{\operatorname{loc}}(r)}}{\lambda^p}
\sum_{I \in \mathcal{G}}\int_{I} |f(y)|^p w(y)\,dy
\le 2\frac{[w]_{A_p^{\operatorname{loc}}(r)}}{\lambda^p}
\int_{(u_j,v_j)} |f(y)|^p w(y)\,dy,$$
where, in the last step, we used \eqref{eq-interval} to obtain the fact
$\bigcup_{I\in\mathcal{G}}I \subset (u_j,v_j)$.
Since this inequality holds for any bounded and closed
interval $K:=[A,B]\subset (u_j, v_j)$, we infer that
$$w((u_j, v_j)) = \lim_{A\to u_j, B\to v_j}w(K)\leq
2\frac{[w]_{A_p^{\operatorname{loc}}(r)}}{\lambda^p}
\int_{(u_j, v_j)} |f(y)|^p w(y)\,dy$$
and hence
\begin{align*}
w(E_\lambda) &= \sum_{j\in\mathbb{N}} w((u_j, v_j))
\leq2\sum_{j\in\mathbb{N}}\frac{ [w]_{A_p^{\operatorname{loc}}(r)}}{\lambda^p}
\int_{(u_j, v_j)} |f(y)|^p w(y) \, dy \\
&=2\frac{[w]_{A_p^{\operatorname{loc}}(r)}}{\lambda^p}
\int_{E_\lambda} |f(y)|^p w(y) \, dy
\leq 2\frac{[w]_{A_p^{\operatorname{loc}}(r)}}{\lambda^p}
\int_{\mathbb{R}} |f(y)|^p w(y) \, dy.
\end{align*}
This completes the proof of Theorem \ref{thm:weak}.
\end{proof}

Next, we recall the following quantitative reverse H\"{o}lder inequality for
the class $A_{p}$ on $\mathbb{R}^n$ from \cite[Theorem 2.3]{hp13}.
\begin{lemma}\label{lem:rhi}
Let $p\in(1,\infty)$ and $w \in A_p$. Then there exists a positive
constant $C_n$, depending on $n$, such that, for any cube
$Q\subset\mathbb{R}^n$,
\begin{align*}
\left[\fint_Q w^{1+\delta}(y)\,dy\right]^{\frac{1}{1+\delta}}
\leq 2\fint_Q w(y)\,dy,
\end{align*}
where $\delta:=\frac{C_n}{[w]_{A_p}}$.
\end{lemma}

As a corollary of Lemma \ref{lem:rhi}
and Theorem \ref{thm-extension}, we can obtain the following
quantitative open property of the local weight class
$A^{\operatorname{loc}}_p(r)$ on $\mathbb{R}^n$.

\begin{corollary}\label{cor-Aq}
Let $r\in(0,\infty)$, $p\in(1,\infty)$, and $w \in A^{\operatorname{loc}}_p(r)$.
Then there exists an exponent $q \in (1, p)$ such
that $w \in A^{\operatorname{loc}}_q(r)$
has the following two properties:
\begin{itemize}
\item[{\rm (i)}] $[w]_{A_q^{\operatorname{loc}}(r)}
\le 2^{p-1} [w]_{A_p^{\operatorname{loc}}(r)}$.
\item[{\rm (ii)}] $p-q\sim\frac{1}{[w]^{\frac{1}{p-1}}_{A_{p}^{\operatorname{loc}}(r)}}$,
where the equivalence constants depend on $p$ and $n$.
\end{itemize}
\end{corollary}

\begin{proof}
For any cube $Q\subset\mathbb{R}^n$ with $\ell(Q)\leq r$,
by Theorem \ref{thm-extension}, there exists $v_Q\in A_p$ such that
$v_Q=w$ on $Q$ and
\begin{align*}
[v_Q]_{A_p}\leq 3^{np}[w]_{A^{\operatorname{loc}}_p(r)}.
\end{align*}
Let $\delta:=\frac{C_n}{[v_Q]_{A_p}}$ and
$\widetilde{\delta}:=\frac{C_n}{3^{np}[w]_{A^{\operatorname{loc}}_p(r)}}$,
where $C_n$ is the same constant as in Lemma \ref{lem:rhi}.
Note that $\widetilde{\delta} \le \delta$.
Using this extension property, H\"{o}lder's inequality,
and Lemma \ref{lem:rhi},
we conclude that
\begin{align*}
\left[\fint_Q w^{1+\widetilde{\delta}}(y)\,dy
\right]^{\frac{1}{1+\widetilde{\delta}}}
&=\left[\fint_Q v_Q^{1+\widetilde{\delta}}(y)\,dy
\right]^{\frac{1}{1+\widetilde{\delta}}}
\leq\left[\fint_Q v_Q^{1+\delta}(y)\,dy
\right]^{\frac{1}{1+\delta}}\\
&\leq 2\fint_Q v_Q(y)\,dy
= 2\fint_Q w(y)\,dy.
\end{align*}
Applying a standard duality argument, we obtain
$\widetilde{w}:= w^{-\frac{1}{p-1}} \in
A_{p^{\prime}}^{\operatorname{loc}}(r)$ with
$[\widetilde{w}]_{A_{p^{\prime}}^{\operatorname{loc}}(r)}=
[w]^{\frac{1}{p-1}}_{A_{p}^{\operatorname{loc}}(r)}$.

Let $\widetilde{\delta}_1:=\frac{C_n}{3^{np^{\prime}}
[\widetilde{w}]_{A_{p^{\prime}}^{\operatorname{loc}}(r)}}$.
Next, we prove that $w\in A_{q}^{\operatorname{loc}}(r)$,
where $q\in(1,p)$ satisfies $\frac{p-1}{q-1}=1+\widetilde{\delta}_1$.
Indeed, for any cube $Q\subset\mathbb{R}^n$ with $\ell(Q)\leq r$,
we have
\begin{align*}
&\fint_Q w(x)\,dx\left[\fint_Q w^{-\frac{1}{q-1}}(x)
\,dx\right]^{q-1}\\
&\quad=\fint_Q w(x)\,dx\left[\fint_Q w^{-\frac{1}{p-1}(1+\widetilde{\delta}_1)}(x)
\,dx\right]^{\frac{p-1}{1+\widetilde{\delta}_1}}\\
&\quad=\fint_Q w(x)\,dx\left( \left[\fint_Q \widetilde{w}^{1+\widetilde{\delta}_1}(x)
\,dx\right]^{\frac{1}{1+\widetilde{\delta}_1}} \right)^{p-1}
\leq 2^{p-1}\fint_Q w(x)\,dx\left[\fint_Q \widetilde{w}(x)
\,dx\right]^{p-1}\\
&\quad= 2^{p-1}\fint_Q w(x)\,dx\left[\fint_Q w(x)^{-\frac{1}{p-1}}
\,dx\right]^{p-1}
\leq 2^{p-1}[w]_{A^{\operatorname{loc}}_p(r)}.
\end{align*}
Taking the supremum over all such cubes $Q$,
we obtain $w\in A^{\operatorname{loc}}_q(r)$ and
$[w]_{A_q^{\operatorname{loc}}(r)} \le 2^{p-1}
[w]_{A_p^{\operatorname{loc}}(r)}$.
Note that
\begin{align*}
p-q&=(p-1)-(q-1)=(p-1)-(p-1)\left(\frac{1}{1+\widetilde{\delta}_1}\right)\\
&=(p-1)\left(\frac{\widetilde{\delta}_1}{1+\widetilde{\delta}_1}\right)
\sim\frac{1}{[\widetilde{w}]_{A_{p^{\prime}}^{\operatorname{loc}}(r)}}=
\frac{1}{[w]^{\frac{1}{p-1}}_{A_{p}^{\operatorname{loc}}(r)}},
\end{align*}
where the equivalence constants depend on $p$ and $n$.
This completes the proof of Corollary \ref{cor-Aq}.
\end{proof}

We now begin to prove Theorem \ref{thm:strong}.

\begin{proof}[Proof of Theorem \ref{thm:strong}]
It is a standard fact that
$\left\|M_{r} f\right\|_{L^\infty(w)} \leq\|f\|_{L^\infty(w)}$
for any $f\in L^\infty(w)$.
By Corollary \ref{cor-Aq} and Theorem \ref{thm:weak},
there exists an exponent $q\in(1,p)$ as in Corollary \ref{cor-Aq}
such that,
for any $f \in L^q(w)$ and $\lambda\in(0,\infty)$,
\begin{align*}
w\left(\left\{x \in \mathbb{R} : M_{r}
f(x) > \lambda\right\}\right) \le 2
\frac{[w]_{A_q^{\operatorname{loc}}(r)}}{\lambda^q}
\int_{\mathbb{R}} |f(y)|^q w(y) \,dy,
\end{align*}
which further implies that
$$\left\|M_{r} f\right\|_{L^{q,\infty}(w)}
\leq \left(2[w]_{A_q^{\operatorname{loc}}(r)}\right)^{\frac{1}{q}}
\left\|f\right\|_{L^q(w)}.$$
Applying the Marcinkiewicz interpolation theorem
(see, for instance, \cite[Theorem 1.3.2]{g14c}) and the quantitative
properties from Corollary \ref{cor-Aq}, we obtain,
for any $f\in L^p(w)$,
\begin{align*}
\left\|M_{r} f\right\|_{L^p(w)}
&\leq 2\left(\frac{p}{p-q}\right)^{\frac{1}{p}}
\left(2[w]_{A_q^{\operatorname{loc}}(r)}\right)^{\frac{1}{p}} \|f\|_{L^p(w)}\\
&\leq 4\left(\frac{p}{p-q}\right)^{\frac{1}{p}}
[w]^{\frac{1}{p}}_{A_p^{\operatorname{loc}}(r)} \|f\|_{L^p(w)}\\
&\lesssim [w]^{\frac{1}{p(p-1)}+\frac{1}{p}}_{A_p^{\operatorname{loc}}(r)}
\|f\|_{L^p(w)}=[w]^{\frac{1}{p-1}}_{A_p^{\operatorname{loc}}(r)}\|f\|_{L^p(w)}.
\end{align*}
This completes the proof of Theorem \ref{thm:strong}.
\end{proof}

\textbf{Acknowledgements.} \quad The authors would
like to thank both referees for their careful reading and many valuable
comments which improved the presentation of this article.

This project is partially supported by the National Natural
Science Foundation of China (Grant
Nos. 12371093 and 12431006), the Beijing Natural
Science Foundation (Grant No. 1262011),
the Fundamental Research Funds for the Central Universities
(Grant No. 2253200028),
and the Research Council of Finland
(Grant Nos. 364208 and 371637).

\bigskip

\noindent Tuomas Hyt\"onen

\medskip

\noindent Department of Mathematics and Systems Analysis, Aalto University,
P.O. Box 11100 (Otakaari~1), FI-00076 Aalto, Finland

\smallskip

\noindent{\it E-mail:} \texttt{tuomas.hytonen@aalto.fi}

\bigskip

\noindent Dachun Yang (Corresponding author),
Wen Yuan and Mingdong Zhang.

\medskip

\noindent Laboratory of Mathematics and Complex Systems
(Ministry of Education of China),
School of Mathematical Sciences,
Institute for Advanced Study, Beijing Normal University,
Beijing 100875, The People's Republic of China

\smallskip

\noindent{{\it E-mails:}}
\texttt{dcyang@bnu.edu.cn} (D. Yang)

\noindent\phantom{{\it E-mails:}}
\texttt{wenyuan@bnu.edu.cn} (W. Yuan)

\noindent\phantom{\it E-mails:}
\texttt{mdzhang@mail.bnu.edu.cn} (M. Zhang)
\end{document}